\documentclass[11pt,a4paper]{amsart}[2000]
\usepackage{amsmath}
\usepackage{amsfonts}
\usepackage{amssymb}
\usepackage{amsthm}
\usepackage[hang]{footmisc}
\usepackage[all]{xypic}
\usepackage{hyperref}
\usepackage{color}

\title[Nuclear dimension for topological flows]{The nuclear dimension of $C^*$-algebras associated to topological flows and orientable line foliations}

\author{Ilan Hirshberg}

\address{Department of Mathematics, Ben Gurion University of the Negev, \phantom{----------------}\linebreak\text{}\hspace{3.5mm}
P.O.B. 653, Be'er Sheva 84105, Israel}
\email{ilan@math.bgu.ac.il}

\author{Jianchao Wu}

\author{Jianchao Wu}

\address{Department of Mathematics, Texas A\&M University, \phantom{-------------------------------------}\linebreak \text{}\hspace{3.5mm} Mailstop 3368, College Station, TX 77843, USA}

\email{jwu@tamu.edu}

\thanks{This research was supported by GIF grant 1137/2011, Israel Science Foundation grant no.~ 476/16, NSF DMS-1564401, ERC Advanced Grant ToDyRiC 267079, and SFB 878 \emph{Groups, Geometry and Actions}.}

\allowdisplaybreaks

\theoremstyle{plain}

\newtheorem{Thm}{Theorem}[section]
\newtheorem{Cor}[Thm]{Corollary}
\newtheorem{Lemma}[Thm]{Lemma}

\theoremstyle{definition}

\newtheorem{Notation}[Thm]{Notation}

\theoremstyle{plain}
\newtheorem{thm}[Thm]{Theorem}
\newtheorem{lem}[Thm]{Lemma}
\newtheorem{cor}[Thm]{Corollary}
\newtheorem{prop}[Thm]{Proposition}

\theoremstyle{definition}
\newtheorem{defn}[Thm]{Definition}

\newtheorem{eg}[Thm]{Example}
\newtheorem{rmk}[Thm]{Remark}
\newtheorem{notn}[Thm]{Notation}
\newtheorem{cnstrct}[Thm]{Construction}

\newcommand{\R}{{\mathbb R}}

\newcommand{\Z}{{\mathbb Z}}

\newcommand{\aut}{\mathrm{Aut}}
\newcommand{\supp}{\mathrm{supp}}

\newcommand{\eps}{\varepsilon}
\numberwithin{equation}{section}

\newcommand{\dr}{\mathrm{dr}}
\newcommand{\id}{\mathrm{id}}

\newcommand{\dimnuc}{\operatorname{dim}_{\mathrm{nuc}}}
\newcommand{\dimnucone}[0]{\dimnuc^{\!+1}}

\newcommand{\calpha}{{\widehat{\alpha}}}

\newcommand{\hol}{\operatorname{hol}}

\newcommand{\middlebar}{\;\middle\vert\;}

\theoremstyle{definition}

\newtheorem{nota}[Thm]{Notation}

\numberwithin{equation}{Thm}

\newtheorem{clminproof}{Claim}[section]

\begin{document}
\begin{abstract}
We show that for any locally compact Hausdorff space $Y$ with finite 
covering dimension and for any continuous flow $\R \curvearrowright Y$, the 
resulting crossed product $C^*$-algebra $C_0(Y) \rtimes \R$ has finite nuclear 
dimension. This generalizes previous results for free flows, where this was 
proved using Rokhlin dimension techniques. As an application, we obtain bounds 
for the nuclear dimension of $C^*$-algebras associated to one-dimensional 
orientable foliations. This result is analogous to the one we obtained earlier 
for non-free actions of $\Z$. Some novel techniques in our proof include the 
use of a conditional expectation constructed from the inclusion of a clopen 
subgroupoid, as well as the introduction of what we call fiberwise groupoid 
coverings that help us build a link between foliation $C^*$-algebras and 
crossed products. 
\end{abstract}
\maketitle
\tableofcontents

\section{Introduction}

Nuclear dimension for $C^*$-algebras was introduced by Winter and Zacharias in 
\cite{winter-zacharias}, as a noncommutative generalization of covering 
dimension. Since then, it has come to play a crucial role in structure and 
classification of $C^*$-algebras: indeed, it is now known that simple unital 
separable $C^*$-algebras with finite nuclear dimension and which satisfy the 
UCT are classified via the Elliott invariant (\cite{TWW,EGLN}). It was shown in 
\cite{winter-zacharias} that if $Y$ is a locally compact second countable 
Hausdorff 
space, then $\dimnuc(C_0(Y))$ coincides with the covering dimension of $Y$, and 
the property of having finite nuclear dimension is preserved under various 
constructions: forming direct sums and tensor products, passing to quotients 
and hereditary subalgebras, and forming extensions. There has been considerable 
interest in seeing to what extent finite nuclear dimension is preserved under 
forming crossed products, which in particular led to the development of the 
notion of Rokhlin dimension for various group actions; see 
\cite{HWZ,hirshberg-phillips,szabo,SWZ, gardella} for actions of finite groups, 
$\Z$, $\Z^m$, and compact group actions.

The case of flows, that is, actions of $\R$, was addressed in a joint paper 
\cite{HSWW16} by the authors and Szab{\'o} and Winter. We developed a theory of 
Rokhlin dimension for flows which to a great extent parallels the theory for 
actions of $\Z$ (though the technicalities worked out to be rather different). 
This generalized Kishimoto's Rokhlin property for flows (\cite{kishimoto96}). 
In particular, we showed that if $\alpha$ is a flow on a $C^*$-algebra $A$ with 
finite Rokhlin dimension, and $A$ has finite nuclear dimension, then the 
crossed product $A \rtimes_{\alpha} \R$ has finite nuclear dimension. We 
furthermore showed that  if $Y$ is a locally compact second countable Hausdorff 
space 
with finite covering dimension and $\calpha$ is a \emph{free} flow on $Y$ then 
the induced flow $\alpha$ on $C_0(X)$ has finite Rokhlin dimension. Thus, 
$C_0(Y) \rtimes_{\alpha} \R$ has finite nuclear dimension.

This leaves the case of non-free flows. Those cannot have finite Rokhlin dimension. 
We addressed the parallel case of non-free $\Z$-actions on topological spaces 
in \cite{Hirshberg-Wu16}, where we showed that if $Y$ is a locally compact 
second countable Hausdorff space with finite covering dimension and $\alpha$ is 
\emph{any} automorphism  of $C_0(Y)$ then $C_0(Y) \rtimes_{\alpha} \Z$ has 
finite nuclear dimension. As in the case of actions of $\Z$, non-free actions 
of $\R$ are quite prevalent. Indeed, if $Y$ is a compact smooth manifold then 
any vector field on $Y$ gives rise to a flow, but typically such flows may have 
fixed points or periodic orbits. 

The purpose of this paper is to provide a parallel to \cite{Hirshberg-Wu16} for possibly non-free flows. We show in Theorem \ref{thm-main} that 
if $Y$ is a locally compact, second countable Hausdorff with finite covering 
dimension, 
and $\alpha$ is a flow on $C_0(Y)$, then
\[
\operatorname{dim}_\mathrm{nuc}({C}_{0}(Y) \rtimes {\mathbb{R}^{}} ) \leq 5 \big( \dim(Y) \big) ^2 + 12 \dim(Y) + 6 \; .
\]

In fact, this estimate for flows implies a similar estimate for $\Z$-actions, by applying the mapping torus construction to obtain a flow from a $\Z$-action (see Corollary~\ref{cor:mapping-torus}). Thus we also recover the main theorem of \cite{Hirshberg-Wu16}, albeit with a less sharp bound. We point out that despite the similarity between these results, there is some significant difference in the technical tools used in them (see the comments after Corollary~\ref{cor:mapping-torus}).

As an application, we provide a bound for the nuclear dimension of 
$C^*$-algebras associated to orientable line foliations. It was shown by 
Whitney (\cite{Whitney}) that any orientable line foliation on a locally 
compact second countable Hausdorff space arises from a flow. A crucial object 
in the 
study of index theorems for foliations (\cite{Connes-Skandalis}) is the 
construction of the holonomy groupoid $G_{\mathcal{F}}$ of a foliation 
$\mathcal{F}$ and the $C^*$-algebra $C^*(G_{\mathcal{F}})$ thereof (see also 
\cite{moore-schochet} for a discussion of $C^*$-algebras associated to 
foliations). However, for an orientable line foliation $\mathcal{F}$, this 
foliation $C^*$-algebra $C^*(G_{\mathcal{F}})$ typically differs from the 
crossed product by the flow that gives rise to the foliation. 

To make a link between these two important $C^*$-algebraic constructions, we develop a general theory for a kind of groupoid homomorphisms between topological groupoids that we call \emph{fiberwise groupoid covering maps} (see Definition~\ref{def:fiberwise-groupoid-covering}). They induce quotient maps between the corresponding maximal groupoid $C^*$-algebras (see Theorem~\ref{thm:fiberwise-groupoid-covering-locally-compact} and~\ref{thm:fiberwise-groupoid-covering-quotient}). A canonical quotient map from the transformation groupoid of a flow to the holonomy groupoid $G_{\mathcal{F}}$ of the induced foliation is shown to be a fiberwise groupoid covering map (see Proposition~\ref{prop:foliation-covering}).
As a consequence, $C^*(G_{\mathcal{F}})$ can be obtained as a quotient of the crossed product associated to the flow (see Corollary~\ref{cor:foliation-covering}). Thus we have
\[
\operatorname{dim}_\mathrm{nuc}(C^*(G_{\mathcal{F}}) ) \leq 5 \big( \dim(Y) \big) ^2 + 12 \dim(Y) + 6 \; .
\]
where $Y$ is the underlying space of the foliation (see Corollary~\ref{cor:foliation-algebra-dimnuc}).

We now briefly sketch the idea of our proof. The basic idea is similar to the case of integer actions, though the techniques are different. The case in which the flow has uniformly compact orbits, that is, when all points are fixed or periodic with period bounded by some constant $R$, was already done in \cite[Section 3]{Hirshberg-Wu16} (indeed, we included it in that generality for use in the present paper), except for an estimate of the covering dimension of the quotient space $Y / \R$, which we complete in Section~\ref{section: bounded periods}. We showed there that one can find in such a case a bound on the nuclear dimension of $C_0(Y) \rtimes_{\alpha} \R$, which, crucially, does not depend on the maximal period length $R$.

Now, if we fix a some $R>0$, we can consider the set of points which are periodic with orbit length $\leq R$, which we denote by $Y_{\leq R}$, and we let $Y_{>R} = Y \smallsetminus Y_{\leq R}$. Then one shows that $Y_{\leq R}$ is a closed subset, and we have an equivariant extension 
$$
0 \to C_0(Y_{> R}) \to C_0(Y) \to C_0(Y_{\leq R}) \to 0
.
$$
Although the restriction of $\alpha$ to $C_0(Y_{>R})$ does not necessarily have finite Rokhlin dimension, we can use a quantitative version of the techniques we used in \cite{HSWW16}, based on \cite{BarLRei081465306017991882} and \cite{Kasprowski-Rueping}, to construct dimensionally controlled long and thin covers of $C_0(Y_{>R})$, provided ``long'' means ``not too long compared to $R$''. This technique, which was used in \cite{HSWW16} to construct Rokhlin elements, is then used to construct decomposable approximations for $C_0(Y_{>R}) \rtimes_{\alpha} \R$, for certain finite subsets and a level of precision which improves as $R$ increases. (One could also use those covers to construct Rokhlin-type elements and then use them to construct decomposable approximations; however in the abelian setting, constructing suitable flow-wise Lipschitz partitions of unity, which in the setting of \cite{HSWW16} is a step towards constructing Rokhlin elements, is enough to obtain the required decomposable approximations, with an improved bound, so we do not need those Rokhlin elements for our result.) Having constructed those decomposable approximations for both $C_0(Y_{>R}) \rtimes_{\alpha} \R$ and for $C_0(Y_{\leq R}) \rtimes_{\alpha} \R$, we patch them together to get decomposable approximations for $C_0(Y) \rtimes_{\alpha} \R$ in a manner similar to the one done in \cite{Hirshberg-Wu16} for the case of $\Z$-actions, where the required precision of the final approximation predetermines how large $R$ has to be at the beginning of this analysis.

We note that in previous proofs of finite nuclear dimension for uniform Roe algebras and, more generally, groupoid $C^*$-algebras, such as those in \cite{winter-zacharias} and \cite{GWY}, Arveson's extension theorem is used in the construction of the downward completely positive map in the nuclear approximation. In our case, we give instead an explicit description of this map as a sum of compositions of compressions by a partition of unity of the unit space and conditional expectations associated to the clopen inclusion of relatively compact subgroupoids into larger ones. In the course of doing this, we prove a general result that any clopen inclusion of locally compact groupoids induces a conditional expectation (see Theorem~\ref{thm:clopen-subgroupoid}). Aside from simplifying the proof somewhat, the fact that it keeps good track of the Cartan subalgebras could make it useful for future work. 

\paragraph{\textbf{Acknowledgements:}} The authors would like to thank George Elliott, Alexander Kumjian, and Claude Schochet for some helpful discussions. Part of the research underlying this paper was carried out during the authors' visits to the Mittag-Leffler Institute, University of M\"{u}nster, the Penn State University, Centre de Recerca Matem\`{a}tica in Barcelona, the Fields Institute, and the Banff International Research Station.

\section{Preliminaries}
\label{section:prelim}

Throughout the paper, we use the following conventions. To simplify formulas, we may use the notations $\dimnucone(A) = \dimnuc(A)+1$, $\dim^{+1}(Y) = \dim(Y)+1$ and $\dr^{+1}(A) = \dr(A)+1$. If $A$ is a $C^*$-algebra, we denote by $A_+$ the positive part, and by $A_{+, \leq 1}$ the set of positive elements of norm at most $1$. If $G$ is a locally compact Hausdorff group and $A$ is a $C^*$-algebra, we denote by $\alpha \colon G \curvearrowright A$ an action, that is, a continuous homomorphism $\alpha \colon G \to \aut(A)$, where $\aut(A)$ is topologized by pointwise convergence.  If $X$ is a metric space, $S \subseteq X$ and $r>0$, we denote $N_r(S) = \left\{x \in X \middlebar \operatorname{dist}(x,S)<r \right\}$.

We are interested here in the case in which $A$ is commutative, that is, $A 
\cong C_0(Y)$ for some locally compact Hausdorff space $Y$ (namely the spectrum 
$\widehat{A}$). By Gel'fand's theorem, an action $\alpha \colon G 
\curvearrowright C_0(Y)$ is completely determined by an action $\calpha \colon 
G  \curvearrowright Y$ on the spectrum (which is continuous in the sense that 
the map $G \times Y \to Y$ given by $(g,y) \mapsto \calpha_g(y)$ is continuous) 
and vice versa. They are related by the identity $\alpha_g (f) = f \circ 
\calpha_{g^{-1}}$ for any $f \in C_0(Y)$ and any $g \in G$. Thus taking a 
$C^*$-algebraic point of view, we will denote by $\calpha \colon G  
\curvearrowright Y$ an action on a locally compact Hausdorff space by 
homeomorphisms, and save the notation $\alpha$ for the corresponding action on 
$C_0(Y)$. 

When $G = \R$, the additive group of the real numbers, we often call a topological dynamical system $(Y, \R, \calpha)$ a \emph{flow}. We also make the following definitions. For any $y \in Y$, we let $\mathrm{per}_{\calpha}(y) = \inf\left\{t > 0 \middlebar \calpha_t(y) = y \right\}$ be the minimal period of the flow at $y$, or in other words, the length of the orbit of $y$ (with the convention $\mathrm{per}_{\calpha}(y) = \infty$ if $\calpha_t(y) \neq y$ for any $t \neq 0$). Thus this number is constant within each orbit. For a positive real number $R$, we decompose $Y$ as $Y_{\leq R} \sqcup Y_{> R}$, where
 \begin{align*}
  Y_{\leq R} & = \left\{ y \in Y \middlebar \mathrm{per}_{\calpha}(y) \leq R \right\} = \left\{ y \in Y \middlebar \calpha_{\R}(y) = \calpha_{[0,R]} (y) \right\} \; , \\
  Y_{> R} & =  Y \setminus Y_{\leq R} = \left\{ y \in Y \middlebar \mathrm{per}_{\calpha}(y) > R \right\} \; .
 \end{align*}
This decomposition is then invariant under the flow. Intuitively, $Y_{\leq R}$ and $Y_{> R}$ are the short-period part and the long-period part of the flow, respectively.

We also recall some basic facts about (possibly non-Hausdorff) locally compact groupoids and their $C^*$-algebras. A \emph{topological groupoid} is a small category whose morphisms are all invertible and whose set of morphisms is equipped with a topology such that the multiplication map $G^2 = \left\{(x,y) \in G \times G \middlebar  d(x) = r(y) \right\} \to G$ given by $(x,y) \to x \cdot y$ and the inversion map $G \to G$ given by $x \to x ^{-1}$ are both continuous, where $d, r \colon G \to G^0$ (the subset of $G$ consisting of all units, called the \emph{unit space}) are the domain and range maps, and $G^2$ (called the set of composable pairs) is given the subset topology inherited from the product topology on $G \times G$. We also write $G^u = r^{-1}(\left\{u\right\})$ and $G_u = d^{-1}(\left\{u\right\})$ for $u \in G^0$. 

\begin{defn}
	Let $Y$ be a topological space. We say that $Y$ is \emph{locally Hausdorff} if any point has a closed neighborhood which is compact Hausdorff. 
\end{defn}

\begin{notn}\label{nota:C_c_0}
	Let $Y$ be a topological space. We denote by $C_c(Y)_0$ the set of complex-valued functions $f$ for which there exists a Hausdorff open set $U$ so that $f$ vanishes outside of $U$, and $f|_U$ is continuous and compactly supported. Note that such functions may not be continuous when $Y$ is not Hausdorff. We then define $C_c(Y)$ to be the linear span of $C_c(Y)_0$ in the linear space of all complex-valued functions on $Y$.
\end{notn}

Observe that if $X$ is an open subset of $Y$, then we have canonical embeddings $C_c(X)_0 \hookrightarrow C_c(Y)_0$ and $C_c(X) \hookrightarrow C_c(Y)$. Also, if $Y$ is locally compact and Hausdorff, then both $C_c(Y)_0$ and $C_c(Y)$ are simply the set of compactly-supported continuous functions on $Y$. 

\begin{defn}(\cite[Definition~2.2.1 
and~2.2.2]{paterson})\label{def:locally-compact-groupoids}
	A \emph{locally compact locally Hausdorff groupoid} is a topological groupoid that satisfies the following axioms: 
	\begin{enumerate}
		\item\label{def:locally-compact-groupoids:unit-space} $G^0$ is locally compact Hausdorff in the relative topology inherited from $G$;
		\item\label{def:locally-compact-groupoids:locally-Hausdorff} there is a 
		countable family $\mathcal{C}$ of compact Hausdorff subsets of $G$ such 
		that the family $\left\{C^o \middlebar  C \in \mathcal{C} \right\}$ of 
		interiors of members of $\mathcal{C}$ is a basis for the topology of 
		$G$. (In particular, $G$ is locally Hausdorff;)
		\item\label{def:locally-compact-groupoids:sections} for any $u \in G^0$, $G^u$ is locally compact Hausdorff in the relative topology inherited from $G$;
		\item\label{def:locally-compact-groupoids:Haar-system} $G$ admits a \emph{(left) Haar system} $\left\{\lambda^u \right\}_{u \in G^0}$, in the sense that each $\lambda^u$ is a positive regular Borel measure on the locally compact Hausdorff space $G^u$, such that 
		\begin{enumerate}
			\item\label{def:locally-compact-groupoids:Haar-system:full-support} the support of each $\lambda^u$ is the whole of $G^u$, 
			\item\label{def:locally-compact-groupoids:Haar-system:continuous} for any $g \in C_c(G)$, the function $g^0$ given by the integral $\int_{G^u} g \, d \lambda^u$, belongs to $C_c(G^0)$,
			\item\label{def:locally-compact-groupoids:Haar-system:invariance} for any $x \in G$ and $f \in C_c(G)$, we have 
			\[
				\int_{G^{d(x)}} f(xz) \, d \lambda^{d(x)} (z) = \int_{G^{r(x)}} f(y) \, d \lambda^{r(x)} (y) \; .
			\]
		\end{enumerate} 
	\end{enumerate}
\end{defn}


Note that most of the technicality in this definition is due to our need to 
deal with the non-Hausdorffness of certain groupoids arising from holonomies of 
line foliations. If a locally compact locally Hausdorff groupoid is indeed 
Hausdorff, then 
the existence of a Haar system is known to be a consequence of the first three 
axioms; nevertheless, uniqueness may still fail (a counterexample is given by 
the pair groupoid construction below); hence we always fix a left Haar measure 
as part of the data of a locally compact locally Hausdorff groupoid.

\begin{eg}\label{example:groupoids}
	Basic examples of locally compact Hausdorff groupoids include: 
	\begin{enumerate}
		\item \label{item:example:groupoids:pair} the \emph{pair groupoid}\footnote{It is also known as the 
		\emph{trivial groupoid} in \cite{paterson}.} $X \times X$ for a locally 
		compact second countable Hausdorff space $X$, where $\left( X \times X 
		\right)^0$ 
		is the diagonal, $d(x,y) = (y, y)$, $r(x,y) = (x , x)$ and $(x,y) \cdot 
		(y,z) = (x, z)$, for any $x, y, z \in X$, and the positive regular 
		Borel measures on $X$ with full support are in one-to-one 
		correspondence with the left Haar measures of $X \times X$;
		\item \label{item:example:groupoids:transformation} the \emph{transformation groupoid} $X \rtimes_\calpha G$ for an 
		action $\calpha$ of a locally compact group $G$ on a locally compact 
		second countable Hausdorff space $X$, which is defined to be $X \times 
		G$ with 
		$\left( X \rtimes G \right)^0 = X \times \{1\}$, $d(x,g) = 
		\left(\calpha^{-1}_g (x), 1 \right)$, $r(x,g) = \left( x, 1 \right)$ 
		and $(x, g) \cdot (\calpha^{-1}_g (x), h) = (x, gh)$, where the left 
		Haar measure of $G$ induces a left Haar system for $X \rtimes_\calpha 
		G$; 
		\item \label{item:example:groupoids:reduction} if $G$ is a locally compact locally Hausdorff groupoid and $U$ is 
		an open 
		subset of $G^0$, then $G_U$, the \emph{reduction} of $G$ to $U$, as 
		defined by $\{ x \in G \mid d(x) \in U, \, r(x) \in U \}$, is an open 
		subset of $G$ that inherits the structure of a locally compact locally
		Hausdorff groupoid from $G$, where the left Haar system is induced by 
		the inclusion $C_c(G_U) \subset C_c(G)$. If $G$ is Hausdorff then so is 
		the reduction
	\end{enumerate}
\end{eg}

The main motivation for considering Haar systems lies in the representation 
theory for a groupoid $G$, which is often better studied through 
representations of the convolution $*$-algebra $C_c(G)$, where the convolution 
product is defined by the following formula (as in \cite[(2.20) and 
(2.21)]{paterson})
\begin{equation}\label{eq:groupoid-multiplication-r}
	f \ast g (x) = \int_{G^{r(x)}} f(y) g(y^{-1}x) \, d \lambda^{r(x)} (y)
\end{equation}
or equivalently
\begin{equation}\label{eq:groupoid-multiplication-d}
	f \ast g (x) = \int_{G^{d(x)}} f(xt) g(t^{-1}) \, d \lambda^{d(x)} (t)
\end{equation}
for any $f, g \in C_c(G)$ and any $x \in G$, and the $*$-operation is given by
\begin{equation}\label{eq:groupoid-star-operation}
	f^* (x) = \overline{f(x^{-1})}
\end{equation}
for any $f \in C_c(G)$ and any $x \in G$; (see \cite[(2.22)]{paterson}).

We now review the construction of the reduced $C^*$-algebra of a locally 
compact locally Hausdorff groupoid $G$. To this end, we write 
$\left\{\lambda_u\right\}_{u \in G^0}$ for the \emph{right Haar measure} 
corresponding to the left Haar measure $\lambda = \left\{\lambda^u\right\}_{u 
\in G^0}$, where $\lambda_v(E) = \lambda^v(E^{-1})$ for any Borel set $E 
\subset G_v$. Now any $v \in G^0$ determines a $*$-representation 
$\operatorname{Ind}v$, the \emph{left regular representation at $v$}, of 
$C_c(G)$ on the Hilbert space $L^2(G_v, \lambda_v)$, defined by the follwing 
formula 
(as in \cite[(3.41)]{paterson})
\[
	\left( \operatorname{Ind}v (f) \cdot \xi \right) (x) = \int_{t \in G_v} f(xt^{-1}) \xi(t) \, d\lambda_v(t)
\]
for any $f \in C_c(G)$, any $\xi \in L^2(G_v, \lambda_v)$ and any $x \in G_v$; thus in terms of the inner product, we have
\[
	\left\langle \operatorname{Ind}v (f) \cdot \xi, \eta \right\rangle = \int_{x \in G_v} \int_{t \in G_v} f(xt^{-1}) \xi(t) \overline{\eta(x)} \, d\lambda_v(t) \, d\lambda_v(x)
\]  
for any $\eta \in L^2(G_v, \lambda_v)$. The $C^*$-seminorm induced by 
$\operatorname{Ind}v$ is bounded by the \emph{$I$-norm}; (see 
\cite[(2.25)]{paterson}.) Thus, we have $\sup_{v \in G^0} \| \operatorname{Ind} 
v (f) \| < \infty$. In fact, this estimate works for any $*$-representation (see 
\cite[Theorem~2.2.1]{paterson}). 

\begin{defn}\label{def:reduced-groupoid-algebra}
	Let $G$ be a locally compact locally Hausdorff groupoid with a left Haar 
	system $\lambda$. The \emph{maximal groupoid $C^*$-algebra} $C^*(G, 
	\lambda)$ of $G$ is the $C^*$-envelope of $C_c(G)$. The \emph{reduced 
	groupoid $C^*$-algebra} $C^*_{r}(G, \lambda)$ of $G$ is the 
	$C^*$-completion of $C_c(G)$ under the norm $\| f \|_{\operatorname{red}} = 
	\sup_{v \in G^0} \| \operatorname{Ind} v (f) \|$. Equivalently, $C^*_{r}(G, 
	\lambda)$ is the completion of the image of the $*$-representation 
	$\bigoplus_{v \in G^0} \operatorname{Ind}v$ of $C_c(G)$ on $\bigoplus_{v 
	\in G^0} L^2(G_v, \lambda_v)$.
\end{defn}

We may write $C^*(G)$ and $C^*_{r}(G)$ instead of $C^*(G, \lambda)$ and $C^*_{r}(G, \lambda)$ if the Haar system $\lambda$ is clear from the context. There is a canonical embedding 
\begin{equation} \label{eq:canonical-embedding-unit-space}
C_0(G^0) \hookrightarrow M(C^*_r(G)) \; ,
\end{equation}
where $M(C^*_r(G))$ is the multiplier algebra of $C^*_r(G)$, so that for any $ f \in C_0(G^0) $ and for any $ g $ in the dense subalgebra $ C_c(G) $ in $ C^*_r(G) $, the convolution product of $ g $ by $ f $ from the left and right are also in $C_c(G)$ and are given by
\[
(f \ast g) \ (x) = f(r(x)) \cdot g (x) \ \text{and}\  (g \ast f ) \ (x) =  g(x) \cdot f \left(d (y)\right)
\]
for any $x \in G$.

\begin{eg}\label{example:reduced-groupoid-algebra}
	For the groupoids in Example~\ref{example:groupoids}, we have:
	\begin{enumerate}
		\item for any locally compact Hausdorff space $X$ and a positive 
		regular Borel measure $\mu$ on $X$ with full support, the reduced 
		$C^*$-algebra of its pair groupoid $X \times X$ together with the left 
		Haar measure associated to $\mu$ is isomorphic to the algebra of 
		compact operators on $L^2(X, \mu)$. (see, for example, 
		\cite[Theorem~3.1.2]{paterson};) 
		\item for any action $\calpha$ of a locally compact group $G$ on a 
		locally compact second countable Hausdorff space $X$, the reduced 
		$C^*$-algebra 
		of its transformation groupoid is isomorphic to the reduced crossed 
		product $C_0(X) \rtimes_r G$. In fact, when $G$ is unimodular with a 
		Haar measure $m$, the canonical identification between $C_c(X \rtimes 
		G)$ and $C_c(G, C_c(X))$, where a function on $X \times G$ is 
		identified with an iterated function on $G$ and $X$, commutes with the 
		convolution products and the $*$-operation. Here the convolution 
		product and the $*$-operation on $C_c(G, C_c(X))$ are defined by 
		\[
			(f \ast f') (g)(x) = \int_{G} f(h)(x) \cdot f'(h^{-1}g)(\calpha_{h^{-1}}(x)) \, d m (h) 
		\]
		and
		\[
			f^* (g)(x) = \overline{f(g^{-1})( \calpha_{g^{-1}}(x) )}
		\]
		for any $g \in G$ and $x \in X$, so that it forms a dense subalgebra of $C_0(X) \rtimes_r G$;
		\item if $G_U$ is the reduction of a locally compact locally Hausdorff 
		groupoid $G$ to an open subset $U$ of the unit space $G^0$, then the 
		inclusion $C_c(G_U) \subset C_c(G)$ induces an embedding $C^*_r(G_U) 
		\hookrightarrow C^*_r(G)$, whose image coincides with the hereditary 
		subalgebra $C_0(U) \cdot C^*_r(G) \cdot C_0(U)$, where we used the 
		canonical embeddings $C_0(U) \subset C_0(G^0) \hookrightarrow 
		M(C^*_r(G))$. 
	\end{enumerate}
\end{eg}

Next we review some facts about order zero maps and nuclear dimension. If $A = A_0 \oplus A_1 \oplus \ldots \oplus A_d$ is a $C^*$-algebra, and $\varphi^{(k)} \colon A_k \to B$ are order zero contractions into some $C^*$-algebra $B$ for $k=0,1,\ldots,d$, we say that the map $\varphi = \sum_{k=0}^d \varphi^{(k)}$ is a \emph{piecewise contractive} $(d+1)$-decomposable completely positive map.

The following fact concerning order zero maps is standard and used often in the literature. It follows immediately from the fact that cones over finite dimensional $C^*$-algebras are projective. See \cite[Proposition 1.2.4]{winter-covering-II} and the proof of \cite[Proposition 2.9]{winter-zacharias}. We record it here for further reference.

\begin{Lemma}
\label{Lemma:lifting-decoposable-maps}
Let $A$ be a finite dimensional $C^*$-algebra, let $B$ be a $C^*$-algebra and let $I \lhd B$ be an ideal. Then any piecewise contractive $(d+1)$-decomposable completely positive map $\varphi \colon A \to B/I$ lifts to a piecewise contractive $(d+1)$-decomposable completely positive map $\widetilde{\varphi} \colon A \to B$. \qed
\end{Lemma}

We record for the reader's convenience a few lemmas from \cite{Hirshberg-Wu16} which we will re-use in this paper.
\begin{Lemma}[{\cite[Lemma 1.2]{Hirshberg-Wu16}}] 
 \label{Lemma:finite-dimnuc}
 Let $B$ be a separable and nuclear $C^*$-algebra and $B_0$ a dense subset of the unit ball of $B$. Then $\dimnuc(B) \leq d$ if and only if for any finite subset $F \subseteq B_0$ and for any $\eps>0$ there exists a $C^*$-algebra $A_{\eps} = A_{\eps}^{(0)} \oplus \cdots \oplus A_{\eps}^{(m)}$ and completely positive maps 
 \[
  \xymatrix{
   B \ar[dr]_{\psi = \bigoplus_{l=0}^m \psi^{(l)} \quad } \ar@{.>}[rr]^{\id} &  & B \\
   & A_{\eps} = \bigoplus_{l=0}^m A_{\eps}^{(l)} \ar[ur]_{\quad\varphi = \sum_{l=0}^m \varphi^{(l)}} &
  }
 \]
 so that
 \begin{enumerate}
  \item $\psi$ is contractive,
  \item each $\varphi^{(l)}$ is a sum $\varphi^{(l)} = \sum_{k=0}^{ d^{(l)} } \varphi^{(l,k)}$ of $(d^{(l)} + 1)$-many order zero contractions,
  \item $\|\varphi(\psi(x)) - x\| < \eps$ for all $x \in F$, and
  \item $\displaystyle \sum_{l=0}^m ( \dimnuc(A_{\eps}^{(l)}) + 1) (d^{(l)}+1) \leq d+1$.
 \end{enumerate}
 \qed
\end{Lemma}

\begin{lem}[{\cite[Lemma 1.3]{Hirshberg-Wu16}}]\label{lem:separable-dimnuc} 
 Let $G$ be a locally compact Hausdorff and second countable group, and let $A$ 
 be a $G$-$C^*$-algebra. Then any countable subset $S \subset A$ is contained 
 in a $G$-invariant separable $C^*$-subalgebra $B \subset A$ with $\dimnuc(B) 
 \leq \dimnuc(A)$. In particular, $A$ can be written as a direct limit of 
 separable $G$-$C^*$-algebras with nuclear dimension no more than $\dimnuc(A)$.
 \qed
\end{lem}

\begin{lem}[{\cite[Lemma 1.4]{Hirshberg-Wu16}}] \label{lem:quasicentral-approximate-unit}
 Let $X$ be a locally compact Hausdorff space, let $G$ be a locally compact Hausdorff group, and let $\calpha: G \curvearrowright X$ be a continuous action. Suppose $U$ is a $G$-invariant open subset of $X$. Then there is a quasicentral approximate unit for $C_0(U) + C_0(U) \rtimes_\alpha G \subset M(C_0(X) \rtimes_\alpha G)$ which is contained in $C_c(U)_{+, \leq 1}$.
 \qed
\end{lem}

We record the following two results from classical dimension theory, which are used later in the paper.
Those two results apply to the case of second countable Hausdorff spaces, since 
any 
second countable Hausdorff space is paracompact, Hausdorff and totally normal. 
For a 
discussion of different variants of paracompactness and normality, we refer the 
reader to \cite[Chapter 1, section 4]{Pears75}.

\begin{thm}[{\cite[Chapter 3, Theorem 6.4]{Pears75}}]\label{thm:Pears75}
 If $M$ is a subspace of a totally normal space $X$, then $\dim (M) \leq \dim (X)$.
 \qed
\end{thm}

\begin{prop}[{\cite[Chapter 9, Proposition 2.16]{Pears75}}] \label{prop:Pears75}
 If $X$ and $Y$ are weakly paracompact $T_4$-spaces and $f\colon X \to Y$ is a continuous open surjection such that $f^{-1}(y)$ is finite for each point of $Y$, then $\dim(X) = \dim (Y) $. \qed
\end{prop}

\section{Subtubular covers for flows without short periods}
\label{Section:Tube}

In this section, we study the case where there is a nontrivial lower bound on the periods, i.e.\ $Y = Y_{> R}$ for some $R>0$.  
Our aim is to show that when this lower bound $R$ is large enough, we can obtain covers of the space by long enough tubes (or subsets of tubes) with controlled dimensions, and these in turn yield partitions of unity made up of almost invariant functions. The content of this section parallels and extends \cite[Section 6]{HSWW16}, where the case of free flows is treated. Thus unsurprisingly, this section also makes uses of technical results from \cite{Kasprowski-Rueping} (which is a generalization of results from \cite{BarLRei081465306017991882}).

Before we introduce the necessary terminology borrowed from \cite{BarLRei081465306017991882} and \cite{Kasprowski-Rueping}, we point out that we are not using the full power of the constructions and results in those papers. Instead we are looking at a somewhat simplified situation: while in those papers, they need to consider a proper action by a discrete group $G$ which commutes with the flow, this is of no particular interest to us in the present paper, and thus we consider only the flow itself. In other words, we take $G$ to be the trivial group.

\begin{defn}[cf.~{\cite[Definition 2.2]{BarLRei081465306017991882}}]
\label{defn:tubes}
 Let $\calpha \colon  G  \curvearrowright Y$ be a flow.
 A \emph{tube} (or a \emph{box}) is a compact subset $B \subset Y$ such that 
 there exists a real number $ l = l_B $ with the property that for every $y \in 
 B$, there exist real numbers $ a_-(y) \le 0 \le a_+(y)$ and $\varepsilon(y) > 
 0 $ satisfying 
 \begin{align*}
  l = &\  a_+(y) - a_-(y) ;\\
  \calpha_t(y) \in &\ B \ \text{ for\ } t \in [a_-(y), a_+(y)  ]  ;\\
  \calpha_t(y) \not\in &\ B \ \text{ for\ } t \in ( a_-(y) - \varepsilon(y), a_-(y) ) \cup (a_+(y), a_+(y) +  \varepsilon(y) ).
 \end{align*}
 Moreover, for any tube $B$, the following data are associated to it:
 \begin{enumerate}
  \item The \emph{length} $ l_B $;
  \item The topological interior $B^o$ is called an \emph{open tube};
  \item The subset $ S_B = \left\{ y \in B \middlebar a_-(y) + a_+(y) =0 \right\}$ is called the \emph{central slice} of $B$;
  \item The subsets $ \partial_+ B $ and $ \partial_- B $, respectively called the \emph{top} and the \emph{bottom} of $B$, are defined by 
   $$ \partial_\pm B = \left\{ y \in B \middlebar a_\pm (y) =0 \right\} = \left\{ \calpha_{a_\pm(y)} (y) \middlebar y \in S_B \right\} ; $$
  \item Similarly, the \emph{open top} $ \partial_+ B ^o $ and the \emph{open bottom} $ \partial_- B ^o $ are defined by
   $$ \partial_\pm B^o = \left\{ \calpha_{a_\pm(y)} (y) \middlebar y \in S_B \cap B^o \right\}  \; , $$
   and the \emph{open central slice} is defined by 
   \[
	   S_{B^o} = S_B \cap B^o \; .
   \]
 \end{enumerate}
\end{defn}

Intuitively speaking, what a tube is to a dynamical system of ${\mathbb{R}^{}}$ is what a Rokhlin tower is to a dynamical system of ${\mathbb{Z}^{}}$, in that they function as a local trivialization of the action. 

\begin{lem}[cf.~{\cite[Lemma 2.6]{BarLRei081465306017991882}}]
\label{lem:tubes-basics}
 Let $B \subset Y$ be a tube of length $l = l_B$. Then
 \begin{enumerate}
  \item The maps 
   $$ a_\pm: B \to {\mathbb{R}^{}}, \ y \mapsto a_\pm(y) $$
  are continuous;
  \item There exists $\varepsilon_B > 0 $ depending only on $B$ such that the numbers $\varepsilon(y)$ appearing in the definition of a tube can be chosen so that $ \varepsilon(y) \ge \varepsilon_B$ holds for all $y \in B$ (in fact, it is clear that we can then simply take $\varepsilon(y) = \varepsilon_B$);
  \item The map 
   $$ S_B \times \left[ - \frac{l}{2}, \frac{l}{2} \right] \to B \; , \quad  (y, t) \mapsto \calpha_t(y) $$
  is a homeomorphism, whose inverse is given by the map
  \[
	  B \to S_B \times \left[ - \frac{l}{2}, \frac{l}{2} \right] \; , \quad  y 
	  \to \left( \calpha_{\frac{a_-(y) + a_+(y)}{2}} (y), \frac{l}{2} - a_+(y) 
	  \right) \; .
  \]
 \end{enumerate}
\end{lem}

\begin{rmk}\label{rmk:tube-basics}
 The last statement in the previous lemma can be turned into an alternative definition for tubes: a tube is a pair $(S, l)$, where $S \subset Y$ is compact, $l >0$, and the map 
  \[
	  S \times \left[ - \frac{l}{2}, \frac{l}{2} \right] \to Y , \ (y, t) \mapsto \calpha_t(y) 
  \] 
 is an embedding. To relate to Definition~\ref{defn:tubes}, we can establish, for any $y \in S$ and $t \in \left[ - \frac{l}{2}, \frac{l}{2} \right]$, the identities 
 \[	 
	 a_- \left( \calpha_t(y) \right) = - t - \frac{l}{2} \text{ and } a_+ \left( \calpha_t(y) \right) = - t + \frac{l}{2} \; . 
 \]
\end{rmk}

Recall that $\mathrm{per}_{\calpha}(y) = \inf\left\{t > 0 \middlebar \calpha_t(y) = y \right\}$ is the minimal period of the orbit of $y$. Observe that the existence of a tube $B$ around a point $y \in Y$ requires the necessary condition $l_B < \mathrm{per}_{\calpha}(y)$. In fact, this is also sufficient: 

\begin{lem}[{\cite[Lemma 2.11 and Lemma 2.16]{BarLRei081465306017991882}}]\label{lem:existence-tubes}
 For any $y \in Y$ not fixed by $\calpha$, and for any $l \in (0, \mathrm{per}_{\calpha}(y))$, there exists a tube $B$ with $l_B = l$ and $y \in S_B \cap B^o$. 
\end{lem}

\begin{cor}\label{cor:long-period-open}
 The subset $Y_{> R}$ is open and $\calpha$-invariant, while $Y_{\leq R}$ is closed and $\calpha$-invariant.
\end{cor}
\begin{proof}
 The $\calpha$-invariance is obvious. To show $Y_{> R}$ is open, we see that given any $y \in Y_{> R}$, by Lemma~\ref{lem:existence-tubes}, there is a tube $B_y$ such that $l_{B_y} = R$ and $y \in S_{B_y} \cap B_y^o$. This implies that for any point $y'$ in the open neighborhood $B_y^o$ around $y$, we have $\mathrm{per}_{\calpha}(y') > R$, that is, the neighborhood $B_y^o \subset Y_{> R}$. Therefore $Y_{> R}$ is open while $Y_{\leq R}$ is closed.
\end{proof}

Occasionally we will need to \emph{stretch} a tube, as formalized in the following lemma, whose proof is immediate from the definition.

\begin{lem}\label{lem:stretching-tube}
 Let $B$ be a tube and $\displaystyle 0 < L < \frac{\varepsilon_B}{2} $. Then the set $ \calpha_{[-L, L]} (B) $ is also a tube with the same central slice and a new length $ l_B + 2 L $. 
\end{lem}

In order to study the nuclear dimension of the crossed product ${C}_{0}(Y) \rtimes {\mathbb{R}^{}} $, we would like to decompose $Y$ in a dimensionally controlled fashion into open subsets such that the flow is trivialized when restricted to each open subset. This motivation leads to the definition of \emph{tube dimension} in \cite[Definition~7.6]{HSWW16}, which works well for free flows. In our current situation, we need a quantitative generalization of the notion of tube dimension.

\begin{Notation}
	\label{def:mult}
	Given a topological space $X$ with a collection $\mathcal{U}$ of subsets of $X$, its \emph{multiplicity} $\operatorname{mult}(\mathcal{U})$ is the infimum of natural numbers $d$ satisfying that the intersection of any $d+1$ pairwise distinct elements in $\mathcal{U}$ is empty, while its \emph{chromatic number} $\operatorname{chrom}(\mathcal{U})$ is the infimum of natural numbers $d$ satisfying that $\mathcal{U}$ can be written as a union of $d+1$ many families of disjoint subsets of $X$. 
\end{Notation}

\begin{rmk}\label{rmk:multiplicity_vs_chromatic_number} 
	It is clear from the definition that $\operatorname{mult}(\mathcal{U}) \leq \operatorname{chrom}(\mathcal{U})$ for any collection $\mathcal{U}$. The opposite direction does not hold in general.
\end{rmk}

\begin{defn}\label{defn:subtubular-cover}
  Let $ (Y, {\mathbb{R}^{}}, \calpha) $ be a topological flow and let $K \subset Y$ be a subset. Let $L \in [0,\infty)$ and $d \in \Z^{\geq 0}$. A \emph{subtubular cover of $K$ with width $\geq L$ and multiplicity (respectively, chromatic number) $\leq d+1$} is a finite collection $ \mathcal{U} $ of open subsets of $Y$ satisfying:
 \begin{enumerate}
  \item\label{defn:subtubular-cover:width} for any $ y\in K $, there is $U \in \mathcal{U}$ such that $ \calpha_{[-L, L]}(y) \subset U $;
  \item\label{defn:subtubular-cover:subtubular} each $ U \in \mathcal{U} $ is contained in a tube $B_U$; 
  \item\label{defn:subtubular-cover:mult} $ \operatorname{mult} ( \mathcal{U} ) 
  \leq d+1$ (respectively,  $ \operatorname{chrom} ( \mathcal{U} ) \le d + 1 $).
 \end{enumerate}
\end{defn}

\begin{rmk}
 The relation with tube dimension given in \cite[Definition~7.6]{HSWW16} is that $ \operatorname{dim}_\mathrm{tube} (\calpha) \leq d$ if and only if for any $L> 0 $ and any compact subset $K \subset Y$, there exists a subtubular cover of $K$ with width $\geq L$ and multiplicity $\leq d+1$. In fact, one may also use the chromatic number in place of the multiplicity in the above statement, thanks to \cite[Proposition~7.23]{HSWW16}. 
\end{rmk}

We shall prove that for a flow $\calpha$ on a locally connected, locally 
compact and second countable Hausdorff space $ Y $ with finite topological 
dimension, 
when the lower bound on the periods is not too small compared to $L$, any 
compact subset $K \subset Y$ admits a subtubular cover with width $\geq L$ and 
multiplicity $\leq d+1$. For this, we will need to invoke a result by 
Kasprowski and R\"{u}ping (\cite[Theorem~5.2]{Kasprowski-Rueping}), which 
itself is an improvement of a construction of the so-called ``long thin covers" 
by Bartels, L\"{u}ck and Reich (\cite[Theorem~1.2, 
Proposition~4.1]{BarLRei081465306017991882}). Since we are dealing with a 
simplified situation, we shall give a somewhat different formulation of their 
theorem that is sufficient for our purposes; see 
Remark~\ref{rmk:KRcoverbytubes-differences}.

\begin{thm}[See {\cite[Theorem~5.2]{Kasprowski-Rueping}}]
	\label{thm:KRcoverbytubes}
	Let $Y$ be a locally compact metrizable space with a 
	continuous 
	action $\calpha$ by $\mathbb{R}$. Let $L$ be a positive number. Then there 
	is 
	a collection of open tubes of multiplicity at most $5 
	 \operatorname{dim}^{\!+1}(Y)$ satisfying the property that for any point 
	$y 
	\in Y_{>20L}$, there is an open tube in this collection containing 
	$\calpha_{[-L, L]}(y)$.
	\qed
\end{thm}

\begin{rmk}\label{rmk:KRcoverbytubes-differences}
	We explain some deviation from the original formulation of the above theorem in \cite{Kasprowski-Rueping}:
	\begin{enumerate}
		\item In the original version, the authors consider not only a flow 
		$\calpha$ on $Y$, but also a proper action of a discrete group $G$ that 
		commutes with $\calpha$, and the cover they produce is required to be a 
		so-called \emph{$\mathcal{F}in$}-cover with regard to the second 
		action: it is invariant, and for each open set $U$ in the cover, only 
		finitely many elements of $G$ fix $U$, while any other element carry 
		$U$ to a set disjoint from $U$ 
		({\cite[Notation~2.1(4)]{Kasprowski-Rueping}}). Since this is not 
		needed for proving our main result, we drop this assumption, or 
		equivalently, we assume this extra group $G$ that appears in their 
		theorem to be the trivial group, in which case the 
		{$\mathcal{F}in$}-cover condition is automatic.
		\item The dimension estimate in the original version is in terms of the 
		\emph{small inductive dimension} $\mathrm{ind}(Y)$, but as they 
		remarked in \cite[Theorem~3.5]{Kasprowski-Rueping}, in the context 
		their Theorem~5.2 applies to, where $Y$ is locally compact and second 
		countable, the small inductive dimension is equal to the covering 
		dimension $\operatorname{dim}(Y)$.
		\item It is not made explicit in the original statement of their 
		theorem that the cover consists of open tubes, but this is evident from 
		their proof: the cover is made up of the sets $\calpha_{(-4L, 4L)} 
		(B_i^k)$ for $i \in \mathbb{N}$ and $k \in \left\{0, \ldots, 
		\operatorname{dim}(Y)\right\}$, and each of them is the interior of a 
		tube $\displaystyle \calpha_{[-4L, 4L]} (\overline{B_i^k} )$, which is 
		restricted from the larger tube $\calpha_{[-10L, 10L]} (S_i)$ 
		constructed in \cite[Lemma~4.6]{Kasprowski-Rueping}. 
		\item In the original statement of the theorem, it is only claimed that the cover has {dimension} at most $5 \operatorname{dim}^{\!+1}(Y)$. However an examination of the proof shows that in fact the cover they obtain has \emph{multiplicity} at most $5 \operatorname{dim}^{\!+1}(Y)$.
	\end{enumerate}
\end{rmk}

\begin{cor}\label{cor:estimate-subtubular-cover}
	Let $Y$ be a locally compact and second countable Hausdorff space and 
	$\calpha$ a 
	flow 
	on $Y$. Then for any $L \in [0,\infty)$, any $d \in \Z^{\geq 0}$ and any 
	compact subset $K \subset Y_{> 20 L}$, there is a subtubular cover of $K$ 
	with width $\geq L$ and multiplicity $\leq 5 \cdot 
	\operatorname{dim}^{\!+1}(Y)$. 
\end{cor}

\begin{proof}
	This is almost a direct consequence of Theorem~\ref{thm:KRcoverbytubes}. 
	The only point to be clarified is the finiteness of the cover. Let $L \in 
	[0,\infty)$, $d \in \Z^{\geq 0}$ and a compact subset $K \subset Y_{> 20 
	L}$ be given. By Theorem~\ref{thm:KRcoverbytubes}, there is a collection 
	$\mathcal{U}$ of open tubes of multiplicity at most $5 
	(\operatorname{dim}(Y) + 1)$ satisfying the property that for any point $y 
	\in Y_{>20L}$, there is an open tube $U_y$ in this collection containing 
	$\calpha_{[-L, L]}(y)$. For any $y \in K$, 
	we may write $y = \calpha_{t_y}(z_y)$ for $z_y$ in the central slice 
	$S_{\overline{U_y}}$ and $t_y \in \left[ -\frac{l_{\overline{U_y}}}{2} ,  
	\frac{l_{\overline{U_y}}}{2}\right]$. By the local trivialization 
	associated to the tube $\overline{U_y}$ as in Lemma~\ref{lem:tubes-basics}, 
	we have $[t_y - L , t_y + L] \subset \left( -\frac{l_{\overline{U_y}}}{2} 
	,  \frac{l_{\overline{U_y}}}{2}\right)$. Thus there exists $\delta_y > 0$ 
	such that $t_y - L - \delta_y > -\frac{l_{\overline{U_y}}}{2}$ and $t_y + L 
	+ \delta_y < \frac{l_{\overline{U_y}}}{2}$. Define the set $V_y = 
	\calpha_{(t_y - \delta_y , t_y + \delta_y)}(S) \cap U_y$. Then using the 
	local trivialization, we see that $V_y$ is an open neighborhood of $y$ such 
	that $\calpha_{[-L, L]}(V_y) \subset U_y$. 
	
	By the compactness of $K$, there is a finite subset $\left\{y_1 , \ldots, y_n\right\} \subset K$ such that $K \subset V_{y_1} \cup \ldots \cup V_{y_n}$. Define a finite subcollection $\mathcal{U}'$ of $\mathcal{U}$ to be $\left\{U_{y_1} , \ldots , U_{y_n}\right\}$. By our construction above, $\mathcal{U}'$ satisfies Condition~\ref{defn:subtubular-cover:width} in Definition~\ref{defn:subtubular-cover}, while it inherits Conditions~\ref{defn:subtubular-cover:subtubular} and~\ref{defn:subtubular-cover:mult} from $\mathcal{U}$. Therefore $\mathcal{U}'$ is a subtubular cover of $K$ with width $\geq L$ and multiplicity $\leq d+1$. 
\end{proof}


\section{From subtubular covers to partitions of unity}

In order to apply Definition~\ref{defn:subtubular-cover} and Corollary~\ref{cor:estimate-subtubular-cover} to the context of $C^*$-algebras, our next step is to show that because the subtubular covers we obtained have large overlaps along flow lines, they give rise to partitions of unity that are \emph{almost flat} along the flow lines. What follows is a refinement of the argument presented in \cite[Section~8]{HSWW16}. 

\begin{defn}[{\cite[Definition~7.11]{HSWW16}}]\label{definitionofflowwiseLipschitz}
	Let $ \calpha: {\mathbb{R}^{}} \curvearrowright Y $ be a flow and $ F : Y 
	\to X $ a map to a metric space $(X, d)$. The map $F$ is called 
	\emph{$\calpha$-Lipschitz with constant $\delta$}, if for every $y \in Y$, 
	the map $ t \mapsto F(\calpha_t (y) ) $ is Lipschitz with  constant 
	$\delta$. In other words, we have 
	\[ 
	d( F(\calpha_t (y) ), F(y) ) \le \delta \cdot |t|  
	\]
	for all $y \in Y$ and $t \in {\mathbb{R}^{}}$.
\end{defn}

As in \cite[Remark~8.12]{HSWW16}, one convenient way to produce flow-wise 
Lipschitz functions from any given function is to \emph{smear} it along the 
flow. More precisely, for any bounded Borel function $f$ on $Y$ and any 
$\lambda_+ , \lambda_- \in {\mathbb{R}^{}} $ such that $ \lambda_+ > \lambda_- 
$, we define $\mathbb{E} (\calpha_*)_{[\lambda_-, \lambda_+]} (f) : Y \to 
{\mathbb{C}^{}}$ by
\begin{equation}\label{definitionofsmearing}
	\mathbb{E} (\calpha_*)_{[\lambda_-, \lambda_+]} (f) (y) = 
	\frac{1}{\lambda_+ - \lambda_-} \int_{\lambda_-} ^{\lambda_+} f 
	(\calpha_{-t} (y) ) \: d t .
\end{equation}
This has the advantage that it preserves, if applied to a family of functions, the property of being a partition of unity. This works for the following generalization of partitions of unity.

\begin{defn}[{\cite[Definition~7.14]{HSWW16}}]\label{def:relative-pou}
	Let $ X $ be a topological space and $ A \subset X $ a subset. Let $ \mathcal{U} = \left\{ U_i \right\}_{i\in I} $ be a locally finite collection of open sets in $X$ such that $A \subset \bigcup \mathcal{U}$. Then a \emph{partition of unity for $ A \subset X $ subordinate to $ \mathcal{U} $} is a collection of continuous functions $ \left\{ f_i \middlebar  X \to [0,\infty) \right\}_{i\in I} $ such that
	\begin{enumerate}
		\item for each $i\in I$, the support of $f_i$ is contained in $ U_i $;
		\item for all $x \in A$, one has $\displaystyle \sum_{i\in I} f_i (x) =1 $.
	\end{enumerate}
\end{defn}

We make use of certain simplicial techniques common in dimension theory, which allows us to pass from the weaker notion of multiplicity to the stronger one of chromatic number.

\begin{defn}[{\cite[Definition~7.17]{HSWW16}}]\label{def:simplicial-complex}
	For us, an \emph{abstract simplicial complex} $Z$ consists of:
	\begin{itemize}
		\item a set $Z_0$, called the set of \emph{vertices}, and
		\item a collection of its finite subsets closed under taking subsets, called the collection of \emph{simplices}. 
	\end{itemize}
	We often write $\sigma \in Z$ to denote that $\sigma$ is a simplex of $Z$. We also associate the following structures to $Z$:
	\begin{enumerate}
		\item\label{def:simplicial-complex:dimension} The dimension of a simplex is the cardinality of the corresponding finite subset minus $1$, and the (simplicial) dimension of the abstract simplicial complex is the supremum of the dimensions of its simplices.
		\item\label{def:simplicial-complex:realization} The \emph{geometric realization} of an abstract simplicial complex $Z$, denoted as $|Z|$, is the set of tuples
		\[
		\bigcup_{\sigma \in Z} \left\{ (z_v)_{v} \in [0,1]^{Z_0} ~\middle|~  \sum_{v\in\sigma} z_v = 1 \,, ~\text{and}~ z_v = 0 ~\text{for~any}~ v \in Z_0 \setminus \sigma \right\}.
		\]
		Similarly for a simplex $\sigma$ of $Z$, we define its \emph{closed} (respectively, \emph{open}) \emph{geometric realization} $\overline{|\sigma|}$ (respectively, $|\sigma|$) as follows:
		\begin{align*}
		\overline{|\sigma|} & = \left\{ (z_v)_v \in |Z| \middlebar  \sum_{v\in\sigma} z_v = 1 \right\} \\
		|\sigma| & = \left\{ (z_v)_v \in |Z| \middlebar  \sum_{v\in\sigma} z_v = 1 ~ \text{with}~ z_v >0 ~\text{for~any~} v\in\sigma \right\}\; .
		\end{align*}
		\item\label{def:simplicial-complex:metric} 
		We consider the $\ell^1$-topology on $|Z|$, induced by the $\ell^1$-metric $d^1: |Z| \times |Z| \to [0, 2]$ defined by 
		\[ 
		d^1\Bigl( (z_v)_v , (z'_v)_v \Bigl) = \sum_{v \in Z_0} |z_v - z'_v | \; .
		\]
		\item\label{def:simplicial-complex:star} For any vertex $v_0 \in Z_0$, the \emph{(simplicial) star} around $v_0$ is the set of simplices of $Z$ that contain $v_0$, and the \emph{open star} around $v_0$ is the union of the open geometric realizations of such simplices in $|Z|$, that is, the set 
		\[
		\left\{ (z_v)_v \in |Z| \middlebar  z_{v_0} > 0  \right\} \; .
		\]
		\item\label{def:simplicial-complex:cone} The \emph{simplicial cone} $CZ$ is the abstract simplicical complex
		\[
		\left\{ \sigma, \sigma \sqcup \left\{\infty\right\} \middlebar \sigma \in Z \right\} \; ,
		\]
		where $\infty$ is an additional vertex. More concretely, we have $(CZ)_0 = Z_0 \sqcup \left\{\infty\right\}$, each simplex $\sigma$ in $Z$ spawns two simplices $\sigma$ and $\sigma \sqcup \left\{\infty\right\}$ in $CZ$, and all simplices of $CZ$ arise this way. 
		\item\label{def:simplicial-complex:subcomplex} A \emph{subcomplex} of $Z$ is an abstract simplicial complex $Z'$ with $Z'_0 \subset Z_0$ and $Z' \subset Z$. It is clear that there is a canonical embedding $|Z'| \subset |Z|$ preserving the $\ell^1$-metric.
	\end{enumerate}
\end{defn}

As in \cite[Section~8]{HSWW16}, by making use of the smearing technique and certain simplicial techniques, we can pass from a subtubular cover as in Corollary~\ref{cor:estimate-subtubular-cover} to a certain partition of unity that we can later use to estimate the nuclear dimension of the crossed product $C^*$-algebra. This is described in Proposition~\ref{prop:get-partition-of-unity}. Before stating it and two preparatory results, we sketch the main strategy: 
\begin{enumerate}
	\item \label{strategy:prop:get-partition-of-unity:shrink} shrink the cover along the flow lines \textendash\ the width of the cover tells us how long we can do this without destroying the covering property of the shrunken sets;
	\item \label{strategy:prop:get-partition-of-unity:pou} use \cite[Lemma~8.15]{HSWW16} to pick a partition of unity 
	subordinate to the shrunken cover;
	\item \label{strategy:prop:get-partition-of-unity:smear} apply \cite[Lemma~8.13]{HSWW16} to smear this partition of unity to get a new one that is flow-wise Lipschitz, in which process the supports of the functions are allowed to grow back to the original unshrunken cover;
	\item \label{strategy:prop:get-partition-of-unity:simplicial} in order to obtain a control on the chromatic number, instead of just the multiplicity, we make use of the nerve complex of the original cover: viewing the above flow-wise Lipschitz partition of unity as a flow-wise Lipschitz map from the space to the nerve complex of the original cover, we can use it to pull back a canonical partition of unity on this finite dimensional simplicial complex subordinate to a canonical cover with controlled chromatic number. 
\end{enumerate}

The following preparatory result is needed in carrying out Step~\ref{strategy:prop:get-partition-of-unity:simplicial} above. 

\begin{prop}\label{prop:get-simplicial-complex}
 For any $d \in \Z^{\geq 0}$ and for any $\delta > 0$, there exists $Q = 
 Q(\delta, d) > 0$ such that for any flow $\calpha$ on a locally connected, 
 locally compact and second countable Hausdorff space $Y$ with topological 
 dimension 
 $\leq d$, and for any compact subset $ K \subset Y_{> Q} $, there exists a 
 finite simplicial complex $ Z $ of dimension at most $5(d+1)-1$, along with a 
 map $ F: Y \to | CZ | $ satisfying:
  \begin{enumerate}
   \item 
   \label{prop:get-simplicial-complex-a}
   $ F $ is $\calpha$-Lipschitz with  constant $\delta$;
   \item 
   \label{prop:get-simplicial-complex-b}
   for any vertex $  v \in Z_0 $ (the vertex set of $Z$), the preimage of the open star around $v$ is contained in a tube $B_v$; 
   \item 
   \label{prop:get-simplicial-complex-c}
   $ F(K) \subset | Z | $. 
  \end{enumerate}
\end{prop}

\begin{proof}
 Given $d \in \Z^{\geq 0}$ and $\delta > 0$, set $ L = \frac{ 5(d+1) +1 
 }{\delta}  $ and $Q = 20L$. Now given any flow $\calpha$ on a locally 
 connected, locally compact and second countable Hausdorff space $Y$ with 
 topological 
 dimension $\leq d$ and any compact subset $ K \subset Y_{> Q} $, we define $ 
 \widehat{K} = \calpha_{ \left[ - \frac{L}{4}, \frac{L}{4} \right] } (K)$, 
 which is also a compact subset of $ Y_{> Q} $, and then apply 
 Corollary~\ref{cor:estimate-subtubular-cover} with $\widehat{K}$ in place of 
 $K$, to obtain a subtubular cover of $ \widehat{K} $ with width $\geq L$ and 
 multiplicity $ \leq 5(d+1)$.

 For each $U \in \mathcal{U}$, define $ U' = \left\{ y \in Y \middlebar \calpha_{[-L, L]}(y) \subset U \right\} $, which is open by \cite[Lemma~8.22]{HSWW16}. By our construction, we have $ \calpha_{[-L, L]}(U') \subset U $. Because of the first condition in Definition~\ref{defn:subtubular-cover}, $ \mathcal{U}' = \left\{ U' \right\}_{U\in\mathcal{U}} $ covers $ \widehat{K} $. Now fix a partition of unity $ \left\{ f_U \right\}_{U\in\mathcal{U}} $ for $ \widehat{K} \subset Y $ subordinate to $ \mathcal{U}' $. 
 For any $U \in \mathcal{U} $, set  $\widehat{f}_U =  \mathbb{E} (\alpha)_{ \left[ - L, L \right] } (f_U)$.
 Then by \cite[Lemma~8.16]{HSWW16}, the collection $ \big\{ \widehat{f}_U  \big\}_{U \in \mathcal{U}} $ is a partition of unity for the inclusion $ K \subset X$ subordinate to $ \mathcal{U} $, whose members are $\calpha$-Lipschitz with the constant $ \frac{1 }{L} $, and so is the function $ \displaystyle \mathrm{1}_X - \sum_{U \in \mathcal{U}} \mathbb{E} (\alpha)_{ \left[ - L, L \right] } (f_U) $. 
 
 Hence by \cite[formula~(8.2) in Lemma~8.21]{HSWW16}, if we let $ Z = \mathcal{N}(\mathcal{U}) $, the nerve complex of $\mathcal{U}$ (see \cite[Definition~8.19]{HSWW16}), we have a map
  $$ F= \left( \bigoplus_{U\in\mathcal{U}} \widehat{f}_U \right) \oplus \left( \mathrm{1}_X - \sum_{U \in \mathcal{U}} \mathbb{E} (\alpha)_{ \left[ - L, L \right] } (f_U) \right) : Y \to |\mathcal{N}(\mathcal{U}^+) | = | C Z |, $$
 which is $\calpha$-Lipschitz with regard to the $ \mathit{l}^{1} $-metric with  constant $ \frac{ 5(d+1) +1 }{L} = \delta $, as at most $ (5(d+1) +1) $ summands of $F$ are non-zero at each point. It also maps $ K $ into $ | \mathcal{N}(\mathcal{U}) | $ as $ \left( \mathrm{1}_X - \sum_{U \in \mathcal{U}} \mathbb{E} (\alpha)_{ \left[ - L, L \right] } (f_U) \right) $ vanishes on $K$. Finally, the preimage of the open star around each vertex $ U \in \mathcal{U} $ is contained in $ \mathrm{supp}(\widehat{f}_U) \subset U $, which is in turn contained in a tube $B_U$. 
\end{proof}

The following lemma allows us to shrink the support of a partition of unity in a way that respects the Lipschitz constants. It will be used in the proof of Proposition~\ref{prop:get-partition-of-unity} to strengthen the ``separateness'' of supports of a canonical partition of unity mentioned in Step~\ref{strategy:prop:get-partition-of-unity:simplicial} above. If $(X,\operatorname{dist})$ is a metric space, $Y \subset X$, and $r>0$, we denote $N_{r}(Y) = \{x \in X \mid \operatorname{dist}(x,Y)<r\}$.
 
\begin{lem}\label{lem:shrink-pou}
	Let $(X,\operatorname{dist})$ be a metric space and let $d \in \mathbb{N}$ and $\eta > 0$. Suppose $\left\{f_i\right\}_{i \in I}$ is a partition of unity for $X$ such that
	\begin{enumerate}
		\item\label{lem:shrink-pou:1-mult} for any subset $F \subset I$ with cardinality greater than $d+1$, we have $\prod_{i \in F} f_i = 0$; 
		\item\label{lem:shrink-pou:2-Lipschitz} for any $i \in I$, $f_i$ is Lipschitz with constant $\eta$.
	\end{enumerate}
	Then for any $r \in \left ( 0, \frac{1}{\eta (d+1)} \right)$, there is a 
	partition of unity $\left\{g_i\right\}_{i \in I}$ for $X$ such that
	\begin{enumerate}\setcounter{enumi}{2}
		\item\label{lem:shrink-pou:3-supp} for any $i \in I$, $N_{r} \big( \supp(g_i) \big) \subset  \supp(f_i) ^o $; 
		\item\label{lem:shrink-pou:4-Lipschitz} for any $i \in I$, $g_i$ is Lipschitz with constant $\displaystyle \frac{ (d+2) \eta }{\big( 1 - (d+1) \eta r \big)^2}$.
	\end{enumerate}
\end{lem}

\begin{proof}
	Choose $\varepsilon \in \left( \eta r, \frac{1}{d+1} \right)$ such that 
	\begin{equation}\label{eq:lem:shrink-pou:Lipschitz}
		\frac{ \big( 1 - (d+1) \varepsilon \big) \eta + ( 1 - \varepsilon ) (d+1) \eta }{\big( 1 - (d+1) \varepsilon \big)^2} <  \frac{ (d+2) \eta }{\big( 1 - (d+1) \eta r \big)^2} \; ,
	\end{equation}
	which is possible because $\eta r < \frac{1}{d+1}$ and the strict inequality~\eqref{eq:lem:shrink-pou:Lipschitz} holds when we replace $\varepsilon$ by $\eta r$. Now define 
	\[
	h (t) = \max\left\{t-\varepsilon, 0\right\} \, .
	\]
	So $h$ is Lipschitz with constant $1$. Hence for each $i \in I$, $h \circ f_i : X \to [ 0, \infty )$ is Lipschitz with constant $\eta$, while $f_i \leq 1$ and $\varepsilon < \frac{1}{\eta (d+1)} \leq 1$ together also imply the image of $h \circ f_i$ is within $[0, 1 - \varepsilon ]$. Moreover, for any $x \in \supp(h \circ f_i)$, we have $f_i(x) \geq \varepsilon$. Since $f_i$ is Lipschitz with constant $\eta$, any $x'$ with $\operatorname{dist}(x, x') < r$ satisfies $f_i(x') > \varepsilon - \eta r > 0$, and thus we have $B_r(x) \subset (f_i)^{-1} \big( \mathbb{R} \setminus (\eta r - \varepsilon, \varepsilon - \eta r) \big) $. Consequently, we have $N_{r} \big( \supp(h \circ f_i) \big) \subset  \supp(f_i) ^o $. 
	
	Since for any $x \in X$, we have $\sum_{i \in I} f_i (x) = 1 $ and our assumption on $\left\{f_i\right\}_{i \in I}$ implies that the number of $i \in I$ with $f_i (x) > 0$ is between $1$ and $d+1$,  we have  
	\[
	0 < 1 - (d + 1) \varepsilon  \leq \sum_{i \in I} ( h \circ f_i) (x) \leq 1 - \varepsilon \; .
	\]
	Similarly, one checks that the function $\displaystyle \sum_{i \in I} ( h \circ f_i)$ is Lipschitz with constant~$(d+1) \eta$.
	
	Consequently, we may define a new partition of unity $\left\{g_i\right\}_{i \in I}$ for $X$ by 
	\[
		g_i(x) = \frac{( h \circ f_i) (x)}{ \displaystyle \sum_{ j \in I} ( h \circ f_j) (x) } \quad \text{for~any~} x \in X \; .
	\]
	Then for any $i \in I$, we have $\supp(g_i) = \supp(h \circ f_i)$ and thus $N_{r} \big( \supp(g_i) \big) \subset  \supp(f_i) ^o $. The following computation using the ranges and the Lipschitz constants of $\displaystyle \sum_{i \in I} ( h \circ f_i)$ and $h \circ f_i$, for $i \in I$, shows that each $g_i$ is Lipschitz with a constant equal to the left-hand side of \eqref{eq:lem:shrink-pou:Lipschitz}, and thus also the right-hand side, as desired:
	\begin{align*}
		&\ \left| \frac{\left( h \circ f_i \right) \left( x_1  \right)}{\displaystyle \sum_{ j \in I} \left( h \circ f_j \right) \left( x_1  \right)} - \frac{\left( h \circ f_i \right) \left( x_2  \right)}{\displaystyle \sum_{ j \in I} \left( h \circ f_j \right) \left( x_2  \right)} \right| \\
		\leq &\ \left|  \frac{\left( h \circ f_i \right) \left( x_1  \right)}{\displaystyle \sum_{ j \in I} \left( h \circ f_j \right) \left( x_1  \right)} - \frac{\left( h \circ f_i \right) \left( x_2  \right)}{\displaystyle \sum_{ j \in I} \left( h \circ f_j \right)  \left( x_1  \right)  }   \right| +  \left|  \frac{\left( h \circ f_i \right) \left( x_2  \right)}{\displaystyle \sum_{ j \in I} \left( h \circ f_j \right) \left( x_1  \right)} - \frac{\left( h \circ f_i \right) \left( x_2  \right)}{\displaystyle \sum_{ j \in I} \left( h \circ f_j \right) \left( x_2  \right)}   \right| \\
		= &\ \frac{\left|  \left( h \circ f_i \right) \left( x_1  \right) - \left( h \circ f_i \right) \left( x_2  \right) \right|}{\left|  \displaystyle \sum_{ j \in I} \left( h \circ f_j \right) \left( x_1  \right)   \right|} +  \left|  \left( h \circ f_i \right) \left( x_2  \right)  \right| \frac{ \left|  \displaystyle \sum_{ j \in I} \left( h \circ f_j \right) \left( x_2  \right) - \displaystyle \sum_{ j \in I} \left( h \circ f_j \right)  \left( x_1  \right)  \right|   }{\left| \displaystyle \sum_{ j \in I} \left( h \circ f_j \right)  \left( x_1  \right) \right| \left| \displaystyle \sum_{ j \in I} \left( h \circ f_j \right)  \left( x_2  \right) \right|  }  \\
		\leq &\ \frac{ \eta }{\big( 1 - (d+1) \varepsilon \big)} + \frac{ ( 1 - \varepsilon ) (d+1) \eta }{\big( 1 - (d+1) \varepsilon \big)^2}
	\end{align*}
	for any $x_1, x_2 \in X$. 	
\end{proof}

\begin{prop}\label{prop:get-partition-of-unity}
 For any $d \in \Z^{\geq 0}$, for any $\eta > 0$  and for any $L> 0$, there 
 exists $R = R(L, \eta, d) > 0$ such that for any flow $\calpha$ on a locally 
 connected, locally compact and second countable Hausdorff space $Y$ with 
 topological 
 dimension $\leq d$, and for any compact subset $ K \subset Y_{> R} $, there 
 exist a finite partition of unity $ \left\{ f_i \right\}_{i \in I} $ for the 
 inclusion $ K \subset Y$ satisfying:
  \begin{enumerate}
   \item 
   \label{prop:get-partition-of-unity-a}
   for any $ i\in I $, $ f_i $ is $\calpha$-Lipschitz with constant $\eta$;
   \item 
   \label{prop:get-partition-of-unity-b}
   for any $ i\in I $, $ \calpha_{[-L, L]} \big( \mathrm{supp}( f_i ) \big) $ is contained in a tube $B_i$; 
   \item 
   \label{prop:get-partition-of-unity-c}
   there is a decomposition $I = I^{(0)} \cup \cdots \cup I^{(5(d+1)-1)} $ such that for any $ l \in \left\{ 0, \dots, 5(d+1)-1 \right\} $ and any two different $i, j \in I^{(l)} $, we have 
    $$ \calpha_{[-L, L]} \big( \mathrm{supp}( f_i ) \big) \cap \calpha_{[-L, L]} \big( \mathrm{supp}( f_j ) \big) = \varnothing . $$
  \end{enumerate}
\end{prop}

\begin{proof}
 Given $d \in \Z^{\geq 0}$ and $L, \eta > 0$, choose $\delta \in \left( 0, \big( 2(d+2)^2 (d+3 ) (2 d + 5) L \big) ^{-1} \right)$ such that
 \begin{equation}\label{eq:prop:get-partition-of-unity:eta}
  \delta \cdot \frac{  2(d+2) (d+3 )^2 (2 d + 5) }{\big( 1 - 2(d+2)^2 (d+3 ) (2 d + 5) \delta L \big)^2} < \eta \; .
 \end{equation}
 Set $R = Q(\delta,d)$ as in Proposition~\ref{prop:get-simplicial-complex}. Now 
 given any flow $\calpha$ on a locally connected, locally compact and second 
 countable space $Y$ with topological dimension $\leq d$ and any compact subset 
 $ K \subset Y_{> R} $, we apply Proposition~\ref{prop:get-simplicial-complex} 
 to obtain a finite simplicial complex $ Z $ of dimension at most $5(d+1)-1$ 
 and a map $ F: Y \to | CZ | $ satisfying: 
 {
 	\renewcommand{\theenumi}{\theequation}
 	\renewcommand{\labelenumi}{(\theenumi)}
 \begin{enumerate}
 	\stepcounter{equation}\item\label{eq:prop:get-partition-of-unity:F-a}
 	$ F $ is $\calpha$-Lipschitz with  constant $\delta$;
 	\stepcounter{equation}\item\label{eq:prop:get-partition-of-unity:F-b}
 	for any vertex $  v \in Z_0 $ (the vertex set of $Z$), the preimage of the open star around $v$ is contained in a tube $B_v$; 
 	\stepcounter{equation}\item\label{eq:prop:get-partition-of-unity:F-c}
 	$ F(K) \subset Z $. 
 \end{enumerate}
 }
  
 By \cite[Lemma~8.18]{HSWW16}, there is a partition of unity $ \left\{ \nu_\sigma \right\}_{\sigma \in CZ } $ for $|CZ|$ indexed by the simplices of $CZ$, such that each $\nu_\sigma$ is $2(\dim(CZ) + 1) (\dim(CZ) + 2 ) (2\dim(CZ) + 3)$-Lipschitz and for any different simplices $\sigma$ and $\sigma'$ with the same dimension, we have $\nu_\sigma\nu_{\sigma'} = 0$. In particular, $ \left\{ \nu_\sigma \right\}_{\sigma \in CZ } $ satisfies Conditions~\eqref{lem:shrink-pou:1-mult} and~\eqref{lem:shrink-pou:2-Lipschitz} of Lemma~\ref{lem:shrink-pou}, with $d+1$ in place of $d$ and $2(d+2) (d+3 ) (2 d + 5)$ in place of $\eta$.
 
 Set $r = \delta L$. Since $2(d+2) (d+3 ) (2 d + 5) \cdot (d+2) \cdot r < 1$, we can apply Lemma~\ref{lem:shrink-pou} to $ \left\{ \nu_\sigma \right\}_{\sigma \in CZ } $ and $r$ to obtain a partition of unity $\left\{\mu_\sigma\right\}_{\sigma \in CZ } $ for $|CZ|$ satisfying 
 {
 	\renewcommand{\theenumi}{\theequation}
 	\renewcommand{\labelenumi}{(\theenumi)}
 \begin{enumerate}
 	\stepcounter{equation}\item\label{eq:prop:get-partition-of-unity:mu-containment} for any $\sigma \in CZ$, $N_{r} \big( \supp(\mu_\sigma) \big) \subset  \supp(\nu_\sigma) ^o $; 
 	\stepcounter{equation}\item\label{eq:prop:get-partition-of-unity:mu-Lipschitz}  for any $\sigma \in CZ$, $\mu_\sigma$ is Lipschitz with constant \[
 	\displaystyle \frac{ (d+3) \cdot 2(d+2) (d+3 ) (2 d + 5) }{\big( 1 - (d+2) \cdot 2(d+2) (d+3 ) (2 d + 5) r \big)^2} \;. \]
 \end{enumerate}
 }
 
 Define $I = Z$, that is, the collection of all simplices in $Z$. Then there is a decomposition $ I = \bigsqcup_{l=0}^{d} I^{(l)} $, where $I^{(l)}$ consists of all $l$-dimensional simplices in $Z$, for $l = 0, \ldots, d$. For any $\sigma \in I$, define $ f_\sigma = \mu_\sigma \circ F : Y \to [0,1] $. We claim that $ \left\{f_\sigma\right\}_{\sigma\in I} $ together with the decomposition $ I = \bigcup_{l=0}^{d} I^{(l)} $ is a partition of unity for the inclusion $ K \subset Y $ satisfying the required conditions. 
  
 First, by \eqref{eq:prop:get-partition-of-unity:F-c}, we know that $F(y) \in |Z|$ for any $y \in K$, which then implies $\sum_{\sigma \in Z} \nu_\sigma ( F(y) ) = 1 $ and $\nu_\sigma ( F(y) ) = 0$ for any $\sigma \in CZ \setminus Z$. By \eqref{eq:prop:get-partition-of-unity:mu-containment}, this implies $\mu_\sigma ( F(y) ) = 0$ for any $\sigma \in CZ \setminus Z$ and thus $\sum_{\sigma \in Z} \mu_\sigma ( F(y) ) = 1 $ because $\left\{\mu_\sigma\right\}_{\sigma \in CZ } $ is a partition of unity for $|CZ|$. This shows that $ \left\{f_\sigma\right\}_{\sigma\in I} $ is a partition of unity for the inclusion $ K \subset Y $. It remains to verify the three conditions in the statement of the proposition.
 \begin{enumerate}
 	\item For any $\sigma \in I$, since we know the Lipschitz constants of $\mu_\sigma$ and $F$, we see that $f_\sigma$ is Lipschitz with a constant equal to
 	\[
	 	\delta \cdot \frac{ (d+3) \cdot 2(d+2) (d+3 ) (2 d + 5) }{\big( 1 - (d+2) \cdot 2(d+2) (d+3 ) (2 d + 5) \delta L \big)^2} \;, 
 	\]
 	and thus it is also Lipschitz with constant $\eta$, because of \eqref{eq:prop:get-partition-of-unity:eta}.
 	\item For any $\sigma \in I$, since $F$ is Lipschitz with constant $\delta$ and $\delta L = r$, we have 
 	\begin{equation}\label{eq:prop:get-partition-of-unity:containments}
 		F \left( \calpha_{[-L, L]} \big( \mathrm{supp}( \mu_\sigma \circ F ) \big) \right) \subset N_r \big( \mathrm{supp}( \mu_\sigma ) \big) \subset \mathrm{supp}( \nu_\sigma ) ^o \; .
 	\end{equation}
 	By \cite[Lemma~8.18]{HSWW16}, $\mathrm{supp}( \nu_\sigma ) ^o$ is contained in the open set $V_\sigma$, which, by its definition, is contained in the open star of any vertex in $\sigma$. Hence the preimage of $\calpha_{[-L, L]} \big( \mathrm{supp}( \mu_\sigma \circ F ) \big) $ under $F$ is contained in an open star in $|Z|$. Since $F$ was chosen to satisfy \eqref{prop:get-simplicial-complex-b} of Proposition~\ref{prop:get-simplicial-complex}, we see that $\calpha_{[-L, L]} \big( \mathrm{supp}( \mu_\sigma \circ F ) \big) $ is contained in a tube.
 	\item For any $ l \in \left\{ 0, \dots, 5(d+1)-1 \right\} $ and any two different $\sigma, \sigma' \in I^{(l)} $, by \cite[Lemma~8.18]{HSWW16}, we have $V_\sigma \cap V_{\sigma'} = \varnothing$, and thus $\mathrm{supp}( \nu_\sigma ) ^o \cap \mathrm{supp}( \nu_{\sigma'} ) ^o = \varnothing$. It then follows from \eqref{eq:prop:get-partition-of-unity:containments} that
 	\[
	 	\calpha_{[-L, L]} \big( \mathrm{supp}( f_i ) \big) \cap \calpha_{[-L, L]} \big( \mathrm{supp}( f_j ) \big) = \varnothing \; .
 	\]
 \end{enumerate}
\end{proof}

\section{Actions with bounded periods}
\label{section: bounded periods}

In this section, we shift our attention to the case where a flow $(Y, \R, \calpha)$ has bounded periods, i.e.\ $Y = Y_{\leq R}$ for some $R > 0$. 
Notice that this is equivalent to saying that there is $R > 0$, such that for any $y \in Y$, the orbit of $y$ is equal to $\left\{ \calpha_t (y) \middlebar t \in [0,R] \right\}$. In particular, this implies (and in fact is equivalent to) the following condition defined more generally for a locally compact Hausdorff second countable group $G$ and a locally compact Hausdorff space $Y$:

\begin{defn}[{\cite[Definition~3.1]{Hirshberg-Wu16}}]
	A continuous action $\calpha: G \curvearrowright Y$ is said to have \emph{uniformly compact orbits} if there exists a compact subset $K \subset G$ such that for any $y \in Y$, we have $G \cdot y = K \cdot y$.
\end{defn}

Hence by \cite[Lemma~3.2]{Hirshberg-Wu16}, we know that the quotient space $Y / \mathbb{R}$ is Hausdorff and locally compact and the quotient map $\pi \colon Y \to Y / \mathbb{R}$ is proper. 

\begin{lem}\label{lem:quotient-is-metrizable}
	Let $Y$ be a locally compact second countable Hausdorff space and let 
	$\calpha: \R 
	\curvearrowright Y$ be a continuous action with bounded periods. Then $Y / 
	\mathbb{R}$ is also locally compact second countable Hausdorff.
\end{lem}

\begin{proof}
	Recall that a locally compact Hausdorff space $X$ is second countable if 
	and only if the $C^*$-algebra $C_0(X)$ is separable. Since the quotient map 
	$\pi \colon Y \to Y/ \mathbb{R}$ is a proper map between two locally 
	compact Hausdorff space, by Gelfand duality, we see that 
	$C_0(Y/\mathbb{R})$ embeds into $C_0(Y)$, and in particular 
	$C_0(Y/\mathbb{R})$ is separable.
\end{proof}

We now want to bound the dimension of $Y / \mathbb{R}$. We recall the notion of local dimension. 

\begin{defn}[{\cite[Chapter 5, Definition~1.1]{Pears75}}]
	The \emph{local dimension}, $\operatorname{locdim} (X)$, of a topological space $X$ is defined as follows. If $X$ is empty, then $\operatorname{locdim}(X) = -1$. Otherwise $\operatorname{locdim}(X)$ is the least integer $n$ such that for any $x \in X$, there is some open set $U$ containing $x$ such that $\operatorname{dim}(\overline{U}) \leq n$, or $\infty$ if there is no such integer.
\end{defn}

The local nature of this notion makes it useful when studying the behavior of dimensions with regard to continuous maps. 

\begin{prop}[{\cite[Chapter 5, Proposition~3.4]{Pears75}}]\label{prop:locdim-dim}
	If $X$ is a weakly paracompact normal space (e.g., if $X$ is second 
	countable), then $\operatorname{locdim}(X) = \dim(X)$. 
\end{prop}

Returning to the case of $\R$-actions, we have:

\begin{prop}\label{prop:dim-quotient-space-no-fixed-points}
	Let $Y$ be a locally compact second countable Hausdorff space and let 
	$\calpha: \R 
	\curvearrowright Y$ be a continuous action with bounded periods but no 
	fixed points. Then $\dim (Y / \mathbb{R}) \leq  \dim(Y) $.
\end{prop}

\begin{proof}
	By \ref{prop:locdim-dim}, it suffices to show $\operatorname{locdim} (Y / \mathbb{R}) \leq  \dim(Y) $. Now given any $x_0 \in Y / \mathbb{R}$, we write $x_0 = \pi(y_0)$ for some $y_0 \in Y$. Since $\calpha$ has no fixed points, by Lemma~\ref{lem:existence-tubes}, there is a tube $B$ with $y_0 \in S_B \cap B^o$. Observe that $\pi(B^o)$ is an open neighborhood of $x_0$ since $\pi^{-1}(\pi(B^o)) = \mathbb{R} \cdot B^o$ is open, and thus the compact set $\pi(B)$ is a closed neighborhood of $x_0$. Moreover, we have $\pi(S_B) = \pi(B)$ by the definition of tubes. 
	
	Let $l_B$ be the length of the tube $B$ and let $R$ be such that $Y = Y_{\leq R}$. We claim that for any $x \in Y / \mathbb{R}$, we have
	\begin{equation}\label{eq:bounding-preimage}
	| \pi^{-1}(x) \cap S_B | < \frac{R}{l_B} + 1 \; .
	\end{equation}
	Indeed, for any $y \in \pi^{-1}(x) \cap S_B$, we know that $\pi^{-1}(x) = 
	\left\{ \calpha_t (y) \middlebar t \in [0,R] \right\}$. Hence if we define 
	$D = \left\{ t \in [0, R] \middlebar  \calpha_t (y) \in S_B \right\}$, then 
	the map $\calpha_{-} (y) \colon t \in D \mapsto \calpha_t (y) \in 
	\pi^{-1}(x) \cap S_B$ is surjective. Now for any $t \in D$, by 
	Lemma~\ref{lem:tubes-basics}, we know that the map $\calpha_{-} (y) \colon 
	s \in \left[ t - \frac{l_B}{2}, t + \frac{l_B}{2} \right] \mapsto \calpha_t 
	(y) \in \pi^{-1}(x)$ is injective, which implies by translation that the 
	map $\calpha_{-} (y) \colon s \in \left[ t , t + l_B \right] \mapsto 
	\calpha_t (y) \in \pi^{-1}(x)$ is injective, and hence $\left( t , t + l_B 
	\right] \cap D = \varnothing$. It follows that $D$ is discrete in $[0, R]$ 
	and is of the form $\left\{d_0, d_1, \ldots, d_n\right\}$ with $d_0 = 0$ 
	and $d_i > d_{i - 1} + l_B$ for $i = 1, \ldots, n$. Hence $n l_B < R$, 
	which implies 
	\[
	| \pi^{-1}(x) \cap S_B | \leq |D| < \frac{R}{l_B} + 1 \; .
	\]
	To finish the proof, we observe that since both $S_B$ and $\pi(S_B)$ are compact Hausdorff, the topology of the latter coincides with the quotient topology induced by the continuous surjection $\pi|_{S_B} \colon S_B \to \pi(S_B)$. Furthermore, it follows from \eqref{eq:bounding-preimage} that $\pi|_{S_B}$ is a finite-to-one map. Hence by Theorem~\ref{thm:Pears75} and Proposition~\ref{prop:Pears75}, we have $\dim(\pi(S_B)) = \dim(S_B) \leq \dim(Y)$. To summarize, we have found a closed neighborhood $\dim(\pi(S_B))$ of $x_0$ of dimension at most $\dim(Y)$. Since $x_0$ is arbitrary, we have $\operatorname{locdim} (Y / \mathbb{R}) \leq  \dim(Y) $.
\end{proof}

\begin{prop}\label{prop:dim-quotient-space}
	Let $Y$ be a locally compact second countable Hausdorff space and let 
	$\calpha: \R 
	\curvearrowright Y$ be a continuous action with bounded periods. Then $\dim 
	(Y / \mathbb{R}) \leq \dim(Y) $.
\end{prop}

\begin{proof}
	Observe that the fixed-point set $Y^{\R}$ is a closed subset of $Y$. Thus $\dim(Y^{\R} / \R) = \dim(Y^{\R}) \leq \dim(Y)$. By Proposition~\ref{prop:dim-quotient-space-no-fixed-points}, we have $\dim((Y \setminus Y^{\R}) / \R) \leq \dim(Y)$. Since $Y \setminus Y^{\R}$ is an open subset of $Y$, its image $(Y \setminus Y^{\R}) / \R$ is an open subset of $Y / \R$. Now the desired inequality follows by applying \cite[Chapter 3, Corollary~5.8]{Pears75} to the decomposition $Y / \R = Y^{\R} / \R \sqcup (Y \setminus Y^{\R}) / \R$. 
\end{proof}

Recall the following theorem: 

\begin{thm}\label{thm:dimnuc-uniformly-compact-orbits} \cite[Theorem 3.4]{Hirshberg-Wu16}
	Let $\calpha: G \curvearrowright Y$ be a continuous action with uniformly compact orbits. Then 
	\[
	\dimnuc^{+1} ( C_0(Y) \rtimes G ) \leq \dim^{+1} (Y / G) \cdot \sup_{y \in Y} \dimnuc^{+1} ( C^*(G_y) ) 
	\]
	and 
	\[
	\dr^{+1} ( C_0(Y) \rtimes G ) \leq \dim^{+1} (Y / G) \cdot \sup_{y \in Y} \dr^{+1} ( C^*(G_y) ) \; . 
	\]
\end{thm}

\begin{cor}\label{cor:dimnuc-bounded-periods}
	Let $Y$ be a locally compact second countable Hausdorff space and let 
	$\calpha: \R 
	\curvearrowright Y$ be a continuous action with bounded periods. Then 
	\[
	\dr^{+1} ( C_0(Y) \rtimes \R ) \leq 2 \dim^{+1}(Y) \; . 
	\]
\end{cor}

\begin{proof}
	Since for any $y \in Y$, the stabilizer $\R_y$, being a closed subgroup of $\R$, must be either $0$, $\R$, or $p\Z$ for some $p>0$, thus $\dr ( C^*(\R_y) )$ is either $0$ or $1$. The result then follows from Theorem~\ref{thm:dimnuc-uniformly-compact-orbits} and Proposition~\ref{prop:dim-quotient-space}. 
\end{proof}

\section{Clopen subgroupoids and conditional expectations}

In this section, we digress to prove a general result about the reduced 
groupoid $C^*$-algebra of a locally compact locally Hausdorff groupoid $G$: a 
clopen subgroupoid $H$ of $G$ induces an inclusion $C^*_r(H) \subset C^*_r(G)$ 
which admits a natural conditional expectation.\footnote{We recall that a 
conditional expectation is a contractive, completely positive map $e \colon B 
\to A$, where $A$ is a $C^*$-subalgebra of $B$, such that 
$	e(abc) = a \cdot e(b) \cdot c $
and 
$e(a) = a $
for any $b \in A$ and $b, c \in B$. See \cite{tomiyama} or \cite[II.6.10]{blackadar-operator-algebras} for other equivalent definitions.} 
Such conditional expectations will play a role in the nuclear dimension estimate in our main result. 

The following lemma is straightforward from Definition~\ref{def:locally-compact-groupoids}. 

\begin{lem}\label{lem:open-subgroupoid}
	Let $G$ be a locally compact locally Hausdorff groupoid with a left Haar 
	system $\lambda = \left\{\lambda^u \right\}_{u \in G^0}$. Let $H$ be an 
	open subset of $G$ which forms a subgroupoid (i.e., $H^0 = H \cap G^0$ and 
	the functions $d$ and $r$ as well as the multiplication are inherited from 
	$G$). Then $H$ is also a locally compact locally Hausdorff groupoid, the 
	linear space $C_c(H)$ embeds into $C_c(G)$ by extending functions trivially 
	from $H$ to $G$, and $\lambda$ restricts to a left Haar system $\lambda_H$ 
	on $H$ via the embedding $C_c(H) \subset C_c(G)$, such that this inclusion 
	extends to an inclusion $C^*_r(H, \lambda_H) \subset C^*_r(G, \lambda)$ of 
	$C^*$-algebras. \qed
\end{lem}

Our main result of this section is the following:

\begin{thm}\label{thm:clopen-subgroupoid}
	Let $G$ be a locally compact locally Hausdorff groupoid with a left Haar 
	system $\lambda = \left\{\lambda^u \right\}_{u \in G^0}$. Let $H$ be a 
	\emph{clopen} subset of $G$ which forms a subgroupoid and let $\lambda_H$ 
	be the restriction of $\lambda$ to $H$ as in 
	Lemma~\ref{lem:open-subgroupoid}. Then the natural inclusion $C^*_r(H, 
	\lambda_H) \subset C^*_r(G, \lambda)$ admits a conditional expectation $e 
	\colon C^*_r(G, \lambda) \to C^*_r(H, \lambda_H)$.
\end{thm}

\begin{proof}
	For any open Hausdorff subset $U$ of $G$, the subset $H \cap U$ is clopen in $U$, and thus for any compactly supported continuous function $f$ on $U$, the pointwise product $f \cdot \chi_{H \cap U}$, where $\chi_{H \cap U}$ is the characteristic function of $H \cap U$ in $U$, is also continuous and compactly supported in $U$. This allows us to define a map \[\rho_H \colon C_c(G) \to C_c(H) \subset C_c(G)\] given by $f \mapsto f \cdot \chi_H$, where $\chi_H$ is the characteristic function of $H$ in $G$ and the multiplication ``$\cdot$'' is pointwise. Similarly define \[\rho_{G \setminus H} \colon C_c(G) \to C_c(G \setminus H) \subset C_c(G)\] by $f \mapsto f \cdot \chi_{G \setminus H}$, where $C_c(G \setminus H)$ is defined in the same way as $C_c(G)$ and $C_c(H)$. It is clear that 
	\begin{enumerate}
		\item both $\rho_H$ and $\rho_{G \setminus H}$ are idempotents, that is, $\rho_H \circ \rho_H = \rho_H$ and $\rho_{G \setminus H} \circ \rho_{G \setminus H} = \rho_{G \setminus H}$,
		\item the image of $\rho_H$ is $C_c(H)$ and that of $\rho_{G \setminus H}$ is $C_c(G \setminus H)$, and
		\item $\rho_H + \rho_{G \setminus H} = \operatorname{id}_{C_c(G)}$. 
	\end{enumerate}
	Note that the set 
	\[
		H \cdot (G \setminus H) \cdot H = \{ x_1 \cdot y \cdot x_2 \mid x_1, x_2 \in H ,\, y \in G \setminus H, \, d(x_1) = r(y), \, d(y) = r(x_2) \}
	\]
	has empty intersection with $H$. Thus we have 
	\[
		C_c(H) \ast C_c(G \setminus H) \ast C_c(H) \subset C_c(G \setminus H) = \ker(\rho_H) \; ,
	\]
	 It follows that for any $g \in C_c(G)$ and $f_1, f_2 \in C_c(H)$, 
	\begin{equation}\label{eq:conditional-expectation-proof}
		\rho_H(f_1 g f_2) = \rho_H \big(f_1 \rho_H(g) f_2 \big) + \rho_H \big(f_1 \rho_{G \setminus H}(g) f_2 \big) = f_1 \rho_H(g) f_2 \; . 
	\end{equation}

	Following Definition~\ref{def:reduced-groupoid-algebra}, for any $v \in G^0$, we write $\operatorname{Ind}^G v \colon C_c(G) \to B(L^2(G_v, \lambda_v))$ for the left regular representation at $v$. Similarly, if $v \in H^0$, we have the left regular representation $\operatorname{Ind}^H v \colon C_c(H) \to B(L^2(H_v, (\lambda_H)_v))$. Since for any $v \in H^0$, the measure $(\lambda_H)_v$ is restricted from $\lambda_v$, we have an isometry
	\[
		V_v \colon L^2(H_v, (\lambda_H)_v) \to L^2(G_v, \lambda_v)
	\]
	given by trivially extending $L^2$-functions on $H_v$ to $G_v$, which induces, via compression, the completely positive contraction 
	\[
		\varphi_v \colon B(L^2(G_v, \lambda_v)) \to B(L^2(H_v, (\lambda_H)_v)), \; T \mapsto V_v^* T V_v \; .
	\]
	For any $v \in H^0$, we have the following:
	\begin{enumerate}
		\item For any $f \in C_c(H)$, we have $\varphi_v ( \operatorname{Ind}^G v (f) ) = \operatorname{Ind}^H v (f)$, because for any $\xi, \eta \in L^2(H_v, (\lambda_H)_v)$, we can calculate the inner product
		\begin{align*}
			& \left\langle \varphi_v ( \operatorname{Ind}^G v (f)) \cdot \xi \;,\; \eta \right\rangle_{L^2(H_v, (\lambda_H)_v)}  \\
			= & \left\langle V_v^* \cdot \operatorname{Ind}^G v (f) \cdot V_v \cdot \xi \;,\; \eta \right\rangle_{L^2(H_v, (\lambda_H)_v)}  \\
			= & \left\langle \operatorname{Ind}^G v (f) \cdot  (V_v \xi) \;,\; V_v \eta \right\rangle_{L^2(G_v, \lambda_v)}  \\
			= & \int_{x \in G_v} \int_{t \in G_v} f(xt^{-1}) \cdot (V_v \xi)(t) \cdot \overline{(V_v \eta)(x)} \, d\lambda_v(t) \, d\lambda_v(x) \\
			= & \int_{x \in H_v} \int_{t \in H_v} f(xt^{-1}) \cdot \xi(t) \cdot \overline{ \eta(x)} \, d\lambda_v(t) \, d\lambda_v(x) \\
			= & \left\langle \operatorname{Ind}^H v (f) \cdot \xi \;,\; \eta \right\rangle_{L^2(H_v, (\lambda_H)_v)} \; .
		\end{align*}
		\item For any $f \in C_c(G \setminus H)$, we have $\varphi_v ( \operatorname{Ind}^G v (f) ) = 0$, because for any $\xi, \eta \in L^2(H_v, (\lambda_H)_v)$, we similarly calculate the inner product
		\begin{align*}
			& \left\langle \varphi_v ( \operatorname{Ind}^G v (f)) \cdot \xi \;,\; \eta \right\rangle_{L^2(H_v, (\lambda_H)_v)}  \\
			= & \int_{x \in H_v} \int_{t \in H_v} f(xt^{-1}) \cdot \xi(t) \cdot \overline{ \eta(x)} \, d\lambda_v(t) \, d\lambda_v(x) \\
			= & 0 \; 
		\end{align*}
		as $xt^{-1} \in H$ for any $x, t \in H_v$ but $f(H) = \{0\}$. 
	\end{enumerate}
	Hence for any $f \in C_c(G)$, we have
	\[
		\varphi_v \left( \operatorname{Ind}^G v (f) \right) = \varphi_v \left( \operatorname{Ind}^G v \left(\rho_H(f) + \rho_{G \setminus H}(f) \right) \right) = \operatorname{Ind}^H v \left(\rho_H(f)\right) + 0 \; ,
	\]
	and thus we have shown $\varphi_v \circ \operatorname{Ind}^G v = \operatorname{Ind}^H v \circ \rho_H$ for any $v \in H^0$. Writing $\varphi_v = 0$ for any $v \in G \setminus H$, we define a completely positive contraction 
	\[
		\varphi = \bigoplus_{v \in G^0} \varphi_v \colon  B\left(\bigoplus_{v \in G^0} L^2(G_v, \lambda_v)\right) \to B\left(\bigoplus_{v \in H^0} L^2(H_v, (\lambda_H)_v)\right) \; .
	\]
	It follows that 
	\[
		\varphi \circ \left(\bigoplus_{v \in G^0} \operatorname{Ind}^G v\right) = \left(\bigoplus_{v \in H^0} \operatorname{Ind}^H v\right) \circ \rho_H \; .
	\]
	Since by Definition~\ref{def:reduced-groupoid-algebra}, $C^*_r(G, \lambda)$ (respectively, $C^*_r(H, \lambda_H)$) is the norm completion of the image of $\left(\bigoplus_{v \in G^0} \operatorname{Ind}^G v\right)$ (respectively, $\left(\bigoplus_{v \in H^0} \operatorname{Ind}^H v\right)$), we see that $\varphi$ restricts to a completely positive contraction from $C^*_r(G, \lambda)$ to $C^*_r(H, \lambda_H)$ that extends the map $\rho_H \colon C_c(G) \to C_c(H)$. Since $\rho_H|_{C_c(H)} = \operatorname{id}_{C_c(H)}$ and $C_c(H)$ is dense in $ C^*_r(H, \lambda_H)$, we have $\varphi(a) = a$ for any $a \in C^*_r(H, \lambda_H)$. Similarly, \eqref{eq:conditional-expectation-proof} implies $\varphi(abc) = a \varphi(b) c$ for any $b \in C^*_r(G, \lambda)$ and $a, c \in C^*_r(H, \lambda_H)$.
\end{proof}

\begin{rmk}
	By a theorem of Tomiyama (\cite{tomiyama}), every contractive projection from a $C^*$-algebra to a $C^*$-subalgebra is a conditional expectation. We could have used it to shorten the proof of Theorem~\ref{thm:clopen-subgroupoid}; however, we choose to present a more down-to-earth proof for the sake of completeness. 
\end{rmk}

\begin{eg}
	If $G$ is an \'{e}tale groupoid and $H$ is taken to be its unit space $G^0$, which is a clopen subgroupoid with trivial operations, then our construction recovers the standard conditional expectation from $C^*_r(G, \lambda)$ onto its Cartan subalgebra $C_0(G^0)$. 
\end{eg}

\begin{prop}\label{prop:clopen-subgroupoid-hereditary}
	Let $G$, $H$ and $e \colon C^*_r(G, \lambda) \to C^*_r(H, \lambda_H)$ be as in Theorem~\ref{thm:clopen-subgroupoid}. Then under the canonical embeddings
	\[
		C_0(H^0) \hookrightarrow C_0(G^0) \hookrightarrow M(C^*(G, \lambda)) 
	\]
	as in \eqref{eq:canonical-embedding-unit-space}, for any any $b \in C^*_r(G, \lambda)$ and for any $a,c \in C_0(H^0)$, we have
	\[
		e(abc) = a \cdot e(b) \cdot c
	\]
	 In particular, for any $a \in C_0(H^0)$, the map $e$ restricts to a conditional expectation
	\[
		e|_{a C^*_r(G, \lambda) a^*} \colon a C^*_r(G, \lambda) a^* \to a C^*_r(H, \lambda_H) a^* \; .
	\] 
\end{prop}

\begin{proof}
	Since $C_c(G)$ is dense in $C^*_r(G, \lambda)$, it suffices to show $e(abc) = a \cdot e(b) \cdot c$ for any any $b \in C_c(G)$ and any $a,c \in C_0(H^0)$. To this end, we notice that the set $H^0 \cdot (G \setminus H) \cdot H^0$ given by
	\[
	\{ x_1 \cdot y \cdot x_2 \mid x_1, x_2 \in H^0 ,\, y \in G \setminus H, \, d(x_1) = r(y), \, d(y) = r(x_2) \}
	\]
	has empty intersection with $H$. By the way the embedding $C_0(G^0) \hookrightarrow M(C^*(G, \lambda)) $ is defined, we have 
	\[
		C_0(H^0) \ast C_c(G \setminus H) \ast C_0(H^0) \subset C_c(G \setminus H) \; ,
	\]
	where $\ast$ denotes the convolution product. Let $\rho_H \colon C_c(G) \to C_c(H)$ and $\rho_{G \setminus H} \colon C_c(G) \to C_c(G \setminus H)$ be defined as in the proof of Theorem~\ref{thm:clopen-subgroupoid}. Then for any $b \in C_c(G)$ and $a, c \in C_0(H^0)$, we have
	\begin{equation*}
	\rho_H(a b c) = \rho_H \big(a \rho_H(b) c \big) + \rho_H \big(a \rho_{G \setminus H}(b) c \big) = a \rho_H(b) c \; , 
	\end{equation*}
	because $\rho_H + \rho_{G \setminus H} = \operatorname{id}_{C_c(G)}$ and $\operatorname{ker} \rho_{H} = C_c(G \setminus H)$. 
\end{proof}

\begin{rmk}
	A standard approximation argument shows that
	\begin{equation*} 
		e(abc) = a \cdot e(b) \cdot c
	\end{equation*}
	for any $b \in C^*_r(G)$ and any $a$ and $c$ in the closure of $C^*_r(H)$ in the multiplier algebra $M(C^*_r(G))$ in the strict topology. 
\end{rmk}

\section{Subgroupoids associated to a tube}

Recall from Example~\ref{example:groupoids} that a topological flow $(Y, 
{\mathbb{R}^{}}, \calpha)$ gives rise to the transformation groupoid $Y 
\rtimes_{\calpha} \mathbb{R}$. We often drop the subscript and simply write $Y 
\rtimes \mathbb{R}$ when there is no confusion. In this section, we associate 
two open subgroupoids $\mathcal{G}_B$ and $\mathcal{H}_B$ to any tube $B$ in 
the space $Y$. These subgroupoids will play important roles in the proof of the 
main theorem in the next section. 

\begin{defn}\label{def:tube-groupoids}
	Let $(Y, {\mathbb{R}^{}}, \calpha)$ be a topological flow and let $B$ be a tube in $Y$ with length $l_B$. Let
	\[
		a_+, a_- \colon B \to \mathbb{R}
	\]
	be as in Definition~\ref{defn:tubes}, which are continuous functions by Lemma~\ref{lem:tubes-basics}. Let $S_{B^o}$ be the open central slice of $B$ as in Definition~\ref{defn:tubes}. Let $Y \rtimes \mathbb{R}$ be the transformation groupoid as in Example~\ref{example:groupoids}: as a set, we have $Y \rtimes \mathbb{R} = Y \times \mathbb{R}$. When there is no confusion, we also identify the unit space $Y \times \{0\}$ with $Y$. We define subsets of $Y \rtimes \mathbb{R}$ as follows:
	\[
		\mathcal{G}_B = \{ x \in Y \rtimes \mathbb{R} \mid d(x) \text{ and } r(x) \in B^o \} = \{ (y, t) \in Y \times \mathbb{R} \mid  \calpha_{-t}(y) \text{ and } y \in B^o \}
	\]
	and
	\[
		\mathcal{H}_B = \left\{ (y, t) \in Y \times \mathbb{R} \middlebar  y \in B^o \text{ and } - t \in \big(a_-(y), a_+(y)\big) \right\} \; . 
	\]
	We also define an auxiliary groupoid 
	\[
		\mathcal{K}_B = S_{B^o} \times \left( \left( - \frac{l_B}{2} , \frac{l_B}{2} \right) \times \left( - \frac{l_B}{2} , \frac{l_B}{2} \right) \right) \;,
	\]
	i.e., the product of the open central slice $S_{B^o}$, as a space, with the pair groupoid $\left( \left( - \frac{l_B}{2} , \frac{l_B}{2} \right) \times \left( - \frac{l_B}{2} , \frac{l_B}{2} \right) \right)$. We endow it with the product topology. The unit space $\left(\mathcal{K}_B\right)^0$ is given by $\{ (y, t, s) \in \mathcal{K}_B \mid t = s \}$. We have $d(y,t,s) = (y,s,s)$, $r(y,t,s) = (y,t,t)$, and $(y,t,s) \cdot (y, s, s') = (y, t, s')$ for any $y \in S_{B^o}$ and $t,s,s' \in \left( - \frac{l_B}{2} , \frac{l_B}{2} \right)$. The Haar system is given by Lebesgue measure on $\left( - \frac{l_B}{2} , \frac{l_B}{2} \right)$, which is identified with each $\left( \mathcal{K}_B \right)^u$. 
\end{defn}

\begin{lem}\label{lem:groupoids-G-H}
	Let $(Y, {\mathbb{R}^{}}, \calpha)$ be a topological flow and let $B$ be a tube in $Y$. We have
	\begin{enumerate}
		\item $ \mathcal{G}_B $ is an open subgroupoid of $Y \rtimes \mathbb{R}$, and
		\item $ \mathcal{H}_B $ is a clopen subset of $\mathcal{G}_B$. 
	\end{enumerate}
\end{lem}

\begin{proof}
	\begin{enumerate}
		\item This is clear because $\mathcal{G}_B$ is the reduction of $Y \rtimes \mathbb{R}$ to the open subset $B^o$ of the unit space $Y$ (see Example~\ref{example:groupoids}).  
		In particular, it is open since it can be written as $d^{-1}(B^o) \cap r^{-1}(B^o)$, the intersection of two open subsets. 
		\item We first note that $\mathcal{H}_B \subset \mathcal{G}_B$ because for any $(y,t) \in \mathcal{H}_B$, we have $y \in B^o$ and $\calpha_{-t}(y) \in B^o$ by Definition~\ref{defn:tubes}. 
		
		To see $\mathcal{H}_B$ is clopen in $\mathcal{G}_B$, we first note that $\mathcal{H}_B$ is open because it is the preimage of the open subset $\{ (t_1, t_2, t_3) \in \mathbb{R}^3 \colon t_1 < t_2 < t_3 \} \subset \mathbb{R}^3$ under the continuous map $(y, t) \in \mathcal{G}_B \to \left(a_-(y) , - t , a_+(y)\right) \in \mathbb{R}^3$. 
		
		On the other hand, by Definition~\ref{defn:tubes} and the continuity of $\calpha$, for any $y \in B^o$, we have $\calpha_{a_-(y)} (y) \notin B^o$, that is, $r\left( y, - a_-(y) \right) \notin B^o$, and thus $\left( y, - a_-(y) \right) \notin \mathcal{G}_B$. Similarly, $\left( y, - a_+(y) \right) \notin \mathcal{G}_B$. Hence the image of the (same) continuous map $(y, t) \in \mathcal{G}_B \to \left(a_-(y) , - t , a_+(y)\right) \in \mathbb{R}^3$ does not intersect the subset $\{ (t_1, t_2, t_3) \in \mathbb{R}^3 \colon t_1 = t_2 \text{~or~} t_2 = t_3 \}$. It follows that $\mathcal{H}_B$ is also the preimage of the closed subset $\{ (t_1, t_2, t_3) \in \mathbb{R}^3 \colon t_1 \leq t_2 \leq t_3 \}$ under this continuous map, which shows it is also closed.
	\end{enumerate}
\end{proof}

In the following, we recall from Example~\ref{example:groupoids}\eqref{item:example:groupoids:transformation} that $Y \rtimes \mathbb{R}$ carries a canonical left Haar system  arising from the canonical Haar measure on $\mathbb{R}$, i.e., the Lebesgue measure. 

\begin{lem}\label{lem:groupoids-H-K}
	Let $(Y, {\mathbb{R}^{}}, \calpha)$ be a topological flow and let $B$ be a tube in $Y$ with length $l_B$. Define a map
	\[
		\widetilde{\tau}_B \colon \mathcal{K}_B \to Y \rtimes \mathbb{R} , \quad (y, t, s) \mapsto \left( \calpha_t(y) , t - s \right)
	\]
	Then we have
	\begin{enumerate}
		\item $ \widetilde{\tau}_B $ is a groupoid homomorphism (i.e., it intertwines the unit space, the maps $d$ and $r$, and the multiplication), 
		\item $ \widetilde{\tau}_B(\mathcal{K}_B) = \mathcal{H}_B$, and $ 
		\widetilde{\tau}_B $ is a homeomorphism onto its image (and thus 
		$\mathcal{H}_B$ is a locally compact Hausdorff groupoid which is 
		isomorphic to $\mathcal{K}_B$ as a topological groupoid), and
		\item $ \widetilde{\tau}_B $ intertwines $\lambda_{\mathcal{K}}$, the left Haar system on $\mathcal{K}_B$, and $\lambda_{\mathcal{H}}$, the restriction of the canonical left Haar system on $Y \rtimes \mathbb{R}$ to $\mathcal{H}_B$. 
	\end{enumerate}
\end{lem}

\begin{proof}
	\begin{enumerate}
		\item For any $(y, t, t) \in \left(\mathcal{K}_B\right)^0$, we have $\widetilde{\tau}_B (y, t, t) = \left( \calpha_{t} (y), 0 \right)$.  Moreover, for any $y \in S_{B^o}$ and $t,s,s' \in \left( - \frac{l_B}{2} , \frac{l_B}{2} \right)$, we use Example~\ref{example:groupoids} and Definition~\ref{def:tube-groupoids} to check that 
		\[
			d \left(\widetilde{\tau}_B (y,t,s)\right) = d \left( \calpha_{t} (y), t - s \right) = \left( \calpha_{s} (y), 0 \right) = \widetilde{\tau}_B (y,s,s) = \widetilde{\tau}_B \left(	d(y,t,s)\right) \; ,
		\]
		\[
			r \left(\widetilde{\tau}_B (y,t,s)\right) = r \left( \calpha_{t} (y), t - s \right) = \left( \calpha_{t} (y), 0 \right) = \widetilde{\tau}_B (y,t,t) = \widetilde{\tau}_B \left(	r(y,t,s)\right) \; ,
		\]
		and 
		\begin{align*}
			\left(\widetilde{\tau}_B (y,t,s)\right) \cdot \left(\widetilde{\tau}_B (y,s,s')\right) & = \left( \calpha_{t} (y), t - s \right) \cdot \left( \calpha_{s} (y), s - s' \right) \\
			&= \left( \calpha_{t} (y), t - s' \right) \\
			&= \widetilde{\tau}_B (y,t,s') \; .
		\end{align*}
		Together these show that $ \widetilde{\tau}_B $ is a groupoid homomorphism. 
		\item We first show $ \widetilde{\tau}_B \left(\mathcal{K}_B\right) \subset \mathcal{H}_B$. For any $(y,t,s) \in \mathcal{K}_B$, where $y \in S_{B^o}$ and $t, s \in \left( - \frac{l_B}{2}, \frac{l_B}{2} \right)$, it follows from Lemma~\ref{lem:tubes-basics} that $\calpha_{t}(y) \in B^o$, and 
		\[
			a_-(\calpha_{t}(y)) = - \frac{l_B}{2} - t \text{~and~} a_+(\calpha_{t}(y)) = \frac{l_B}{2} - t \; .
		\]
		Hence we have 
		\[
			a_-(\calpha_{t}(y)) = - \frac{l_B}{2} - t < s-t < \frac{l_B}{2} - t = a_+(\calpha_{t}(y)) \; .
		\]
		Thus by Definition~\ref{def:tube-groupoids}, we have $\widetilde{\tau}_B (y,t,s) \in \mathcal{H}_B$. This shows $ \widetilde{\tau}_B \left(\mathcal{K}_B\right) \subset \mathcal{H}_B$. 
		
		To prove the other direction, recall from Lemma~\ref{lem:tubes-basics} that there are a pair of homeomorphisms that implement a local trivialization for $B$:
		\[
		\tau_B \colon  S_B \times \left[ - \frac{l_B}{2}, \frac{l_B}{2} \right] \to B \; , \quad (y, t) \mapsto \calpha_t(y)
		\]
		and
		\[
		\theta_B \colon B \to S_B \times \left[ - \frac{l_B}{2}, \frac{l_B}{2} \right] \; , \quad y \to \left( \calpha_{\frac{a_-(y) + a_+(y)}{2}} (y), \frac{l_B}{2} - a_+(y) \right) \; .
		\]
		 We write  $\theta_{B,1} \colon B \to S_B$ and $\theta_{B,2} \colon B \to \left[ - \frac{l_B}{2}, \frac{l_B}{2} \right]$ for the two components of $\theta_B$ and define a map
		\[
			\widetilde{\theta}_B \colon \mathcal{H}_B \to \mathcal{K}_B \, , \quad (y, t) \mapsto \left( \theta_{B,1}(y), \theta_{B,2}(y), \theta_{B,2}(y) - t \right) \, ,
		\]
		which is well defined since $\theta_{B,1} \left(B^o\right) \subset S_{B^o}$, $\theta_{B,2} \left(B^o\right) \subset \left( - \frac{l_B}{2}, \frac{l_B}{2} \right)$, and $\theta_{B,2} (y) - t \in \left( - \frac{l_B}{2}, \frac{l_B}{2} \right)$ because 
		\[
			- t \in \left(a_-(y) , a_+(y) \right) = \left( - \frac{l_B}{2} - \theta_{B,2} (y) , \frac{l_B}{2} - \theta_{B,2} (y) \right) \; .
		\]
		Then we check that for any $(y,t) \in \mathcal{H}_B$, we have
		\begin{align*}
			\widetilde{\tau}_B \left( \widetilde{\theta}_B (y,t) \right) 		&= \widetilde{\tau}_B \left( \theta_{B,1}(y), \theta_{B,2}(y), \theta_{B,2}(y) - t \right) \\
			&= \left( \tau \left( \theta (y) \right) , \theta_{B,2}(y) - \left( \theta_{B,2}(y) - t \right) \right) \\
			&= (y,t) \; .
		\end{align*}
		This shows that $\widetilde{\tau}_B \cdot \widetilde{\theta}_B = \operatorname{id}_{\mathcal{H}_B}$ and thus $\widetilde{\tau}_B\left(\mathcal{K}_B\right) = \mathcal{H}_B$. 
		
		 Now for any $(y,t,s) \in \mathcal{K}_B$, we have
		\begin{align*}
			\widetilde{\theta}_B \left( \widetilde{\tau}_B (y,t,s) \right) 		&= \widetilde{\theta}_B \left( \tau (y,t), t - s \right) \\
			&= \left( \theta_{B,1}(\tau (y,t)), \theta_{B,2}(\tau (y,t)), \theta_{B,2}(\tau (y,t)) - (t-s) \right) \\
			&= (y,t,s) \; .
		\end{align*}
		This shows $\widetilde{\theta}_B$ and $\widetilde{\tau}_B$ are mutual inverses. Both are continuous, by the continuity of $\tau$ and $\theta$. Therefore $ \widetilde{\tau}_B $ is a homeomorphism onto its image. 
		\item For any $u = (y, t, t) \in \left( \mathcal{K}_B \right)^0$, we have $\widetilde{\tau}_B(u) = (\calpha_t (y), 0)$ and $\widetilde{\theta}_B$ restrict to a map
		\[
			\widetilde{\theta}_B |_{ \lambda_{\mathcal{K}}^{u} } \colon \left(\mathcal{K}_B \right)^{u} \to \left(\mathcal{H}_B \right)^{\widetilde{\tau}_B(u)} \, , \quad (y, t, s) \mapsto \left( \calpha_t (y), t - s \right) \; ,
		\]
		which is a linear map in $s$ with coefficient $-1$. Since the left Haar measures $\lambda_{\mathcal{K}}^{u}$ on $\left( \mathcal{K}_B \right)^{u}$ and $\lambda_{\mathcal{H}}^{\widetilde{\tau}_B(u)}$ on $\left( \mathcal{H}_B \right)^{\widetilde{\tau}_B(u)}$ are given, respectively, by the Lebesgue measures on the third coordinate and the second coordinate, we see that the push-forward measure $\left( \widetilde{\theta}_B \middle|_{ \lambda_{\mathcal{K}}^{u} } \right)_* \left( \lambda_{\mathcal{K}}^{u} \right)$ agrees with $\lambda_{\mathcal{H}}^{\widetilde{\tau}_B(u)}$. 
	\end{enumerate}
\end{proof}

Intuitively speaking, if we think of elements in the groupoid $Y \rtimes \mathbb{R}$ as directed paths following the orbits of the action $\calpha$, then $\mathcal{G}_B$ contains all the paths that both start and end in $B^o$, while $\mathcal{H}_B$ contains all the paths that lie entirely inside $B^o$.

 Next, we look at the $C^*$-algebras of the groupoids $\mathcal{G}_B$ and $\mathcal{H}_B$. 
 
 \begin{lem}\label{lem:groupoid-algebra-H}
 	Let $(Y, {\mathbb{R}^{}}, \calpha)$ be a topological flow and let $B$ be a tube in $Y$. Then
 	\begin{enumerate}
 		\item $C^*_r(\mathcal{H}_B) \cong C_0(S_{B^o}) \otimes K$, and
 		\item The embedding $\mathcal{H}_B \hookrightarrow \mathcal{G}_B$ induces an inclusion $C^*_r(\mathcal{H}_B) \subset C^*_r(\mathcal{G}_B)$ that admits a conditional expection $e_B \colon C^*_r(\mathcal{G}_B) \to C^*_r(\mathcal{H}_B)$. 
 	\end{enumerate}
 	
 \end{lem} 
 
 \begin{proof}
 	The first statement follows from Lemma~\ref{lem:groupoids-H-K}, 
 	Example~\ref{example:reduced-groupoid-algebra} and the fact that 
 	$C^*_r(\mathcal{G}_1 \times \mathcal{G}_2) \cong C^*_r(\mathcal{G}_1) 
 	\otimes_{\text{min}} C^*_r(\mathcal{G}_2)$ for locally compact locally 
 	Hausdorff groupoids $\mathcal{G}_1$ and $\mathcal{G}_2$;
 	(see \cite[Proposition~6.11]{moore-schochet}.) 
 	
 	Since $\mathcal{H}_B$ embeds into $\mathcal{G}_B$ as a clopen subgroupoid by Lemma~\ref{lem:groupoids-G-H}, the second statement follows from Theorem~\ref{thm:clopen-subgroupoid}. 
 \end{proof}
 
 Adopting the intuitive picture as above, we see that the effect of the conditional expectation $e$ is that it annihilates all the paths in $\mathcal{G}_B$ that roam outside of $B$. On the other hand, we see that if a path is very short, then having both ends inside $B^o$ implies that its entirety falls inside $B^o$. To make this idea precise, we introduce a notion of \emph{propagation}.

 \begin{defn}\label{def:propagation}
 	For any element $(y, t)$ in the transformation groupoid $Y \rtimes \mathbb{R}$, we define its \emph{propagation} by $\operatorname{prop}(y, t) = |t|$. 
 	
 	 	For any $f \in C_c(Y \rtimes \mathbb{R})$, we define its \emph{propagation} by 
 	 	\[ 
 	 	\operatorname{prop}(f) = \sup\{\operatorname{prop}(x) \mid x \in \supp(f)\} =  \max\{\operatorname{prop}(x) \mid x \in \supp(f)\}. 
 	 	\]
 \end{defn}
 
 \begin{rmk}\label{rmk:def:propagation} 
 	Note that under the standard identification of the convolution algebras $C_c(Y \rtimes \mathbb{R})$ and $C_c(\mathbb{R}, C_c(Y))$ as given in Example~\ref{example:reduced-groupoid-algebra}, the propagation of $f \in C_c(\mathbb{R}, C_c(Y))$ is $\sup\{|t| \mid f(t) \not= 0\}$.
 \end{rmk}
 
 Recall that there is a canonical embedding 
 \begin{equation} \label{eq:canonical-embedding-Y}
 {C}_{0}(Y) \hookrightarrow M({C}_{0}(Y) \rtimes {\mathbb{R}^{}}) \cong M(C^*_r(Y \rtimes \mathbb{R}))
 \end{equation}
 as given in \eqref{eq:canonical-embedding-unit-space}, so that for any $ f \in {C}_{0}(Y) $ and for any $ g $ in the dense subalgebra $ {C}_{c}(Y \rtimes \mathbb{R}) $ in $ C^*_r(Y \rtimes \mathbb{R}) $, the multiplication of $ g $ by $ f $ from the left and right are also in ${C}_{c}(Y \rtimes \mathbb{R})$ and are given by
 \[
 (f \cdot g) \ (y, t) = f(y) \cdot g (y, t) \ \text{and}\  (g \cdot f ) \ (y, t) =  g(y, t) \cdot \alpha_{t} ( f ) (y) = g(y, t) \cdot f \left( \calpha_{-t} (y)\right)
 \]
 for any $(y,t) \in Y \rtimes \mathbb{R}$. 
 
 \begin{lem}\label{lem:groupoid-algebra-compression}
 	Let $(Y, {\mathbb{R}^{}}, \calpha)$ be a topological flow and let $B$ be a tube in $Y$. Let $f \in C_0(Y)$ have support inside $B$. Then for any $g$ in the dense subalgebra $C_c(Y \rtimes \mathbb{R})$, we have $f g f^* \in C_c(\mathcal{G}_B)$ and $\operatorname{prop}(f g f^*) \leq \operatorname{prop}(g)$. Moreover, the hereditary subalgebra $\overline{f C^*_r(Y \rtimes \mathbb{R}) f^*}$ is contained in the subalgebra $C^*_r(\mathcal{G}_B)$ and is in fact equal to $\overline{f C^*_r(\mathcal{G}_B) f^*}$.  
 \end{lem}
 
 \begin{proof} 
 	For any $g \in C_c(Y \rtimes \mathbb{R})$, the support of $f g f^* \in C_c(Y \rtimes \mathbb{R})$ is 
 	\[
	 	\overline{\left\{ (y,t) \in Y \rtimes \mathbb{R} \mid g(y,t) \not= 0 , \, f(y) \not= 0 ,\, f(\calpha_{-t}(y)) \not= 0 \right\}} \subseteq \supp(g)\; ,
 	\]
 	which is also in $\mathcal{G}_B$ because $\{ y \in Y \mid f(y) \not= 0 \}$ is contained in $B^o$. The first containment gives us $\operatorname{prop}(f g f^*) \leq \operatorname{prop}(g)$, while the second tells us $f g f^* \in C_c(\mathcal{G}_B)$. Thus we have $f C^*_r(Y \rtimes \mathbb{R}) f^* \subset C^*_r(\mathcal{G}_B)$ by a density argument. Hence
 	\[
	 	f C^*_r(Y \rtimes \mathbb{R}) f^* = f^2 C^*_r(Y \rtimes \mathbb{R}) (f^2)^* \subset f C^*_r(\mathcal{G}_B) f^* \subset f C^*_r(Y \rtimes \mathbb{R}) f^* \; ,
 	\]
 	which implies $f C^*_r(Y \rtimes \mathbb{R}) f^* = f C^*_r(\mathcal{G}_B) f^*$. 
 \end{proof}
 
 \begin{lem}\label{lem:groupoid-propagation-G-H}
 	Let $(Y, {\mathbb{R}^{}}, \calpha)$ be a topological flow and let $B$ be a tube in $Y$. Let $\varepsilon$ be a positive number as in Definition~\ref{defn:tubes}. Then any $x \in \mathcal{G}_B$ with $\operatorname{prop}(x) \leq \varepsilon$ is contained in $\mathcal{H}_B$. 
 \end{lem}
 
 \begin{proof}
 	Given $(y,t) \in \mathcal{G}_B$ with $\operatorname{prop}(x) \leq \varepsilon$, we have $y \in B^o$, $\calpha_{-t}(y) \in B^o$ and $|t| \leq \varepsilon$. By Definition~\ref{defn:tubes} and by the continuity of $\calpha$, the first two conditions above imply that $-t \notin \left[ a_-(y) - \varepsilon, a_-(y) \right] \cup \left[ a_+(y), a_+(y) + \varepsilon \right]$. Since $a_-(y) < 0 < a_+(y)$ and $-t \in [-\varepsilon, \varepsilon]$, we have $-t \in \left( a_-(y), a_+(y) \right)$. Therefore $(y,t) \in \mathcal{H}_B$.
 \end{proof}
 
 \begin{lem}\label{lem:groupoid-algebra-propagation-G-H}
 	Let $Y$, $\calpha$, $B$ and $\varepsilon$ be as in Lemma~\ref{lem:groupoid-propagation-G-H}. Then any $g \in C_c(\mathcal{G}_B)$ with $\operatorname{prop}(g) \leq \varepsilon$ is contained in the subalgebra $C_c(\mathcal{H}_B)$. In particular, if we let $e_B \colon C^*_r(\mathcal{G}_B) \to C^*_r(\mathcal{H}_B)$ be the conditional expectation as in Lemma~\ref{lem:groupoid-algebra-H},  we have $e_B(g) = g$. 
 \end{lem}
 
 \begin{proof}
 	This follows immediately from Lemma~\ref{lem:groupoid-propagation-G-H}. 
 \end{proof}

\section{The main result}

In this section, we bound the nuclear dimension of the crossed product algebra $ {C}_{0}(Y) \rtimes {\mathbb{R}^{}} $ in terms of the covering dimension of $Y$.

\begin{thm}\label{thm-main}
	Let $(Y, {\mathbb{R}^{}}, \calpha)$ be a topological flow. Assume $Y$ is a 
	locally compact second countable Hausdorff space with finite covering 
	dimension. Then 
	\[
	\operatorname{dim}_\mathrm{nuc}({C}_{0}(Y) \rtimes {\mathbb{R}^{}} ) \leq (\dim(Y)+1)(5\dim(Y)+7) - 1 \; .
	\]
\end{thm}

\begin{proof}
	\setcounter{clminproof}{\value{Thm}} 
	Set $d=\dim (Y)$. We need to show $\dimnuc ( C_0(Y) \rtimes_\alpha \R ) + 1 \leq (d+1)(5d+7)$. Since $ {C}_{c}({\mathbb{R}^{}}, {C}_{c}(Y)) $ is dense in $ {C}_{0}(Y) \rtimes {\mathbb{R}^{}}  $, it suffices to show that for any finite subset $ F \subset {C}_{c}({\mathbb{R}^{}}, {C}_{c}(Y)) \cap ( {C}_{0}(Y) \rtimes {\mathbb{R}^{}})_{\leq 1} $ and $ \varepsilon > 0 $, the condition of Lemma~\ref{Lemma:finite-dimnuc} holds for $F$ and $(7d+10) \varepsilon$. Given such $ F $ and $\varepsilon$, we start by finding $ L > 0 $ such that $ F \subset {C}_{c}( (-L, L), {C}_{c}(Y)) \subset {C}_{c}({\mathbb{R}^{}}, {C}_{c}(Y)) $. 
	Denote by $ \left\| b \right\|_1 = \int_{-\infty}^{\infty}\|b(t)\|dt $, that is, the $L^{1}$-norm of $ b \in F $ as a function over $\mathbb{R}$.
	Set 
	\begin{equation}\label{eq:defn-eta}
	\eta = \frac{\varepsilon ^2 }{ 4 L ^3 \left( \sup_{b \in F} \left\| b \right\|_1 \right) ^2  } \, .
	\end{equation}

	We define $R = R(L, \eta, d) > 0$ as in Proposition~\ref{prop:get-partition-of-unity}. Recall that
	\begin{align*}
	Y_{\leq R} & = \left\{ y \in Y \middlebar \mathrm{per}_{\calpha}(y) \leq R \right\} = \left\{ y \in Y \middlebar \calpha_{\R}(y) = \calpha_{[0,R]} (y) \right\} \; , \\
	Y_{> R} & =  Y \setminus Y_{\leq R} = \left\{ y \in Y \middlebar \mathrm{per}_{\calpha}(y) > R \right\} \; .
	\end{align*}
	
	By Corollary~\ref{cor:long-period-open}, $Y_{> R}$ is an open $\calpha$-invariant subset, and $Y_{\leq R}$ is a closed $\calpha$-invariant subset. Thus we have the following exact sequence 
	\begin{equation}\label{eq:exact-sequence}
	0 \to C_0(Y_{> R}) \rtimes \R \overset{\theta}{\longrightarrow} C_0(Y ) \rtimes \R \overset{\pi}{\longrightarrow} C_0(Y_{\leq R}) \rtimes \R  \to 0 \; .
	\end{equation}
	Let us first focus on the quotient algebra in this exact sequence. Since the restricted action $\R \curvearrowright Y_{\leq R}$ has bounded periods, by Corollary~\ref{cor:dimnuc-bounded-periods}, we have 
	\[
	\dimnuc ( C_0(Y_{\leq R}) \rtimes \R ) \leq \dr ( C_0(Y_{\leq R}) \rtimes \R ) \leq 2 (d + 1) - 1 \; .
	\]
	Hence by definition, we can find a piecewise contractive $2 (d + 1)$-decomposable approximation for $(\pi (F), \varepsilon)$:
	\begin{equation}\label{eq:decomposable-approximation-for-quotient}
	\xymatrix{\displaystyle
		C_0(Y_{\leq R}) \rtimes \R \ar[dr]_{\psi_{\leq R} = \bigoplus_{l=0} ^{2d+1} \psi_{\leq R}^{(l)} \ \ \ } \ar@{.>}[rr]^{\id} &  & C_0(Y_{\leq R}) \rtimes \R \\
		\displaystyle & A_{\leq R} = \bigoplus_{l=0} ^{2d+1} A_{\leq R}^{(l)} \ar[ur]_{\ \ \ \ \varphi_{\leq R} = \sum_{l=0} ^{2d+1} \varphi_{\leq R}^{(l)} } &
	}
	\end{equation}
	By Lemma~\ref{Lemma:lifting-decoposable-maps}, we may lift $\varphi_{\leq 
	R}$ to a piecewise contractive $2 (d + 1)$-decomposable completely positive 
	map 
	\[
	\widetilde{\varphi}_{\leq R} = \sum_{l=0} ^{ 2d+1} \widetilde{\varphi}_{\leq R}^{(l)} : A_{\leq R}  \to C_0(Y) \rtimes \R \; .
	\]
	By \cite[Proposition 1.4]{winter-zacharias}, there exists $\delta > 0$ such that for any positive contraction $e \in C_0(Y) \rtimes \R$, if $\left\| [ (1 - e), \widetilde{\varphi}_{\leq R}^{(l)} (a) ] \right\| \leq \delta \|a\| $ for any $a \in A_{\leq R}$ and for any $l \in \left\{0, \ldots, 2d+1\right\}$, then there are completely positive contractive order zero maps $\widehat{\varphi}_{\leq R}^{(l)} : A_{\leq R} \to C_0(Y ) \rtimes \R $ such that 
	\[
	\left\| \widehat{\varphi}_{\leq R}^{(l)} (a) - (1 - e)^\frac{1}{2} \widetilde{\varphi}_{\leq R}^{(l)} (a) (1 - e)^\frac{1}{2} \right\| \leq \varepsilon \| a \|
	\]
	for all $a \in A_{\leq R}$ and for all $l \in \left\{0, \ldots, 2d+1\right\}$. 
	
	By Lemma~\ref{lem:quasicentral-approximate-unit}, there is a quasicentral approximate unit for 
	\[
		C_0(Y_{> R}) + C_0(Y_{> R}) \rtimes \R \subseteq M(C_0(Y ) \rtimes \R)
	\]
	which is contained in $C_c(Y_{> R})_{+, \leq 1}$. Thus, we may choose an element $e \in C_c(Y_{> R})_{+, \leq 1}$ which satisfies:
	\begin{enumerate}
		\item $\left\| [ (1 - e), \widetilde{\varphi}_{\leq R}^{(l)} (a) ] \right\| \leq \delta \|a\| $ for any $a \in A_{\leq R}$ and for any $l \in \left\{0, \ldots, 2d+1\right\}$;
		\item $\left\| e^\frac{1}{2} \: b \: e^\frac{1}{2} + (1-e)^\frac{1}{2} \: b \: (1-e)^\frac{1}{2} - b \right\| \leq \varepsilon $ for any $b \in F$;
		\item $\left\| (1-e)^\frac{1}{2} ( (\widetilde{\varphi}_{\leq R} \circ \psi_{\leq R} \circ \pi) (b) - b ) (1-e)^\frac{1}{2} \right\| \leq 2 \varepsilon $ for any $b \in F$.
	\end{enumerate}
	Therefore, we can find maps $\widehat{\varphi}_{\leq R}^{(l)}$ as described in the previous paragraph, and sum them up to obtain a piecewise contractive $ 2(d+1)$-decomposable completely positive map 
	\[
	\widehat{\varphi}_{\leq R} = \sum_{l=0} ^{ 2d+1} \widehat{\varphi}_{\leq R}^{(l)}  : A_{\leq R} \to C_0(Y ) \rtimes \R
	\]
	such that for any $a \in A_{\leq R}$,
	\[
	\left\| \widehat{\varphi}_{\leq R} (a) - (1 - e)^\frac{1}{2} \widetilde{\varphi}_{\leq R} (a) (1 - e)^\frac{1}{2} \right\| \leq 2(d+1) \varepsilon \| a \| \; .
	\]

	\begin{clminproof}\label{clm:estimate-1:thm:estimate-dimnuc-Z}
		The diagram
		\[
		\xymatrix{\displaystyle
			C_0(Y) \rtimes \R \ar[dr]_{\psi_{\leq R} \circ \pi \ \ } \ar@{.>}[rr]^{\id} &  & C_0(Y) \rtimes \R \\
			\displaystyle & A_{\leq R}  \ar[ur]_{\ \widehat{\varphi}_{\leq R}  } &
		}
		\]
		commutes on $(1 - e)^\frac{1}{2} F (1 - e)^\frac{1}{2}$ up to errors bounded by $(2d+4)  \varepsilon$.
	\end{clminproof}
	\begin{proof}
		Observe that for any $b \in F$, 
		\begin{align*}
		& \left\| (\widehat{\varphi}_{\leq R} \circ \psi_{\leq R} \circ \pi ) \left( (1 - e)^\frac{1}{2} b (1 - e)^\frac{1}{2} \right) - (1 - e)^\frac{1}{2} b (1 - e)^\frac{1}{2} \right\| & \\
		= & \  \left\| (\widehat{\varphi}_{\leq R} \circ \psi_{\leq R} \circ \pi ) \left( b \right) - (1 - e)^\frac{1}{2} b (1 - e)^\frac{1}{2} \right\| & \Big[\scriptscriptstyle{ \pi(e) = 0 }\Big] \\
		\leq & \ \left\| (\widehat{\varphi}_{\leq R} \circ \psi_{\leq R} \circ \pi ) ( b ) - (1 - e)^\frac{1}{2} \cdot ( \widetilde{\varphi}_{\leq R} \circ \psi_{\leq R} \circ \pi ) ( b ) \cdot (1 - e)^\frac{1}{2} \right\| \\
		& + \left\| (1 - e)^\frac{1}{2} \big( (\widetilde{\varphi}_{\leq R} \circ \psi_{\leq R} \circ \pi ) ( b ) - b \big) (1 - e)^\frac{1}{2} \right\| \\
		\leq & \ 2(d+1) \varepsilon \left\| (\psi_{\leq R} \circ \pi ) ( b ) \right\| + 2 \varepsilon \\
		\leq & \ (2d+4)  \varepsilon .  & \Big[\scriptstyle{\|b\| \leq 1 }\Big] 
		\end{align*}
		This proves the claim.
	\end{proof}

	Our next step is to find an approximation for $ e^\frac{1}{2} F e^\frac{1}{2} $, which concerns the kernel in the exact sequence in~\eqref{eq:exact-sequence}. Since $e$ is compactly supported and $\mathrm{supp}(e) \subset  Y_{> R}$, we may pick a compact subset $K \subset Y_{> R}$ such that $ \mathrm{supp}(e) \subset K^o $. Our choice of $R$ allows us to apply Proposition~\ref{prop:get-partition-of-unity} to obtain a finite partition of unity $\left\{f_i\right\}_{i \in I}$ for the inclusion $K \subset Y_{>R}$ satisfying 
	\begin{enumerate}
		\item \label{prop:get-partition-of-unity-a-re1} for any $ i\in I $, $ f_i $ is $\calpha$-Lipschitz with constant $\eta$;
		\item  \label{prop:get-partition-of-unity-b-re1} for any $ i\in I $, $ \calpha_{[-L, L]} \big( \mathrm{supp}( f_i ) \big) $ is contained in a tube $B_i$ in $Y_{>R}$; 
		\item \label{prop:get-partition-of-unity-c-re1}  there is a decomposition $I = I^{(0)} \cup \cdots \cup I^{(5(d+1)-1)} $ such that for any $ l \in \left\{ 0, \dots, 5(d+1)-1 \right\} $ and any two different $i, j \in I^{(l)} $, we have 
		$$ \calpha_{[-L, L]} \big( \mathrm{supp}( f_i ) \big) \cap \calpha_{[-L, L]} \big( \mathrm{supp}( f_j ) \big) = \varnothing . $$
		By eliminating redundancies, we may assume without loss of generality that $I^{(0)} , \cdots , I^{(5(d+1)-1)}$ are disjoint. 
	\end{enumerate}
	
	Now for each $ i\in I $, let us shrink the tube $B_i$ to 
	\[
	\widehat{B}_i = \left\{ y \in Y_{>R} \middlebar \calpha_{[-L,L]}(y) \subset B_i \right\} \; ,
	\]
	which is a tube with the same central slice but a shorter length $l_{\widehat{B}_i} = l_{{B}_i} - 2 L$, while by the definition of the value $\varepsilon_{{B}_i}$ in Lemma~\ref{lem:tubes-basics}, we have
	\begin{equation}\label{eq:varepsilon_B-hat}
	\varepsilon_{\widehat{B}_i} \geq \varepsilon_{{B}_i} + 2L > 2 L \; .
	\end{equation}
	 Additionally, condition (\ref{prop:get-partition-of-unity-b-re1}) above implies that $ \mathrm{supp}( f_i ) \subset \widehat{B}_i$. These properties of the shrunken tubes will be used in Claim~\ref{clm:estimate-2:thm:estimate-dimnuc-Z}. 
	
	For each tube $\widehat{B}_i$, we define $\mathcal{G}_{\widehat{B}_i}$ and $\mathcal{H}_{\widehat{B}_i}$ as in Definition~\ref{def:tube-groupoids}. Then by Lemma~\ref{lem:groupoids-G-H} and~\ref{lem:groupoids-H-K}, $\mathcal{G}_{\widehat{B}_i}$ is an open subgroupoid of $Y_{>R} \rtimes \mathbb{R}$ (itself an open subgroupoid of $Y \rtimes \mathbb{R}$) and $\mathcal{H}_{\widehat{B}_i}$ is a clopen subgroupoid of $\mathcal{G}_{\widehat{B}_i}$. By Lemma~\ref{lem:open-subgroupoid}, there are embeddings
	\begin{equation}\label{eq:groupoid-algebra-embeddings}
		C^*_r(\mathcal{H}_{\widehat{B}_i}) \subset C^*_r(\mathcal{G}_{\widehat{B}_i}) \subset C^*_r( Y_{>R} \rtimes \mathbb{R} ) \subset C^*_r( Y \rtimes \mathbb{R} ) \; ,
	\end{equation}
	by Lemma~\ref{lem:groupoid-algebra-H}, there is a conditional expectation
	\[
		e_i = e_{\widehat{B}_i} \colon C^*_r(\mathcal{G}_{\widehat{B}_i}) \to C^*_r(\mathcal{H}_{\widehat{B}_i})
	\]
	and, by Lemma~\ref{lem:groupoid-algebra-compression}, there is a compression map
	\[
		\kappa_i \colon C^*_r( Y \rtimes \mathbb{R} ) \to f_i^{\frac{1}{2}} C^*_r(\mathcal{G}_{\widehat{B}_i}) f_i^{\frac{1}{2}} \, , \quad b \mapsto f_i^{\frac{1}{2}} b f_i^{\frac{1}{2}} \; ,
	\]
	which is completely positive and contractive because $\left\| f_i \right\| \leq 1$. 
	
	Since $e \in C_c(Y_{> R})$ and $L$ is chosen such that $ F \subset {C}_{c}( (-L, L), {C}_{c}(Y))$, we have $ e^\frac{1}{2} F e^\frac{1}{2}  \subset {C}_{c}( (-L, L), {C}_{c}(Y_{>R}))$. By Example~\ref{example:reduced-groupoid-algebra} and Remark~\ref{rmk:def:propagation}, under the canonical identification between $C_0(Y_{>R}) \rtimes \mathbb{R}$ and $C^*_r( Y_{>R} \rtimes \mathbb{R} )$, for any $b \in F$, we have $e^\frac{1}{2} b e^\frac{1}{2} \in C_c( Y_{>R} \rtimes \mathbb{R} )$ with propagation $\operatorname{prop}\left( e^\frac{1}{2} b e^\frac{1}{2} \right) < L$. It follows by Lemma~\ref{lem:groupoid-algebra-compression} that for each $i \in I$ and any $b \in F$, we have
	\[
		\operatorname{prop}\left(\kappa_i\left(e^\frac{1}{2} b e^\frac{1}{2}\right)\right) \leq \operatorname{prop}\left( e^\frac{1}{2} b e^\frac{1}{2} \right) < L
	\]
	and thus 
	\begin{equation}\label{eq:e-kappa}
		e_i \left(\kappa_i\left(e^\frac{1}{2} b e^\frac{1}{2}\right)\right) = \kappa_i\left(e^\frac{1}{2} b e^\frac{1}{2}\right)
	\end{equation}
	by Lemma~\ref{lem:groupoid-algebra-propagation-G-H} and \eqref{eq:varepsilon_B-hat}. 
	
	By Proposition~\ref{prop:clopen-subgroupoid-hereditary}, $e_i$ restricts to a conditional expectation
	\[
		e_i \left|_{\overline{f_i^{\frac{1}{2}} C^*_r(\mathcal{G}_{\widehat{B}_i}) f_i^{\frac{1}{2}}}} \right. \colon \overline{f_i^{\frac{1}{2}} C^*_r(\mathcal{G}_{\widehat{B}_i}) f_i^{\frac{1}{2}} }  \to \overline{ f_i^{\frac{1}{2}} C^*_r(\mathcal{H}_{\widehat{B}_i}) f_i^{\frac{1}{2}} } \;,
	\]
	and thus we may define, for each $i \in I$, a completely positive contraction
	\[
		\psi_i = e_i \left|_{\overline{f_i^{\frac{1}{2}} C^*_r(\mathcal{G}_{\widehat{B}_i}) f_i^{\frac{1}{2}}}} \right. \circ \kappa_i \colon C^*_r( Y \rtimes \mathbb{R} ) \to \overline{ f_i^{\frac{1}{2}} C^*_r(\mathcal{H}_{\widehat{B}_i}) f_i^{\frac{1}{2}} }
	\]
	and a $*$-homomorphism 
	\[
		\varphi_i \colon \overline{ f_i^{\frac{1}{2}} C^*_r(\mathcal{H}_{\widehat{B}_i}) f_i^{\frac{1}{2}} } \to C^*_r( Y \rtimes \mathbb{R} )
	\]
	given by the embeddings induced from \eqref{eq:groupoid-algebra-embeddings}. 
	
	For $ l = 0, \dots, 5(d+1)-1 $, we define algebras
	\begin{align*}
		A_{> R}^{(l)} &= \bigoplus_{i\in I^{(l)}} \overline{ f_i^{\frac{1}{2}} C^*_r(\mathcal{H}_{\widehat{B}_i}) f_i^{\frac{1}{2}} } \\
		A_{> R} &= \bigoplus_{l = 0} ^{5(d+1)-1} A_{> R}^{(l)} = \bigoplus_{i\in I} \overline{ f_i^{\frac{1}{2}} C^*_r(\mathcal{H}_{\widehat{B}_i}) f_i^{\frac{1}{2}}}
	\end{align*}
	and maps
	\begin{align*}
		\psi_{> R}^{(l)} &= \bigoplus_{i\in I^{(l)}} \psi_i \colon C^*_r( Y \rtimes \mathbb{R} )  \to A_{> R}^{(l)}  \\
		\varphi_{> R}^{(l)} &= \sum_{i\in I^{(l)}} \varphi_i \colon  A_{> R}^{(l)}  \to C^*_r( Y \rtimes \mathbb{R} ) \; .
	\end{align*}
	Then each $\psi_{> R}^{(l)}  $ is again a completely positive contraction, and since the family $ \left\{  f_i ^{\frac{1}{2}}  \right\}_{i\in I^{(l)}} $ is orthogonal, the collection of subalgebras $\left\{ \overline{ f_i ^{\frac{1}{2}} C^*_r(\mathcal{H}_{\widehat{B}_i}) f_i^{\frac{1}{2}}  } \right\}_{i\in I^{(l)}}$ is also orthogonal, which implies that each $ \varphi_{> R}^{(l)} $ is again a  homomorphism, and in particular, a completely positive order zero contraction. Hence we have the following diagram:
	\begin{displaymath}
	\xymatrix{
		{C}_{0}(Y) \rtimes {\mathbb{R}^{}} \ar@{.>}[rr]^{\mathrm{Id}} \ar[rd]^{\bigoplus_{i\in I} \kappa_i} \ar@/_/[rdd]_{\bigoplus_{l = 0} ^{5d+4} \psi_{> R}^{(l)} = \: \bigoplus_{i\in I} \psi_i} & & {C}_{0}(Y) \rtimes {\mathbb{R}^{}} \\
		& { \bigoplus_{i\in I} \overline{ f_i^{\frac{1}{2}} C^*_r(\mathcal{G}_{\widehat{B}_i}) f_i^{\frac{1}{2}} } }   \ar@/_/[d]_{\bigoplus_{i\in I} e_i}   \ar[ru]^{\sum_{i\in I} \mathrm{Incl}_i} &  \\
		&  A_{> R} = \bigoplus_{i\in I} \overline{ f_i^{\frac{1}{2}} C^*_r(\mathcal{H}_{\widehat{B}_i}) f_i^{\frac{1}{2}} }  \ar@/_/@{^{(}->}[ruu]_{\sum_{l = 0} ^{5d+4} \varphi_{> R}^{(l)} = \: \sum_{i\in I} \varphi_i} \ar@/_/@{^{(}->}[u] &
	}
	\end{displaymath}
	where $\displaystyle \operatorname{Incl}_i \colon \overline{ f_i^{\frac{1}{2}} C^*_r(\mathcal{G}_{\widehat{B}_i}) f_i^{\frac{1}{2}}} \hookrightarrow {C}_{0}(Y) \rtimes {\mathbb{R}^{}}$ denotes the canonical embedding induced from \eqref{eq:groupoid-algebra-embeddings}. 
	
	\begin{clminproof}\label{clm:estimate-2:thm:estimate-dimnuc-Z}
		The large triangle in the above diagram commutes on $e^\frac{1}{2} F e^\frac{1}{2}$ up to errors bounded by $5(d+1) \varepsilon$.
	\end{clminproof}

	\begin{proof}
		First we observe that it suffices to show the upper triangle commutes on $e^\frac{1}{2} F e^\frac{1}{2}$ up to errors bounded by $5(d+1) \varepsilon $. Indeed, for any $ b \in F$, we have
		\begin{equation}\label{eq:large-triangle-upper-triangle}
			\left( \sum_{i \in I} \varphi_i \right) \circ \left( \bigoplus_{i \in I} \psi_i \right) \left( e^\frac{1}{2} b e^\frac{1}{2} \right) 
			=  \sum_{i \in I} e_i \circ \kappa_i \left( e^\frac{1}{2} b e^\frac{1}{2} \right) 
			= \sum_{i \in I} \kappa_i \left( e^\frac{1}{2} b e^\frac{1}{2} \right) 
		\end{equation}
		by \eqref{eq:e-kappa}, and the right-hand side can also be written as 
		\[
			\left( \sum_{i \in I} \operatorname{Incl}_i \right) \circ \left( \bigoplus_{i \in I} \kappa_i \right)  \left( e^\frac{1}{2} b e^\frac{1}{2} \right) \; ,
		\]
		which is the composition of the two lower arrows of the upper triangle. 
		
		Thus we just need to bound the norm of 
		\[
			\left( e^\frac{1}{2} b e^\frac{1}{2} \right) - \sum_{i \in I} \kappa_i \left( e^\frac{1}{2} b e^\frac{1}{2} \right) \; ,
		\]
		which, by the definition of $\kappa_i$, is equal to
		\[
			\left( e^\frac{1}{2} b e^\frac{1}{2} \right) - \sum_{i \in I} f_i^{\frac{1}{2}} \cdot \left( e^\frac{1}{2} b e^\frac{1}{2} \right) \cdot f_i^{\frac{1}{2}}  \; .
		\]
		To do this, we compute, for any $b \in F$, for any $i \in I$ and for $ t \in (- L, L) $, 
		\begin{align*}
		& \left\| \left( f_i^{\frac{1}{2}} \cdot { \left( e^\frac{1}{2} b e^\frac{1}{2} \right) } - { \left( e^\frac{1}{2} b e^\frac{1}{2} \right) } \cdot f_i^{\frac{1}{2}} \right) (t) \right\|_{{C}_{0}(Y)} \\
		= &\ \left\| f_i^{\frac{1}{2}} \cdot { \left( e^\frac{1}{2} b e^\frac{1}{2} \right) } (t) - { \left( e^\frac{1}{2} b e^\frac{1}{2} \right) }  (t) \cdot \alpha_{t} \left(f_i^{\frac{1}{2}} \right) \right\|_{{C}_{0}(Y)} \\
		= &\ \left\| \left( f_i^{\frac{1}{2}} -  \alpha_{t} \left( f_i^{\frac{1}{2}} \right) \right) \cdot  e^\frac{1}{2} \cdot \alpha_{t}(e^\frac{1}{2}) \cdot b(t) \right\|_{{C}_{0}(Y)} \\
		\le &\ \left\| f_i^{\frac{1}{2}} -  \alpha_{t} \left( f_i^{\frac{1}{2}} \right)  \right\|_{{C}_{0}(Y)}  \cdot \left\| e^\frac{1}{2} \right\|_{{C}_{0}(Y)} \cdot \left\| \alpha_{t}(e^\frac{1}{2}) \right\|_{{C}_{0}(Y)} \cdot \left\|  b (t) \right\|_{{C}_{0}(Y)} \\
		\leq &\ \sup_{y\in Y} \left| f_i^{\frac{1}{2}} (y) -  f_i^{\frac{1}{2}} \left( \alpha_{-t} (y) \right) \right| \cdot \left\| b (t) \right\|_{{C}_{0}(Y)} \; .
		\end{align*}
		Since $f_i$ was taken to be $\calpha$-Lipschitz with constant $\eta$, the last line above is bounded above by
		\begin{align*}
		\sup_{s \in {\mathbb{R}^{\ge 0}} , \Delta s \in [0, \eta L ] } \left| \sqrt{ s + \Delta s } - \sqrt{s} \right| \cdot \left\| b (t) \right\|_{{C}_{0}(Y)}  = &\  \sqrt{ \eta L } \cdot \left\| b (t) \right\|_{{C}_{0}(Y)} \; .
		\end{align*}
		Combined with fact that $b$ is supported inside $(-L, L)$ and our definition of $\eta$ in~\eqref{eq:defn-eta}, this gives the estimate 
		\[
		\left\| f_i^{\frac{1}{2}} \cdot { \left( e^\frac{1}{2} b e^\frac{1}{2} \right) } - { \left( e^\frac{1}{2} b e^\frac{1}{2} \right) } \cdot f_i^{\frac{1}{2}} \right\|_1 \le 2 L \cdot \sqrt{ \eta L } \cdot \left\| b \right\|_1 \leq \varepsilon
		\]
		and thus
		\[ 
		\left\| f_i^{\frac{1}{2}} \left[ f_i^{\frac{1}{2}}, { e^\frac{1}{2} b e^\frac{1}{2} } \right] \right\|_{{C}_{0}(Y)\rtimes {\mathbb{R}^{}}} \le \left\| f_i^{\frac{1}{2}} \right\|_{{C}_{0}(Y)}  \cdot  \left\| \left[ f_i^{\frac{1}{2}}, { e^\frac{1}{2} b e^\frac{1}{2} } \right] \right\|_1  \le \varepsilon \; . 
		\]
		Moreover, for each $ l \in \left\{0,\dots ,5d+4 \right\} $, since for any different $ i, j \in I^{(l)} $, we have
		$$ \calpha_{[-L, L]}(\mathrm{supp}(f_i)) \cap \calpha_{[-L, L]}(\mathrm{supp}(f_j)) = \varnothing , $$
		thus for any $ b \in F \subset {C}_{c}( (-L, L), {C}_{c}(Y))  $,
		$$ \left( f_i^{\frac{1}{2}} \cdot { \left( e^\frac{1}{2} b e^\frac{1}{2} \right) } \cdot f_j^{\frac{1}{2}} \right) (t) = \left( f_i^{\frac{1}{2}} \cdot \calpha_{t} \left(f_j^{\frac{1}{2}} \right)  \right) \cdot  \left( e^\frac{1}{2} \cdot \calpha_{t}(e^\frac{1}{2}) \right)  \cdot b(t) = 0 , $$
		both when $ |t| \ge L $, as $b(t) = 0$, and when $ |t| < L $, as $f_i^{\frac{1}{2}} \cdot \calpha_{t} \left(f_j^{\frac{1}{2}} \right) = 0$. Therefore $  f_i^{\frac{1}{2}} \cdot { \left( e^\frac{1}{2} b e^\frac{1}{2} \right) } \cdot f_j^{\frac{1}{2}}  = 0 $, and it follows that $ f_i^{\frac{1}{2}} \left[ f_i^{\frac{1}{2}}, { \left( e^\frac{1}{2} b e^\frac{1}{2} \right) } \right] $ and $ f_j^{\frac{1}{2}} \left[ f_j^{\frac{1}{2}}, { \left( e^\frac{1}{2} b e^\frac{1}{2} \right) } \right]  $ are orthogonal to each other. Consequently, using the fact that $ \sum_{i\in I}  f_i (y) = 1$ for any $ y \in K$, we see that for any $ b \in F $, 
		\begin{align*}
		& \left\|  {  e^\frac{1}{2} b e^\frac{1}{2}  } - \sum_{i\in I}  f_i^{\frac{1}{2}} \cdot { \left( e^\frac{1}{2} b e^\frac{1}{2} \right) } \cdot f_i^{\frac{1}{2}}  \right\|  \\
		= &\ \left\|  \left( \sum_{i\in I}  f_i \right) \cdot { \left( e^\frac{1}{2} b e^\frac{1}{2} \right) }   - \sum_{i\in I}  f_i^{\frac{1}{2}} \cdot { \left( e^\frac{1}{2} b e^\frac{1}{2} \right) } \cdot f_i^{\frac{1}{2}}  \right\|   \\
		= &\ \left\|   \sum_{i\in I}  f_i^{\frac{1}{2}} \left[ f_i^{\frac{1}{2}}, { \left( e^\frac{1}{2} b e^\frac{1}{2} \right) } \right]  \right\| \\
		\leq &\ \sum_{l=0} ^{5d+4} \left\| \sum_{i\in I^{(l)}}  f_i^{\frac{1}{2}} \left[ f_i^{\frac{1}{2}}, { \left( e^\frac{1}{2} b e^\frac{1}{2} \right) } \right]  \right\|   \\
		\le &\ 5(d+1) \varepsilon \; , 
		\end{align*}
		where we used the orthogonality in the norm estimate of the sum. Combined with \eqref{eq:large-triangle-upper-triangle}, we obtain $\displaystyle \left\|  {  e^\frac{1}{2} b e^\frac{1}{2}  } - \left( \sum_{i \in I} \varphi_i \right) \circ \left( \bigoplus_{i \in I} \psi_i \right) \left( e^\frac{1}{2} b e^\frac{1}{2} \right) \right\| \le 5(d+1) \varepsilon$.  
	\end{proof}

	To finish the proof of Theorem~\ref{thm-main}, we define $\psi \colon C_0(Y) \rtimes \R \to  A_{\leq R} \oplus A_{> R}$ by
	$$\psi (b) = ( \psi_{\leq R} \circ \pi ) \left( (1 - e)^\frac{1}{2} b (1 - e)^\frac{1}{2} \right) \oplus \psi_{> R} \left( e^\frac{1}{2} b e^\frac{1}{2} \right)
	$$
	for any $b \in C_0(Y) \rtimes \R$ and consider the diagram
	\[
	\xymatrix{\displaystyle
		C_0(Y) \rtimes \R \ar[dr]_{\psi  \ } \ar@{.>}[rr]^{\id} &  & C_0(Y) \rtimes \R \\
		\displaystyle & A_{\leq R} \oplus A_{> R}  \ar[ur]_{\ \ \  \varphi = \widehat{\varphi}_{\leq R} + \varphi_{> R}  } &
	}
	\]
	
	Observe that:
	\begin{enumerate}
		\item $\psi$ is completely positive and contractive.
		\item $\varphi_{\leq R}$ is a sum of $2(d + 1)$-many order zero contractions, and $\varphi_{> R}$ is a sum of $5(d+1)$-many $*$-homomorphisms (which, in particular, are order zero contractions).
		\item For all $b \in F$, we compute, using the bounds given by Claims~\ref{clm:estimate-1:thm:estimate-dimnuc-Z} and~\ref{clm:estimate-2:thm:estimate-dimnuc-Z} and using the properties of $e$:
		\begin{align*}
		& \left\|\varphi(\psi(b)) - b \right\| \\
		\leq & \ \left\| ( \widehat{\varphi}_{\leq R} \circ \psi_{\leq R} \circ \pi ) \left( (1 - e)^\frac{1}{2} b (1 - e)^\frac{1}{2} \right) - (1 - e)^\frac{1}{2} b (1 - e)^\frac{1}{2} \right\| \\
		& + \left\|  ( \varphi_{> R} \circ \psi_{>R} ) \left( e^\frac{1}{2} b e^\frac{1}{2} \right) -  e^\frac{1}{2} b e^\frac{1}{2} \right\| \\
		& + \left\| (1 - e)^\frac{1}{2} b (1 - e)^\frac{1}{2} + e^\frac{1}{2} b e^\frac{1}{2} - b \right\| \\
		\leq & \ (2d+4) \varepsilon + 5(d+1) \varepsilon + \varepsilon \\
		= & \ (7d+10) \varepsilon \, .
		\end{align*}
		\item Since $A_{\leq R}$ is a finite dimensional algebra (see~\eqref{eq:decomposable-approximation-for-quotient}), we have $\dimnuc(A_{\leq R}) = 0$. On the other hand, by the permanence properties of nuclear dimension with regard to direct sums and hereditary subalgebras (see \cite[Proposition~2.3 and~2.5]{winter-zacharias}), we have
		\[
			\dimnuc(A_{>R}) = \max_{i \in I } \operatorname{dim}_\mathrm{nuc} \left( f_i^{\frac{1}{2}} C^*_r(\mathcal{H}_{\widehat{B}_i}) f_i^{\frac{1}{2}} \right) \leq \max_{i \in I } \operatorname{dim}_\mathrm{nuc} \left( C^*_r(\mathcal{H}_{\widehat{B}_i}) \right) \;
		\]
		and by Lemma~\ref{lem:groupoid-algebra-H}, the permanence property of nuclear dimension with regard to stabilization (see \cite[Corollary~2.8]{winter-zacharias}), and the subset permanence property of covering dimension, it follows that 
		\[
		\dimnuc(A_{>R}) \leq \max_{i \in I } \operatorname{dim}_\mathrm{nuc} \left( C_0(S_{\widehat{B}_i^o}) \right) = \max_{i \in I } \operatorname{dim} \left( S_{\widehat{B}_i^o} \right) \leq \operatorname{dim} (Y) = d \;
		\]
	\end{enumerate}
	With these we can sum up the dimensions: 
	\begin{align*}
	& (\dimnuc(A_{\leq R}) + 1) 2(d+1) + ( \dimnuc(A_{>R}) + 1) 5(d+1) \\
	\leq & \ 2(d+1) + (d + 1) \cdot 5(d+1) \\
	= & \ (d+1)(5d+7) \; .
	\end{align*}
	Therefore the result follows from Lemma~\ref{Lemma:finite-dimnuc}. 
	\setcounter{Thm}{\value{clminproof}} 
\end{proof}

The following corollary shows that the second countability requirement in 
Theorem~\ref{thm-main} may be removed. 

\begin{cor}\label{cor:thm-main-nonmetrizable}
	Let $(Y, {\mathbb{R}^{}}, \calpha)$ be a topological flow. Assume $Y$ is a locally compact Hausdorff space with finite covering dimension. Then 
	\[
		\operatorname{dim}_\mathrm{nuc}({C}_{0}(Y) \rtimes {\mathbb{R}} ) \leq (\dim(Y)+1)(5\dim(Y)+7) - 1 \; .
	\]
\end{cor}

\begin{proof}
	We reduce the situation to the second countable Hausdorff case as follows. 
	By Lemma 
	\ref{lem:separable-dimnuc}, $C_0(Y)$ is the union of $\R$-invariant 
	separable $C^*$-subalgebras $C_0(X)$ with $\dim(X) \leq \dimnuc (C_0(Y))$. 
	Therefore if we let $I$ be the net of all $\R$-invariant separable 
	$C^*$-subalgebras of $C_0(Y)$ with nuclear dimension no more than 
	$\dimnuc(C_0(Y))$, ordered by inclusion, then since the spectrum of a 
	commutative separable $C^*$-algebra is second countable, we have
	\begin{align*}
	\dimnuc^{+1}(C_0(Y) \rtimes \R) \leq & \liminf_{C_0(X) \in I} \dimnuc^{+1}(C_0(X) \rtimes \R) 
	\end{align*}
	and hence the statement follows immediately from Theorem~\ref{thm-main}.
\end{proof}

We are grateful to George Elliott for pointing out a connection to the main result of \cite{Hirshberg-Wu16}: applying the mapping torus construction, we can apply Theorem~\ref{thm-main} to obtain a nuclear dimension bound for crossed products associated to topological $\mathbb{Z}$-actions on finite dimensional space.

\begin{cor}\label{cor:mapping-torus}
	Let $X$ be a locally compact Hausdorff space and $\calpha \in \operatorname{Homeo}(X)$. Then 
	\[
		\operatorname{dim}_\mathrm{nuc}({C}_{0}(X) \rtimes_{\calpha} {\mathbb{Z}} ) \leq (\dim(X)+2)(5\dim(X)+12) - 1 \; .
	\]
\end{cor} 

\begin{proof}
	Let $Y$ be the (topological) mapping torus of $\calpha$, that is, the quotient of $X \times \R$ by the diagonal action of $\Z$, where $\Z$ acts on $X$ by $\calpha$ and on $\R$ by translation (in the reverse direction). Since this $\Z$-action commutes with the translation action by $\R$ on the second factor, we obtain a topological flow $(Y, \R, \widehat{\beta})$. The embedding $X \times (0,1) \subset X \times \R$ induces an embedding $X \times (0,1) \subset Y$, which is invariant under the action of the subgroup $\Z < \R$, and the restriction of $\widehat{\beta}$ to $X \times (0,1) \subset Y$ and $\Z < \R$ agrees with $\calpha \times \operatorname{id}$. Moreover, $(Y \rtimes_{\widehat{\beta}} \R)_{X \times (0,1)}$, the reduction of the transformation groupoid $Y \rtimes_{\widehat{\beta}} \R$ to $X \times (0,1)$ (as a subset of the unit space), decomposes as a product of the transformation groupoid $X \rtimes_{\calpha} \Z$ and the pair groupoid $(0,1) \times (0,1)$. Hence $C^*_r\left((Y \rtimes_{\widehat{\beta}} \R)_{X \times (0,1)}\right)$, a hereditary subalgebra of $C^*_r(Y \rtimes_{\widehat{\beta}} \R)$ by Example~\ref{example:reduced-groupoid-algebra}, is isomorphic to the stabilization of $C^*_r(X \rtimes_{\calpha} \Z)$. It follows from the permanence properties of nuclear dimension with regard to stabilizations and hereditary subalgebras that 
	\begin{equation}\label{eq:mapping-torus-dimnuc}
		\dimnuc \left(C_0(X) \rtimes_{\calpha} \Z\right) = \dimnuc \left(C^*_r\left((Y \rtimes_{\widehat{\beta}} \R)_{X \times (0,1)}\right)\right) \leq \dimnuc \left(C_0(Y) \rtimes_{\widehat{\beta}} \R\right) \; .
	\end{equation}
	Since the complement $Y \setminus (X \times (0,1))$ is homeomorphic to $X$, we see that 
	\begin{equation}\label{eq:mapping-torus-covdim}
		\dim (Y) \leq \max \{ \dim(X \times (0,1)) , \dim (X) \} \leq \dim(X) + 1 \; .
	\end{equation}
	Our claim then follows by combining \eqref{eq:mapping-torus-dimnuc}, \eqref{eq:mapping-torus-covdim}, and Theorem~\ref{thm-main}. 
\end{proof}

Corollary~\ref{cor:mapping-torus} essentially recovers \cite[Theorem~5.1]{Hirshberg-Wu16}, with a less sharp bound. Despite the similarity of the end results, the underlying technical machinery in this paper, namely the long thin covers on flow spaces, is different from the marker property techniques powering the result of \cite{Hirshberg-Wu16}. The former was developed in the study of the Farrell-Jones conjecture by Bartels, L\"{u}ck and Reich \cite{BarLRei081465306017991882} (and improved by Kasprowski and R\"{u}ping \cite{Kasprowski-Rueping}), while the latter was introduced in the context of topological dynamics by Gutman \cite{Gutman-marker} based on \cite{Lindenstrauss95} and later adapted by Szab\'{o} \cite{szabo}.

\section{Fiberwise groupoid coverings and surjective $*$-homomorphisms}\label{section:groupoid-coverings}

In this section, we digress to discuss locally compact locally Hausdorff 
groupoids in general: we are going to define what we call \emph{fiberwise 
groupoid covering maps} and show that these maps induce quotient maps between 
maximal groupoid $C^*$-algebras. The motivating examples are constructed from 
orientable line foliations, which are the topic of 
Section~\ref{section:foliation}. 

\begin{defn}\label{def:fiberwise-groupoid-covering}
	Let $G$ and $H$ be two topological groupoids. A map $\pi \colon G \to H$ is called a \emph{fiberwise groupoid covering map} if 
	\begin{enumerate}
		\item \label{def:fiberwise-groupoid-covering-homomorphism} $\pi$ is a groupoid homomorphism (i.e., it intertwines the unit space, the maps $d$ and $r$, and the multiplication),  
		\item \label{def:fiberwise-groupoid-covering-unit-space} $\pi$ restricts to a homeomorphism between the unit spaces $G^0$ and $H^0$
		\item \label{def:fiberwise-groupoid-covering-surjective} $\pi$ is surjective, and
		\item \label{def:fiberwise-groupoid-covering-local-homeo} $\pi$ is a local homeomorphism. 
	\end{enumerate}
	Suppose $G$ and $H$ are locally compact locally Hausdorff groupoids with 
	Haar systems $\lambda_G$ and $\lambda_H$. Then a fiberwise groupoid 
	covering $\pi \colon G \to H$ is said to \emph{locally intertwine the Haar 
	systems 
	$\lambda_G$ and $\lambda_H$} if for any $u \in G^0$ and any open subset $U 
	\subset G$ for which $\pi|_U$ is a homeomorphism onto its image, the 
	push-forward of the measure $\left.\lambda_{G}^u \middle| _{U \cap G^u} 
	\right.$ under $\left. \pi \middle| _{U \cap G^u} \right.$ is equal to 
	$\left.\lambda_{H}^{\pi(u)}\middle| _{\pi(U) \cap H^{\pi(u)}} \right.$.
\end{defn}

Some immediate algebraic consequences are listed below.

\begin{lem}\label{lem:fiberwise-groupoid-covering-algebraic}
	Let $\pi \colon G \to H$ be a fiberwise groupoid covering map. Then the following are true:
	\begin{enumerate}
		\item \label{lem:fiberwise-groupoid-covering-algebraic-r-s} For any $y \in H$, there are unique $u,v \in G^0$ with $\pi(u) = r(y)$ and $\pi(v) = d(y)$. Moreoever, we have $r(\pi^{-1}(y)) = \left\{ u \right\}$ and $d(\pi^{-1}(y)) = \left\{ v \right\}$. 
		\item \label{lem:fiberwise-groupoid-covering-algebraic-group} For every $v \in H^0$, $\pi^{-1}(v)$ is a group. 
		\item \label{lem:fiberwise-groupoid-covering-algebraic-preimage} For any $x \in G$, writing $u = \pi(r(x))$ and $v = \pi(d(x))$, we have 
		\[
			\pi^{-1}(\pi(x)) = \pi^{-1}(u) \cdot x = x \cdot \pi^{-1}(v) 
		\]
	\end{enumerate}
\end{lem}

\begin{proof}
	\begin{enumerate}
		\item The first claim follows from the fact that $\pi$ restricts to a bijection from $G^0$ to $H^0$. For the second claim, we notice that $\pi(r(\pi^{-1}(y))) = r(\pi(\pi^{-1}(y))) = r(\{y\}) = \{\pi(u)\}$, which implies $r(\pi^{-1}(y)) = \left\{ u \right\}$, again by the bijectivity of $\pi$ on $G^0$. Similarly, we have $d(\pi^{-1}(y)) = \left\{ v \right\}$. 
		\item This follows immediately from part~\eqref{lem:fiberwise-groupoid-covering-algebraic-r-s} since all the elements in $\pi^{-1}(v)$ have the same range and domain, both being equal to the unique preimage of $v$. 
		\item We first observe that $\pi\left( \pi^{-1}(u) \cdot x \right) = \pi\left( \pi^{-1}(u) \right) \pi(x) = u \pi(x) = \pi(x)$, which implies $\pi^{-1}(u) \cdot x \subset \pi^{-1}(\pi(x))$. Conversely, for any $y \in \pi^{-1}(\pi(x))$, we notice that $d(y) = d(x)$ and $r(y) = r(x)$ by part~\eqref{lem:fiberwise-groupoid-covering-algebraic-r-s}, and that $\pi(y x^{-1}) = \pi(y) \pi(x)^{-1} = \pi(x) \pi(x)^{-1} = \pi\left(x x^{-1}\right) = \pi(r(x)) = u$, which implies $y \in \pi^{-1}(u) \cdot x$. Hence $\pi^{-1}(\pi(x)) \subset \pi^{-1}(u) \cdot x$. Combining the two direction, we get $\pi^{-1}(\pi(x)) = \pi^{-1}(u) \cdot x$. Similarly, $\pi^{-1}(\pi(x)) = x \cdot \pi^{-1}(v)$. 
	\end{enumerate}
\end{proof}

Note that a fiberwise groupoid covering map $\pi \colon G \to H$ need not be a topological covering map from $G$ to $H$ (i.e., for any $y \in H$, there is an open neighborhood $V$ of $y$ in $H$ such that $\pi^{-1}(V)$ is a disjoint union of open subsets of $G$, each of which is homeomorphic to $V$ through $\pi$). An example is given by taking $G$ to be the transformation groupoid $[0,1] \rtimes \mathbb{Z}/2\mathbb{Z}$ associated to the trivial action, and $\pi$ to be the map that collapses every fiber to a single unit element except for the fiber at $0 \in [0,1]$. However, the following observation justifies our use of the term ``fiberwise groupoid covering''. 

\begin{lem}\label{lem:fiberwise-groupoid-covering-topological-covering}
	Let $\pi \colon G \to H$ be a fiberwise groupoid covering map. Then for every $u \in G^0$, the restriction $\left. \pi \middle| _{G^u} \right. \colon G^u \to H^{\pi(u)}$ (respectively, $\left. \pi \middle| _{G_u} \right. \colon G_u \to H_{\pi(u)}$) is a regular covering map that is implemented by the action of the group $\pi^{-1}(u)$ on $G^u$ by left multiplication (respectively, on $G_u$ by right multiplication). 
\end{lem}

\begin{proof}
	We only prove the lemma for $\left. \pi \middle| _{G^u} \right. \colon G^u \to H^{\pi(u)}$. The case for $\left. \pi \middle| _{G_u} \right. \colon G_u \to H_{\pi(u)}$ follows by noticing that $\pi$ also serves as a fiberwise groupoid covering from the opposite groupoid $G^{\operatorname{op}}$ to $H^{\operatorname{op}}$. 
	
	Given $y \in H^{\pi(u)}$, we will find an open neighborhood $V$ of $y$ in 
	$H^{\pi(u)}$ and an open set $W$ in $G^u$ such that $\pi^{-1}(V)$ is the 
	disjoint union of $z \cdot W$, where $z$ ranges over the group 
	$\pi^{-1}(u)$, and $\pi$ maps each $z \cdot W$ homeomorphically onto 
	$\pi(W)$. Since $\pi$ is surjective, we can choose a certain $x \in 
	\pi^{-1}(y)$. Lemma~\ref{lem:fiberwise-groupoid-covering-algebraic} implies 
	that $\pi^{-1}(y) \subset G^u$, $\pi^{-1}(v)$ is a group, and $\pi^{-1}(y) 
	= \pi^{-1}(v) \cdot x$. 
	Since $\pi$ is a local homeomorphism, there is an open neighborhood $U$ of $x$ in $G$ such that $\left. \pi \middle| _{U} \right.$ is a homeomorphism onto its image. Let $W = U \cap G^u$ and $V = \pi(W)$, which is contained in $H^{\pi(u)}$. Note that $\pi^{-1}(V) = \pi^{-1}(v) \cdot  W$ by Lemma~\ref{lem:fiberwise-groupoid-covering-algebraic}\eqref{lem:fiberwise-groupoid-covering-algebraic-preimage}. Also notice that for any $z \in \pi^{-1}(v)$, the restriction of $\pi$ to $z W$ is equal to the composition of $\left. \pi \middle| _{W} \right.$ and the left multiplication map by $z^{-1}$, and is thus a homeomorphism onto its image $V$. Lastly, we show $\left\{ z \cdot W \colon z \in  \pi^{-1}(v) \right\}$ is a disjoint family of subsets. Suppose we could find $w \in (z \cdot W) \cap (z' \cap W)$ for different $z , z'  \in \pi^{-1}(v)$, we would get two different elements $z^{-1}w$ and $(z')^{-1}w$ in $W$ having the same image under $\pi$, thus contradicting our assumption that $\left. \pi \middle| _{U} \right.$ is a homeomorphism onto its image. 
	In summary, we have shown that $\pi^{-1}(V)$ is a disjoint union of open subsets $z \cdot W$, for $z \in \pi^{-1}(v)$, each of which is homeomorphic to $V$ through $\left. \pi \middle| _{G^u} \right.$.
\end{proof}

\begin{nota}\label{nota:fiberwise-groupoid-covering-topological-C_c_pi}
	Let $\pi \colon G \to H$ be a fiberwise groupoid covering map. Then
	\begin{enumerate}
		\item we write $C_c(G)_\pi$ for the set of complex-valued functions $f$ on $G$ for which there exists a Hausdorff open set $U \subset G$ so that $\left.\pi\middle|_U\right.$ is a homeomorphism onto its image, $f$ vanishes outside of $U$, and $f|_U$ is continuous and compactly supported;
		\item we write $C_c(H)_\pi$ for the set of complex-valued functions $g$ on $H$ for which there exists a Hausdorff open set $U \subset G$ so that $\left.\pi\middle|_U\right.$ is a homeomorphism onto its image, $f$ vanishes outside of $\pi(U)$, and $f|_U$ is continuous and compactly supported. 
	\end{enumerate} 
\end{nota}

\begin{lem}\label{lem:fiberwise-groupoid-covering-topological-C_c_pi}
	Let $\pi \colon G \to H$ be a fiberwise groupoid covering map. Then $C_c(G)$ (respectively, $C_c(H)$), as given in Definition~\ref{nota:C_c_0}, is linearly spanned by $C_c(G)_\pi$ (respectively, $C_c(H)_\pi$). 
\end{lem}

\begin{proof}
	It is clear that $C_c(G)_\pi \subset C_c(G)_0 \subset C_c(G)$. Thus by Definition~\ref{nota:C_c_0}, one just needs to show each $f \in C_c(G)_0$ lies in the linear span of $C_c(G)_\pi$, but this is a standard argument using a partition of unity. The case for $H$ is similar. 
\end{proof}

\begin{lem}\label{lem:fiberwise-groupoid-covering-topological-finite-preimage}
	Let $\pi \colon G \to H$ be a fiberwise groupoid covering map. Then for any compact set $K \subset G$ and $y \in H$, the set $K \cap \pi^{-1}(y)$ is finite. Therefore, for any complex function $f$ on $G$ with compact support and $y \in H$, there are only finitely many $x \in \pi^{-1}(y)$ such that $f(x) \not= 0$. 
\end{lem}

\begin{proof}
	Let $u \in G^0$ be such that $\pi(u) = r(y)$. Then $K \cap G^u$, being a closed subset of the compact set $K$, is also compact. By Lemma~\ref{lem:fiberwise-groupoid-covering-topological-covering}, the restriction $\left. \pi \middle| _{G^u} \right. \colon G^u \to H^{\pi(u)}$ is a covering map, which entails that $\pi^{-1}(y)$ is a closed discrete subset of $G^u$, and thus $K \cap \pi^{-1}(y)$, being a compact discrete set, must be finite. 
\end{proof}

\begin{lem}\label{lem:fiberwise-groupoid-covering-topological-pi-star}
	Let $\pi \colon G \to H$ be a fiberwise groupoid covering map. Then the formula 
	\begin{equation}\label{eq:sum-preimage}
		\pi_*(f)(y) = \sum_{x \in \pi^{-1}(y)} f(x) \quad \text{for~} f \in C_c(G) \text{~and~} y \in H
	\end{equation}
	defines a surjective linear map $\pi_* \colon C_c(G) \to C_c(H)$. Similarly, for any $u \in G^0$, the formula
	\begin{equation}\label{eq:sum-preimage-u}
		\pi_*^u(f)(y) = \sum_{x \in \pi^{-1}(y)} f(x) \quad \text{for~} f \in C_c(G^u) \text{~and~} y \in H^{\pi(u)}
	\end{equation}
	defines a surjective linear map $\pi_*^u \colon C_c(G^u) \to C_c\left(H^{\pi(u)}\right)$. Moreover, we have, for any $x \in G$ and $f \in C_c \left(G^{d(x)}\right)$, 
	\begin{equation}\label{eq:sum-preimage-equivariant}
		\pi_*^{r(x)}(x \cdot f) = \pi(x) \cdot \pi_*^{d(x)}(f) \; ,
	\end{equation}
	where $x \cdot f (z) = f(x^{-1} z)$ and $\pi(x) \cdot \pi_*^{d(x)}(f) (y) = \pi_*^{d(x)}(f) \left( \left(\pi(x) \right)^{-1} y\right)$ for any $z \in G^{r(x)}$ and $y \in H^{\pi(r(x))}$.
\end{lem}

\begin{proof}
	It follows from Lemma~\ref{lem:fiberwise-groupoid-covering-topological-finite-preimage} that for any $f \in C_c(G)_0$ and $y \in H$, the sum in \eqref{eq:sum-preimage} is always finite. The same is thus also true for all $f \in C_c(G)$ by linearity. Then the assignment $f \mapsto \pi_*(f)$ defines a linear map from $C_c(G)$ to the space of all functions on $H$. For any Hausdorff open subset $U \subset G$ where $\pi$ is restricted to a homeomorphism onto its image, we observe that if $f \in C_c(U)$, then $\pi_*(f)$ is supported in $\pi(U)$ and $\left. \pi_*(f)\middle|_{\pi(U)} \right. = f \circ \left(\pi\middle|_U\right)^{-1}$, and every $g \in C_c(\pi(U))$ can be realized this way. Hence $\pi_*$ maps $C_c(G)_\pi$ surjectively onto $C_c(H)_\pi$. It follows by linearity that $\pi_* \left(C_c(G)\right) = C_c(H)$.
	
	The case for $\pi_*^u$ is entirely analogous. 
	
	To prove the last equation, we observe that for any $x \in G$ and $y \in H^{\pi(r(x))}$, if we write $\widetilde{y} $ for an arbitrary element in $\pi^{-1}(y)$, we have by Lemma~\ref{lem:fiberwise-groupoid-covering-algebraic} that
	\[
		\pi^{-1} \left( \left(\pi(x) \right)^{-1} y\right) = \pi^{-1} \left( \pi\left(x^{-1} \widetilde{y} \right) \right) = x^{-1} \widetilde{y} \cdot \pi^{-1} (d(y)) = x^{-1} \cdot \pi^{-1} \left( y \right) \; .
	\]
	Hence for any $f \in C_c \left(G^{d(x)}\right)$, we have
	\begin{align*}
		\pi(x) \cdot \pi_*^{d(x)}(f) (y) &= \pi_*^{d(x)}(f) \left( \left(\pi(x) \right)^{-1} y\right) \\
		& = \sum_{z \in \pi^{-1} \left(\left(\pi(x) \right)^{-1} y\right)} f(z) \\
		& = \sum_{z \in x^{-1} \cdot \pi^{-1} \left( y \right)} f(z) \\
		& = \sum_{z' \in \pi^{-1} \left( y \right)} f(x^{-1} z') \\
		& = \pi_*^{r(x)}(x \cdot f) (y)
	\end{align*}
	as desired.
\end{proof}

The following lemma is a special case of a general fact about regular covers. We include a proof for completeness. 
\begin{lem}\label{lem:fiberwise-groupoid-covering-topological-pi-star-kernel}
	Let $\pi \colon G \to H$ be a fiberwise groupoid covering map. Then for any $u \in G^0$ such that $G^u$ and $H^{\pi(u)}$ are locally compact and Hausdorff, the kernel of $\pi_*^u \colon C_c(G^u) \to C_c\left(H^{\pi(u)}\right)$ is linearly spanned by functions of the form $f - z \cdot f$, for $f \in C_c(G^u)$ and $z \in \pi^{-1}(u)$, where $z \cdot f$ denotes the function $x \mapsto f(z^{-1} \cdot x)$. 
\end{lem}

\begin{proof}
	We first observe that for any $f \in C_c(G^u)$ and $z \in \pi^{-1}(u)$, since $\pi_*^u(z \cdot f) = \pi_*^u(f)$, the difference $f - z \cdot f$ is indeed in $\ker \left( \pi_*^u \right)$. Thus it suffices to show that any $g \in \ker \left( \pi_*^u \right)$ can be written as a finite sum of such elements. 
	
	We first assume there is an open subset $U \subset G^u$ such that $\pi$ maps $U$ homeomorphically to its image and $g$ is supported in $\pi^{-1}(\pi(U))$, which, by Lemma~\ref{lem:fiberwise-groupoid-covering-topological-covering}, is equal to $\bigsqcup_{z \in \pi^{-1}(u)} z \cdot U$. Thus in this case, $g$ can be written as $\sum_{z \in \pi^{-1}(u)} g_z$, where each $f_z$ is supported in $z \cdot U$. Since $g$ is compactly supported, this is a finite sum, i.e., there are $z_1, \ldots, z_n \in \pi^{-1}(u)$ such that $g = \sum_{i = 1}^n g_{z_i}$. Observe that 
	\[
		\pi_*^u \left( \sum_{i = 1}^n z_i^{-1} \cdot g_{z_i} \right) = \sum_{i = 1}^n \pi_*^u \left(  z_i^{-1} \cdot g_{z_i} \right) = \sum_{i = 1}^n \pi_*^u \left( g_{z_i} \right) = \pi_*^u(g) = 0 \; .
	\]
	On the other hand, $\sum_{i = 1}^n z_i^{-1} \cdot g_{z_i}$ is supported inside $\bigcup_{i=1}^n z_i^{-1}  \cdot z_i \cdot U$, which is just $U$. It follows that its image under $\pi_*^u$ is supported in $\pi(U)$, and
	\[
		\left. \pi_*^u \left( \sum_{i = 1}^n z_i^{-1} \cdot g_{z_i} \right) \middle|_{\pi(U)} \right. = \left( \sum_{i = 1}^n z_i^{-1} \cdot g_{z_i} \right) \circ \left( \pi \middle|_U \right)^{-1} \; .
	\]
	Combining these equations, we see that $\sum_{i = 1}^n z_i^{-1} \cdot g_{z_i} = 0$. Therefore we have
	\[
		g = g - 0 = \sum_{i = 1}^n \left( g_{z_i} - z_i^{-1} \cdot g_{z_i} \right)
	\]
	as desired. 
	
	For a general $g \in \ker \left( \pi_*^u \right)$, we cover the support of $\pi_*^u(g)$ by finitely many open sets $W_i$, for $i = 1, \ldots, n$, where each $W_i$ is the homeomorphic image of some open subset $U_i \subset G^u$ under $\pi$. Let $\{ f_i \colon i = 1, \ldots, n \}$ be a set of functions on $H^{\pi(U)}$ such that each $f_i$ is supported in $W_i$ and $\sum_{i=1}^{n} f_i$ is equal to $1$ on the support of $\pi_*^u(g)$. For $i = 1, \ldots, n$, writing $\widetilde{f}_i = f_i \circ \pi$, we observe that $\pi_*^u(\widetilde{f}_i g) = f_i \pi_*^u(g) = 0$ and $\widetilde{f}_i g$ is supported in $\pi^{-1}(W_i)$; hence by the previous paragraph $\widetilde{f}_i g$ is in the linear span of $\{ f - z \cdot f \in C_c(G^u), z \in \pi^{-1}(u) \}$. Since $\sum_{i=1}^{n} \widetilde{f}_i$ is equal to $1$ on the support of $g$, we have $g = \sum_{i=1}^{n} \widetilde{f}_i g$. It follows that $g$ is also  in the linear span of $\{ f - z \cdot f \in C_c(G^u), z \in \pi^{-1}(u) \}$.
\end{proof}

\begin{lem}\label{lem:fiberwise-groupoid-covering-topological-intertwine}
	Let $G$ and $H$ be locally compact locally Hausdorff groupoids with Haar 
	systems $\lambda_G$ and $\lambda_H$. Let $\pi \colon G \to H$ be a 
	fiberwise groupoid covering map. Then $\pi$ locally intertwines $\lambda_G$ 
	and 
	$\lambda_H$ if and only if for any $u \in G^0$ and $f \in C_c(G^u)$, we have
	\[
		\int_{G^u} f \, d \lambda_G^u = \int_{H^{\pi(u)}} \pi_*^u(f) \, d \lambda_H^{\pi(u)} \; .
	\] 
\end{lem}

\begin{proof}
	By linearity, it suffices to show, for any $u \in G^0$ and any open subset $U \subset G$ for which $\pi|_U$ is a homeomorphism onto its image, that $\left.\lambda_{H}^{\pi(u)}\middle| _{\pi(U) \cap H^{\pi(u)}} \right.$ is the push-forward of $\left. \lambda_{G}^u \middle| _{U \cap G^u}\right.$ under $\left. \pi \middle| _{U \cap G^u} \right.$
	if and only if for any $f \in C_c(U \cap G^u)$, we have
	\[
		\int_{G^{u}} f \, d \lambda_G^{u} = \int_{H^{\pi(u)}} \pi_*^u(f) \, d \lambda_H^{\pi(u)} \; .
	\]
	But this is true since the former statement is equivalent to that for any $f \in C_c(U \cap G^u)$, we have
	\[
		\int_{U \cap G^u} f \, d \lambda_G^{u} = \int_{\pi(U) \cap H^{\pi(u)}} f \circ \left( \pi\middle|_{U \cap G^u} \right)^{-1} \, d \lambda_H^{\pi(u)} \; ,
	\]
	while the left-hand side of this equation is equal to $\int_{G^{u}} f \, d \lambda_G^{u}$ and the right-hand side is equal to $\int_{H^{\pi(u)}} \pi_*^u(f) \, d \lambda_H^{\pi(u)}$ because $\pi_*^u(f)$ is supported in $\pi(U) \cap H^{\pi(u)}$ and $\left. \pi_*^u(f) \middle|_{\pi(U) \cap H^{\pi(u)}} \right. = f \circ \left( \pi\middle|_{U \cap G^u} \right)^{-1}$. This shows the two statements are equivalent, as desired. 
\end{proof}

\begin{thm}\label{thm:fiberwise-groupoid-covering-locally-compact}
	Let $\pi \colon G \to H$ be a fiberwise groupoid covering map. Then the following hold: 
	\begin{enumerate}
		\item If $H$ is a locally compact locally Hausdorff groupoid with Haar 
		system $\lambda_H$, then $G$ is also locally compact locally Hausdorff 
		and can be equipped with a Haar system $\lambda_G$ such that $\pi$ 
		locally 
		intertwines $\lambda_G$ and $\lambda_H$. 
		\item If $G$ is a locally compact locally Hausdorff groupoid with Haar 
		system $\lambda_G$, then $H$ is also locally compact locally Hausdorff 
		and can be 
		equipped with a Haar system $\lambda_H$ such that $\pi$ locally 
		intertwines 
		$\lambda_G$ and $\lambda_H$. 
	\end{enumerate}
\end{thm}

\begin{proof}
	For each of the two statements, to prove a topological groupoid is actually 
	a locally compact locally Hausdorff groupoid, we need to check the four 
	conditions in Definition~\ref{def:locally-compact-groupoids}. In fact, the 
	first three conditions are straightforward \emph{for both statements}: 
	\begin{itemize}
		\item Condition~\eqref{def:locally-compact-groupoids:unit-space} is a direct consequence of Definition~\ref{def:fiberwise-groupoid-covering}\eqref{def:fiberwise-groupoid-covering-unit-space}. 
		\item Condition~\eqref{def:locally-compact-groupoids:locally-Hausdorff} can be easily verified using the fact that $\pi$ is a local homeomorphism. 
		\item Condition~\eqref{def:locally-compact-groupoids:sections} follows from Lemma~\ref{lem:fiberwise-groupoid-covering-topological-covering} and the fact in elementary topology that given a topological covering map, if either the domain or the codomain is locally compact (respectively, Hausdorff), then so is the other. 
	\end{itemize}
	Therefore, all that is left is to construct Haar systems. For this we use the construction in Lemma~\ref{lem:fiberwise-groupoid-covering-topological-pi-star} and treat the two statements separately. 
	\begin{enumerate}
		\item We first prove the statement where we assume $H$ is a locally 
		compact locally Hausdorff groupoid with Haar system $\lambda_H$. 
		In this case, we define, for each $u \in G^0$, a linear map  
		\[
		\Lambda_G^u \colon C_c(G^u) \to \mathbb{C}
		\]
		such that for any $f \in C_c(G^u)$ and $u \in G^0$, we have
		\[
			\Lambda_G^u (f) = \int_{H^{\pi(u)}} \pi_*^u(f) \, d \lambda_H^{\pi(u)} \;.
		\]
		Clearly $\Lambda_G^u$ maps non-negative function to non-negative functions, and thus, by the Riesz representation theorem, determines a positive regular Borel measure $\lambda_G^u$ on $G^u$. 
		
		We claim that $\left\{ \lambda_G^u \right\}_{u \in G^0}$ is a Haar system we look for. To this end, we verify conditions~\eqref{def:locally-compact-groupoids:Haar-system:full-support}-\eqref{def:locally-compact-groupoids:Haar-system:invariance} in Definition~\ref{def:locally-compact-groupoids}: 
		\begin{itemize}
			\item Condition~\eqref{def:locally-compact-groupoids:Haar-system:full-support} follows from the observation that any nonzero non-negative function $f \in C_c(G^u)$ is taken by $\pi_*$ to a nonzero non-negative function in $C_c\left(H^{\pi(u)}\right)$, and thus $\Lambda_G^u (f) > 0$ since $\lambda_H^{\pi(u)}$ has full support. 
			\item To prove condition~\eqref{def:locally-compact-groupoids:Haar-system:continuous}, we define a linear map 
			\[
			\Lambda_G \colon C_c(G) \to C_c(G^0) 
			\]
			such that for any $f \in C_c(G)$ and $u \in G^0$, we have
			\[
			\Lambda_G (f) (u) = \int_{H^{\pi(u)}} \pi_*(f) \, d \lambda_H^{\pi(u)} \; .
			\]
			Here $\Lambda_G (f)$ is indeed in $C_c(G^0)$ thanks to Definition~\ref{def:locally-compact-groupoids}\eqref{def:locally-compact-groupoids:Haar-system:continuous} and Definition~\ref{def:fiberwise-groupoid-covering}\eqref{def:fiberwise-groupoid-covering-unit-space}. Comparing the constructions, we see that for any $f \in C_c(G)$ and $u \in G^0$, we have $\int_{G^u} f \, d \lambda_G^u = \Lambda_G^u \left(\left. f \middle|_{G^u} \right.\right) = \Lambda_G (f) (u)$. Hence the function $u \mapsto \int_{G^u} f \, d \lambda_G^u$ is in $C_c(G^0)$. 
			\item To prove condition~\eqref{def:locally-compact-groupoids:Haar-system:invariance}, we apply Equation~\eqref{eq:sum-preimage-equivariant} to see that for any $x \in G$ and $f \in C_c(G)$, we have
			\[
				\pi_*^{d(x)} \left(x^{-1} \cdot f\middle|_{G^{r(x)}} \right) = \pi(x^{-1}) \cdot \pi_*^{r(x)} \left( f \middle|_{G^{r(x)}} \right)  \; ,
			\]
			which, together with condition~\eqref{def:locally-compact-groupoids:Haar-system:invariance} for $\{\lambda_H^u\}_{u \in H^0}$, implies
			\begin{align*}
				& \ \int_{H^{\pi(d(x))}} \pi_*^{d(x)} \left(x^{-1} \cdot f\middle|_{G^{r(x)}} \right)  \, d \lambda_H^{\pi(d(x))} \\
				& = \int_{H^{\pi(d(x))}} \pi(x^{-1}) \cdot \pi_*^{r(x)} \left( f \middle|_{G^{r(x)}} \right)  \, d \lambda_H^{\pi(d(x))} \\
				& = \int_{H^{\pi(r(x))}} \pi_*^{r(x)} \left( f \middle|_{G^{r(x)}} \right)  \, d \lambda_H^{\pi(r(x))} \;.
			\end{align*}
			Unwrapping the definitions, we see that
			\[
				\int_{G^{d(x)}} f(xz) \, d \lambda_G^{d(x)} (z) = \int_{G^{r(x)}} f(y) \, d \lambda_G^{r(x)} (y) \; .
			\]
		\end{itemize}
		Therefore we have shown that $G$ is a locally compact locally Hausdorff 
		groupoid with a Haar system $\lambda_G$. By 
		Lemma~\ref{lem:fiberwise-groupoid-covering-topological-intertwine}, it 
		is clear from our construction that $\pi$ locally intertwines 
		$\lambda_G$ and 
		$\lambda_H$. 
		\item Then we turn to the second statement where we assume $G$ is a locally compact groupoid with Haar system $\lambda_G$. For each $u \in G^0$, consider the linear functional 
		\[
			\Lambda_G^u \colon C_c(G^u) \to \mathbb{C} \, , \quad f \mapsto \int_{G^u} f \, d \lambda_G^u \; .
		\]
		By condition~\eqref{def:locally-compact-groupoids:Haar-system:invariance}, for any $z \in \pi^{-1}(\pi(u))$, since $r(z) = d(z) = u$, we have 
		\[
			\Lambda_G^u(z \cdot f) =  \int_{G^u} z \cdot f \, d \lambda_G^u = \int_{G^u} f \, d \lambda_G^u = \Lambda_G^u(f) \; .
		\]
		Applying Lemma~\ref{lem:fiberwise-groupoid-covering-topological-pi-star-kernel}, we have
		\[
			\ker \Lambda_G^u \supset \operatorname{span} \left\{ f - z \cdot f \colon f \in C_c(G^u), z \in \pi^{-1}(\pi(u)) \right\} = \ker \pi_*^u \; .
		\]
		Since $\pi_*^u \colon C_c(G^u) \to C_c\left( H^{\pi(u)} \right)$ is a surjection, we see that $\Lambda_G^u$ factors through a linear functional on $C_c(H^{\pi(u)})$, which we may denote by $\Lambda_H^{\pi(u)}$ without any ambiguity because $\pi$ maps $G^0$ bijectively onto $H^0$.
		The positivity of $\Lambda_G^u$ implies that of $\Lambda_H^{\pi(u)}$ because by a partition-of-unity argument, any non-negative function in $C_c\left( H^{\pi(u)} \right)$ has a non-negative preimage in $C_c(G^u)$. Hence by the Riesz representation theorem, $\Lambda_H^{\pi(u)}$ determines a positive regular Borel measure $\lambda_H^{\pi(u)}$ on $H^{\pi(u)}$. It satisfies
		\[
			\int_{G^u} f \, d \lambda_G^u = \int_{H^{\pi(u)}} \pi_*^u(f) \, d \lambda_H^{\pi(u)} 
		\]
		for any $f \in C_c(G^u)$.

		We claim that $\{\lambda_H^u\}_{u \in H^0}$ is a Haar system we look for. To this end, we verify conditions~\eqref{def:locally-compact-groupoids:Haar-system:full-support}-\eqref{def:locally-compact-groupoids:Haar-system:invariance} in Definition~\ref{def:locally-compact-groupoids}: 
		\begin{itemize}
			\item Condition~\eqref{def:locally-compact-groupoids:Haar-system:full-support} follows from the observation that by a partition-of-unity argument, any nonzero non-negative function $f \in C_c\left( H^{\pi(u)} \right)$ has a nonzero non-negative preimage in $C_c(G^u)$, and thus $\Lambda_H^{\pi(u)} (f) > 0$ since $\lambda_G^u$ has full support. 
			\item Condition~\eqref{def:locally-compact-groupoids:Haar-system:continuous} for $\lambda_H^{\pi(u)}$ follows directly from that for $\lambda_G^u$. 
			\item Condition~\eqref{def:locally-compact-groupoids:Haar-system:invariance} for $\lambda_H^{\pi(u)}$ also follows from that for $\lambda_G^u$. Indeed, for any $x \in H$ and $f \in C_c(H)$, choosing $\widetilde{x} \in \pi^{-1}$ and $\widetilde{f} \in \left(\pi_*\right)^{-1} (f)$ by surjectivity, we apply  Equation~\eqref{eq:sum-preimage-equivariant} to get
			\[
				\pi_*^{d(\widetilde{x})} \left(\widetilde{x}^{-1} \cdot \widetilde{f} \middle|_{G^{r(\widetilde{x})}} \right) = x^{-1} \cdot \pi_*^{r(\widetilde{x})} \left( \widetilde{f} \middle|_{G^{r(\widetilde{x})}} \right) = x^{-1} \cdot \left. {f} \middle|_{H^{r({x})}} \right. \; ,
			\]
			which implies
			\begin{align*}
				& \ \int_{H^{d(x)}} f(xz) \, d \lambda_H^{d(x)} (z) \\
				& = \int_{G^{d(\widetilde{x})}} \widetilde{f}(\widetilde{x}y) \, d \lambda_G^{d(\widetilde{x})} (y) \\
				& = \int_{G^{r(\widetilde{x})}} \widetilde{f}(y') \, d \lambda_G^{r(\widetilde{x})} (y') \\
				& = \int_{H^{r(x)}} f(z') \, d \lambda_H^{r(x)} (z')  \; ,
			\end{align*}
			as desired. 
		\end{itemize}
		Therefore we have shown that $H$ is a locally compact locally Hausdorff 
		groupoid with a Haar system $\lambda_H$. By 
		Lemma~\ref{lem:fiberwise-groupoid-covering-topological-intertwine}, it 
		is clear from our construction that $\pi$ locally intertwines 
		$\lambda_G$ and 
		$\lambda_H$. 
	\end{enumerate}
\end{proof}

\begin{thm}\label{thm:fiberwise-groupoid-covering-quotient}
	Let $G$ and $H$ be locally compact locally Hausdorff groupoids with Haar 
	systems $\lambda_G$ and $\lambda_H$. Let $\pi \colon G \to H$ be a 
	fiberwise groupoid covering map that locally intertwines $\lambda_G$ and 
	$\lambda_H$. Then the surjection $\pi_* \colon C_c(G) \to C_c(H)$ given in 
	Lemma~\ref{lem:fiberwise-groupoid-covering-topological-pi-star} is a 
	$*$-homomorphism between the convolution $*$-algebras. Therefore, the 
	$C^*$-algebra $C^*(H)$ is a quotient of $C^*(G)$. 
\end{thm}

\begin{proof}
	Lemma~\ref{lem:fiberwise-groupoid-covering-topological-intertwine} tells us that for any $f \in C_c(G)$, we have 
	\begin{equation}\label{eq:push-forward-integrals}
	\int_{G^u} f \, d\lambda^u = \int_{ H^{\pi(u)}} \pi_*(f) \, d\lambda^{\pi(u)} \; ,
	\end{equation}
	which then entails, by Equation~\eqref{eq:groupoid-multiplication-r},
	\[
		\pi_*( f \ast g) = \pi_*( f) \ast \pi_*(g) \; 
	\]
	for any $g \in C_c(G)$. It is also straightforward to check
	\[
		\pi_*(f^*) = (\pi_*(f))^*
	\]
	by the defining Equation~\eqref{eq:groupoid-star-operation}. This shows $\pi_* \colon C_c(G) \to C_c(H)$ is a $*$-homomorphism. 
	
	The last statement about $C^*$-algebras follows from the surjectivity of $\pi_*$ by a standard density argument. 
\end{proof}

\begin{eg}
	The quotient homomorphism from $\mathbb{R}$ to $\mathbb{R} / \mathbb{Z}$, as locally compact groups, is a fiberwise groupoid covering. It thus induces a surjection from $C^*(\mathbb{R})$ to $C^*(\mathbb{R} / \mathbb{Z})$. Using Pontryagin duality, this corresponds to the restriction $*$-homomorphism $C_0(\widehat{\mathbb{R}})$ to $C_0(\widehat{\mathbb{R} / \mathbb{Z}})$ induced by the inclusion $\widehat{\mathbb{R} / \mathbb{Z}} \cong \mathbb{Z} \hookrightarrow \mathbb{R} \cong \widehat{\mathbb{R}}$. On the level of the convolution algebras $L^1(\mathbb{R}) \to L^1(\mathbb{R}/\mathbb{Z})$, when the latter is viewed as consisting of periodic functions, the map is the periodization map $\pi_*(f)(x) = \sum_{n \in \mathbb{Z}} f(x+n)$.
\end{eg}

A more interesting class of examples is given by orientable line foliations. We will discuss this in the next section. 

\section{One dimensional foliations} \label{section:foliation}

In this section, we show that if $\mathcal{F}$ is a one-dimensional orientable 
foliation of a locally compact second countable Hausdorff space with finite 
covering 
dimension, then the associated foliation $C^*$-algebra has finite nuclear 
dimension. We refer the reader to \cite{moore-schochet} for a discussion of the 
groupoid and $C^*$-algebra associated to a foliation (though in our case we do 
not need to assume that $Y$ is a manifold), and to \cite{paterson} for a 
discussion of the $C^*$-algebra associated to a locally compact locally 
Hausdorff groupoid, which 
we summarized in Section~\ref{section:prelim}.

\begin{defn}[\cite{moore-schochet}]\label{def:line-foliation}
	Recall that a \emph{line foliation chart} or a \emph{one-dimensional 
	foliation chart} on a locally compact second countable Hausdorff space $Y$ 
	is a 
	triple $(U, T, \varphi)$, where $U$ is an open subset of $Y$, $T$ is some 
	locally compact second countable Hausdorff space and $\varphi \colon U \to 
	\R \times 
	T$ is a homeomorphism. The preimages under $\varphi$ of sets of the form 
	$\R \times \{x\}$ are called \emph{plaques}. Two line foliation charts $(U, 
	T, \varphi)$ and $(U', T', \varphi')$ are coherent if for any plaque $P$ in 
	$(U, T, \varphi)$ and $P'$ in $(U', T', \varphi')$, the intersection $P 
	\cap P'$ is open in both $P$ and $P'$. A \emph{line foliation} or a 
	\emph{one-dimensional foliation} on $Y$ is a foliated atlas, i.e., a 
	collection of coherent foliation charts that cover $Y$. Two foliated 
	atlases are considered to produce the same foliation if their union is also 
	a foliated atlas. The relation of being in the same plaque in a chart 
	generates an equivalence relation, whose equivalence classes are called the 
	\emph{leaves} of the foliation. For each $y \in Y$, if $(U, T, \varphi)$ is 
	a chart such that $y \in U$, then $T$, together with the base point given 
	by the second coordinate of $\varphi(y)$, is called a \emph{transversal} of 
	$\mathcal{F}$ at $y$. Although this depends on the choice of the chart, 
	since any two charts $(U, T, \varphi)$ and $(U', T', \varphi')$ that cover 
	$y$ agree on their intersection, there is thus a canonical homeomorphism 
	between neighborhoods of the base points in $T$ and $T'$. 
\end{defn}

\begin{defn}[\cite{moore-schochet}]\label{def:orientable-line-foliation}
	A line foliation is \emph{orientable} if it has a foliated atlas such that for any two line foliation charts $(U, T, \varphi)$ and $(U', T', \varphi')$ in this atlas, the transition map 
	\[
		\left. \varphi \middle|_{U\cap U'} \right. \circ \left( \varphi' \middle|_{U\cap U'} \right)^{-1} \colon  \varphi'(U\cap U') \to \R \times T
	\]
	is increasing in the first coordinate in the sense that for any $(s, x)$ and $(t, x)$ in $\varphi'(U\cap U')$, whenever $s \leq t$, the first coordinate of $\left. \varphi \middle|_{U\cap U'} \right. \circ \left( \varphi' \middle|_{U\cap U'} \right)^{-1} (s, x)$ is less than or equal to that of $\left. \varphi \middle|_{U\cap U'} \right. \circ \left( \varphi' \middle|_{U\cap U'} \right)^{-1} (t, x)$. 
\end{defn}

It is clear that a flow without fixed points gives rise to an orientable line foliation, where the collection of all tubes provides an atlas of local charts, by Lemma~\ref{lem:tubes-basics}. The converse is also true. 

\begin{thm}[{\cite[Theorem~27A]{Whitney}}]\label{thm:Whitney}
	If $\mathcal{F}$ is an orientable line foliation of a locally compact 
	second countable Hausdorff space $Y$, then there exists a flow $\calpha$ on 
	$Y$ with 
	no fixed points such that the leaves of $\mathcal{F}$ are exactly the 
	orbits of $\calpha$. 
\end{thm}

\begin{cnstrct}[\cite{moore-schochet}]\label{construction:holonomy-groupoid}
	Let us recall the construction of the holonomy groupoid of a line foliation, which is in fact defined for general foliations and plays an important role in the index theory for the longitudinal elliptic operators (\cite{Connes-Skandalis}). Suppose $\mathcal{F}$ is a line foliation of a locally compact Hausdorff space $Y$. We first observe that for any line foliation chart $(U, T, \varphi)$ and any $y, y' \in U$, if $y$ and $y'$ are on the same plaque, then $T$ can be taken to be a transversal at both $y$ and $y'$, with the same base point. It follows that in this situation, for any transversals $T'$ at $y$ and $T''$ at $y'$, there is a canonical homeomorphism between neighborhoods of the base points of $T'$ and $T''$. For any path $\gamma \colon [0,1] \to Y$ whose image is contained in one of the leaves of the foliation and any transversals $T$ at $\gamma(0)$ and $T'$ at $\gamma(1)$, if we cover $\gamma$ by a sequence of foliation charts $\{(U_i, X_i, \varphi_i)\}_{i \in \{1, \ldots, n\}}$ such that $U_i \cap U_{i+1} \not= \varnothing$ for each $i \in \{0, \ldots, n-1 \}$ and choose $t_0, \ldots, t_{n} \in [0,1]$ such that $t_0 = 0$, $t_{n} = 1$, $t_0 < \ldots < t_{n}$, $\gamma(t_i) \in U_i$ for each $i \in \{1, \ldots, n \}$ and $\gamma(t_i) \in U_{i+1}$ for each $i \in \{0, \ldots, n - 1\}$, then we may define a homeomorphism $H_\gamma$ between neighborhoods of the base points of $T$ and $T'$ by composing the homeomorphisms between neighborhoods of the base points of some chosen transversals at $x_i$ and $x_{i+1}$. Although this homeomorphism depends on the choice of the cover $\{(U_i, X_i, \varphi_i)\}_{i \in \{1, \ldots, n\}}$, the points $t_0, \ldots, t_{n}$ and the transversals thereat, the \emph{germ} of this homeomorphism, denoted by $[H_\gamma]$, is independent of all these choices, in the sense that for two different choices of the above items, the resulting homeomorphisms will coincide on neighborhoods of the base points. We define the \emph{holonomy class} $[\gamma]_{\hol}$ of $\gamma$ to be the set of all paths $\gamma' \colon [0,1] \to Y$ such that 
	\begin{enumerate}
		\item the image of $\gamma'$ is contained in one of the leaves of $\mathcal{F}$;
		\item $\gamma(0) = \gamma'(0)$ and $\gamma(1) = \gamma'(1)$; 
		\item $[H_\gamma] = [H_{\gamma'}]$. 
	\end{enumerate}
	It can be checked that the equivalence relation of being in the same holonomy class is preserved under concatenation and reversal of paths, and is weaker than homotopy. We define the \emph{holonomy groupoid} $G_{\mathcal{F}}$ to be the set of all holonomy classes of paths along the leaves of $\mathcal{F}$, where the unit space consists of (classes of) the constant paths and is identified with $Y$, taking inverse amounts to reversing (class of) a path, and the range and source of $[\gamma]_{\hol}$ are $\gamma(0)$ and $\gamma(1)$, respectively. This groupoid is topologized as follows: for any $[\gamma] \in G_{\mathcal{F}}$ where $\gamma \colon [0,1] \to Y$ is a path whose image is contained in one of the leaves of the foliation, we choose a sequence of foliation charts $\{(U_i, X_i, \varphi_i)\}_{i \in \{1, \ldots, n\}}$ as above, and define a set $V(\gamma, \{(U_i, X_i, \varphi_i)\}_{i \in \{1, \ldots, n\}})$ consisting of the holonomy classes of all paths $\gamma' \colon [0,1] \to Y$ whose image is contained in one of the leaves of the foliation and also in the union $\displaystyle \bigcup_{i=1}^{n} U_i$. We let these sets generate the topology on $G_{\mathcal{F}}$ as a subbase (in fact, it suffices to pick one representative for each holonomy class to carry out the construction). This topology is locally Hausdorff and locally compact.  When restricted to the unit space $G_{\mathcal{F}}^0$, it coincides with the topology on $Y$, and when restricted to $G_{\mathcal{F}}^u$ for any $u \in G_{\mathcal{F}}^0$, it makes $G_{\mathcal{F}}^u$ into a topological one-dimensional manifold. 
	
	In the prominent case where $Y$ is a smooth manifold and $\mathcal{F}$ is smooth (i.e., given by an atlas whose the transition functions are smooth), the holonomy groupoid $G_{\mathcal{F}}$ becomes a Lie groupoid. Its Haar systems thus are in one-to-one correspondence with $1$-densities on the dual of its Lie algebroid. 
\end{cnstrct}

\begin{prop}\label{prop:foliation-covering}
	Let $Y$ be a locally compact second countable Hausdorff space and let 
	$\calpha$ be a 
	topological flow on $Y$ without fixed points. Let $\mathcal{F}$ be the line 
	foliation induced from $\calpha$. Then there is a fiberwise groupoid 
	covering map $\pi$ from the transformation groupoid $Y \rtimes \R$ to 
	$G_{\mathcal{F}}$. 
\end{prop}	

\begin{proof}
	Given any $y \in Y$ and any $t \in \R$, we define a curve $\gamma_{y,t}:[0,1] \to Y$ by $\gamma_{y,t}(s) = \calpha_{-ts}(y)$.  Noting that the leaves of the foliation $\mathcal{F}$ are all either copies of $\R$ or of the circle $S^1$, it is evident that if $y$ and $y'$ are two points on the same leaf, then any curve from $y$ to $y'$ along the leaf is homotopic inside the leaf via a homotopy which fixes end-points to a curve of the form $\gamma_{y,t}$ for some $t$ such that $\calpha_{-t}(y) = y'$. Thus, the map $\pi$ from $Y \rtimes \R$ to $G_{\mathcal{F}}$ given by 
	\[
		\pi(y,t) = [\gamma_{y,t}]_{\hol}
	\]
	is a surjective map. It is straightforward to see that $\pi$ is also a groupoid homomorphism, and when restricted to the unit spaces, $\pi$ coincides with the identity map on $Y$.

	It remains to show that $\pi$ is a local homeomorphism. Fix $(y,t) \in Y 
	\rtimes \R$. Then there exist tubes $B_1, \cdots,B_n \subseteq Y$ such that 
	$y \in S_{B_1^o}$, $\calpha_{-t}(y) \in S_{B_n^o}$ and for any $z \in 
	S_{B_1^0}$, we have $\calpha_{-st}(z) \in \bigcup_{k=1}^n B_k^o$ for all $s 
	\in [0,1]$, and $\calpha_{-t}(z) \in B_n^o$.  While there is no reason to 
	expect that $\calpha_t(z)$ is in the central slice of $B_n$, there is a 
	unique $s(z) \in \left( - \frac{l_{B_n}}{2} , \frac{l_{B_n}}{2} \right)$ 
	for which $\calpha_{-t - s(z)}(z) \in S_{B_n^o}$, and we note that the map 
	$z \mapsto s(z)$ is continuous. For any $w \in B_1^o$, pick $r(w) \in 
	\left( - \frac{l_{B_1}}{2} , \frac{l_{B_1}}{2} \right)$ such that 
	$\calpha_{-r(w)}(w) \in S_{B_1^o}$, and note that $w \mapsto r(w)$ is 
	continuous. Set $T(w) = t+r(w)+s(\calpha_{-r(w)}(w))$, so that 
	$\calpha_{-T(w)}(w) \in S_{B_n^o}$. Since $T$ is continuous and $T(y) = t$, 
	the set 
	\[
		W(y,t, \{B_1, \cdots,B_n\}) = \left\{(w, T(w) + q) \in Y \rtimes \R \middlebar w \in B_1^o , ~ q \in \left( - \frac{l_{B_n}}{2} , \frac{l_{B_n}}{2} \right) \right\}
	\]
	is an open neighborhood of $(y,t)$ in $Y \rtimes \R$. Since tubes can be made arbitrarily small, the collection of all such $W(y,t, \{B_1, \cdots,B_n\})$ forms a local base of the topology at $(y,t)$. On the other hand, $\pi$ is injective on each $W(y,t, \{B_1, \cdots,B_n\})$ because if the paths $\gamma_{w, T(w) + q}$ and $\gamma_{w', T(w') + q'}$ are holonomy equivalent for $w, w' \in B_1^o$ and $ q , q' \in \left( - \frac{l_{B_n}}{2} , \frac{l_{B_n}}{2} \right)$, then the endpoint condition yields $w = w'$ and $\calpha_{-(T(w) + q)}(w) = \calpha_{-(T(w') + q')}(w')$, which implies $\calpha_{-q}\left(\calpha_{-T(w)}(w)\right) = \calpha_{-q'}\left(\calpha_{-T(w)}(w)\right)$ and thus $q = q'$. Moreover, it follows from the definition of the subsets $V(\gamma_{y,t}, \{(U_i, X_i, \varphi_i)\}_{i \in \{1, \ldots, m\}})$ in $G_{\mathcal{F}}$ earlier in this section, we can see that if $\bigcup_{k=1}^n B_k^o \subset \bigcup_{l=1}^m U_l$ then we have $\pi\left( W(y,t, \{B_1, \cdots,B_n\}) \right) \subset V(\gamma_{y,t}, \{(U_i, X_i, \varphi_i)\}_{i \in \{1, \ldots, m\}})$ and vice versa. Hence the collection of all possible $\pi\left( W(y,t, \{B_1, \cdots,B_n\}) \right)$ forms a local basis of $G_{\mathcal{F}}$ at $\pi(y,t)$. This shows that $\pi$ is a local homeomorphism. 
\end{proof}

\begin{cor}\label{cor:foliation-covering}
	Let $Y$, $\calpha$, $\mathcal{F}$ and $\pi$ be as in 
	Proposition~\ref{prop:foliation-covering}. Let $\mu$ be the Haar system on 
	$Y \rtimes \R$ determined by the Lebesgue measure on $\R$ as in 
	Example~\ref{example:groupoids}. Then there is a Haar system $\lambda$ on 
	$G_{\mathcal{F}}$ such that $\pi$ locally intertwines $\mu$ and $\lambda$. 
	In this case, $C^*(G_{\mathcal{F}}, \lambda)$ is a quotient of $C_0(Y) 
	\rtimes \R$.
\end{cor}

\begin{proof}
	This follows from Theorem~\ref{thm:fiberwise-groupoid-covering-locally-compact} and~\ref{thm:fiberwise-groupoid-covering-quotient}. 
\end{proof}

Intuitively speaking, the Haar system $\lambda$ is obtained by identifying each 
leaf, i.e., each orbit, with the quotient of $\R$ by the stabilizer group of 
the orbit, and then taking Lebesgue measure on this quotient, which is either 
$\R$ or a circle.

We return to the discussion of nuclear dimension. By \cite{MRW87}, two 
different Haar systems on a locally compact locally Hausdorff groupoid will 
yield groupoid $C^*$-algebras that are Morita equivalent. Since the nuclear 
dimension is invariant under Morita equivalence 
(\cite[Corollary~2.8]{winter-zacharias}), there is no ambiguity in talking 
about the nuclear dimension of the $C^*$-algebra associated to a locally 
compact locally Hausdorff groupoid. 

\begin{Cor}\label{cor:foliation-algebra-dimnuc}
	If $Y$ is a locally compact second countable Hausdorff space with covering 
	dimension 
	$d$, and $\mathcal{F}$ is an orientable one-dimensional foliation of $Y$, 
	then the nuclear dimension of the $C^*$-algebra associated to the holonomy 
	groupoid is at most $5d^2+12d+6$.
\end{Cor}

\begin{proof}
	By Theorem~\ref{thm:Whitney}, $\mathcal{F}$ arises from a flow $\calpha$ on $Y$. Thus for the holonomy groupoid $G_{\mathcal{F}}$, we may use, without loss of generality, the Haar system $\lambda_{G_{\mathcal{F}}}$ produced in Proposition~\ref{prop:foliation-covering}. Corollary~\ref{cor:foliation-covering} guarantees $C^*(G_{\mathcal{F}}, \lambda_{G_{\mathcal{F}}})$ is a quotient of $C_0(Y) \rtimes \R$. By \cite[Proposition 2.7]{winter-zacharias}, we have $\dimnuc(C_0(Y)\rtimes \R ) \geq \dimnuc(C^*(G_{\mathcal{F}}, \lambda_{G_{\mathcal{F}}}))$. The result then follows Theorem~\ref{thm-main}. 
\end{proof}


\end{document}